\documentclass{amsart}

\usepackage{amssymb}
\usepackage{amsmath}
\usepackage{amsthm}
\usepackage{amsfonts}

\usepackage{a4wide}

\newtheorem{theorem}{Theorem}[section]
\newtheorem{proposition}{Proposition}[section]
\newtheorem{corollary}{Corollary}[section]
\newtheorem{lemma}{Lemma}[section]

\theoremstyle{definition}
\newtheorem{definition}{Definition}[section]
\newtheorem{remark}{Remark}[section]
\newtheorem{example}{Example}[section]

\numberwithin{equation}{section}

\begin{document}


\title[Artin transfer patterns on descendant trees]{Artin transfer patterns \\ on descendant trees of finite \(p\)-groups}

\author{Daniel C. Mayer}
\address{Naglergasse 53\\8010 Graz\\Austria}
\email{algebraic.number.theory@algebra.at}
\urladdr{http://www.algebra.at}

\thanks{Research supported by the Austrian Science Fund (FWF): P 26008-N25}


\date{November 24, 2015}

\dedicatory{Respectfully dedicated to Professor M. F. Newman}

\begin{abstract}
Based on a thorough theory of the Artin transfer homomorphism \(T_{G,H}:\,G\to H/H^\prime\)
from a group \(G\) to the abelianization \(H/H^\prime\) of a subgroup \(H\le G\) of finite index \(n=(G:H)\),
and its connection with the permutation representation \(G\to S_n\)
and the monomial representation \(G\to H\wr S_n\) of \(G\),
the Artin pattern \(G\mapsto(\tau(G),\varkappa(G))\), which consists of families
\(\tau(G)=(H/H^\prime)_{H\le G}\), resp. \(\varkappa(G)=(\ker(T_{G,H}))_{H\le G}\),
of transfer targets, resp. transfer kernels,
is defined for the vertices \(G\in\mathcal{T}\)
of any descendant tree \(\mathcal{T}\) of finite \(p\)-groups.
It is endowed with partial order relations
\(\tau(\pi(G))\le\tau(G)\) and \(\varkappa(\pi(G))\ge\varkappa(G)\),
which are compatible with the parent-descendant relation \(\pi(G)<G\)
of the edges \(G\to\pi(G)\) of the tree \(\mathcal{T}\).
The partial order enables termination criteria for the \(p\)-group generation algorithm
which can be used for searching and identifying a finite \(p\)-group \(G\),
whose Artin pattern \((\tau(G),\varkappa(G))\) is known completely or at least partially,
by constructing the descendant tree with the abelianization \(G/G^\prime\) of \(G\) as its root.
An appendix summarizes details concerning induced homomorphisms between quotient groups,
which play a crucial role
in establishing the natural partial order on Artin patterns \((\tau(G),\varkappa(G))\)
and explaining the stabilization, resp. polarization, of their components
in descendant trees \(\mathcal{T}\) of finite \(p\)-groups.
\end{abstract}

\maketitle



\section{Introduction}
\label{s:Intro}
In the mathematical field of group theory,
an \textit{Artin transfer} is a certain homomorphism from an arbitrary finite or infinite group
to the commutator quotient group of a subgroup of finite index.

Originally, such transfer mappings arose as group theoretic counterparts of
class extension homomorphisms of abelian extensions of algebraic number fields
by applying Artin's reciprocity isomorphism
\cite[\S 4, Allgemeines Reziprozit\"atsgesetz, p.361]{Ar1}
to ideal class groups and
analyzing the resulting homomorphisms between quotients of Galois groups
\cite[\S 2, p.50]{Ar2}.

However, independently of number theoretic applications,
a \textit{natural partial order} on the \textit{kernels and targets} of Artin transfers,
has recently been found
to be compatible with parent-child relations between finite \(p\)-groups,
where \(p\) denotes a prime number.
Such ancestor-descendant relations can be visualized conveniently in \textit{descendant trees}
\cite[\S 4, pp.163--164]{Ma5}.

Consequently, Artin transfers provide valuable information
for classifying finite \(p\)-groups by \textit{kernel-target patterns}
and for searching and identifying particular groups in descendant trees
by looking for patterns defined by kernels and targets of Artin transfers.
These \textit{strategies of pattern recognition}
are useful not only in purely group theoretic context
but also, most importantly, for applications in algebraic number theory
concerning Galois groups of higher \(p\)-class fields and Hilbert \(p\)-class field towers.
The reason is that the unramified extensions of a base field
contain information in the shape of capitulation patterns and class group structures,
and these arithmetic invariants can be translated into
group theoretic data on transfer kernels and targets
by means of Artin's reciprocity law of class field theory.
The natural partial order on Artin patterns
admits \textit{termination criteria} for a search through a descendant tree
with the aid of recursive executions of the \(p\)-group generation algorithm
by Newman
\cite{Nm}
and O'Brien
\cite{Ob}.

The organization of this article is as follows.
The detailed theory of Artin transfers will be developed in \S\S\
\ref{s:TransversalsPermutations}
and
\ref{s:ArtinTransfer},
followed by computational implementations in \S\
\ref{s:CompImpl}.
It is our intention to present more than the least common multiple of
the original papers by Schur
\cite{Su}
and Artin
\cite{Ar2}
and the relevant sections of the text books by Hall
\cite{Hl},
Huppert
\cite{Hp},
Gorenstein
\cite{Gs},
Aschbacher
\cite{Ab},
Doerk and Hawkes
\cite{DkHw},
Smith and Tabachnikova
\cite{SmTb},
and Isaacs
\cite{Is}.

However, we shall not touch upon fusion and focal subgroups,
which form the primary goal of the mentioned authors, except Artin.
Our focus will rather be on a sound foundation
of \textit{Artin patterns}, consisting of families of transfer kernels and targets,
and their \textit{stabilization}, resp. \textit{polarization},
in descendant trees of finite \(p\)-groups.
These phenomena arise from a 
natural partial order on Artin patterns
which is compatible with ancestor-descendant relations in trees,
and is established in its most general form in \S\
\ref{s:TransferTargetKernel},
followed by impressive applications in \S\
\ref{s:StbAndPol}.

Since our endeavour is to give the most general view of each partial result,
we came to the conviction that categories, functors and natural transformations
are the adequate tool for expressing the appropriate range of validity for
the facts connected with the partial order relation on Artin patterns.
Inspired by Bourbaki's method of exposition
\cite{Bb},
an appendix in \S\
\ref{s:HomSngSbg}
on induced homomorphisms,
which is separated to avoid a disruption of the flow of exposition,
goes down to the origins exploiting set theoretic facts
concerning direct images and inverse pre-images of mappings
which are crucial for explaining the natural partial order of Artin patterns.



\section{Transversals and their permutations}
\label{s:TransversalsPermutations}

\subsection{Transversals of a subgroup}
\label{ss:Transversals}

Let \(G\) be a group and \(H\le G\) be a subgroup of finite index \(n=(G:H)\ge 1\).


\begin{definition}
\label{dfn:Transversals}
(See also
\cite[p.1013]{Su},
\cite[(1.5.1), p.11]{Hl},
\cite[Satz 2.5, p.5]{Hp}.)

\begin{enumerate}

\item
A \textit{left transversal} of \(H\) in \(G\) is an ordered system \((\ell_1,\ldots,\ell_n)\)
of representatives for the left cosets of \(H\) in \(G\)
such that \(G=\dot{\bigcup}_{i=1}^n\,\ell_iH\) is a disjoint union.

\item
Similarly, a \textit{right transversal} of \(H\) in \(G\) is an ordered system \((r_1,\ldots,r_n)\)
of representatives for the right cosets of \(H\) in \(G\)
such that \(G=\dot{\bigcup}_{i=1}^n\,Hr_i\) is a disjoint union.

\end{enumerate}

\end{definition}


\begin{remark}
\label{rmk:Transversals}

For any transversal of \(H\) in \(G\),
there exists a unique subscript \(1\le i_0\le n\) such that \(\ell_{i_0}\in H\), resp. \(r_{i_0}\in H\).
The element \(\ell_{i_0}\), resp. \(r_{i_0}\),
which represents the principal coset (i.e., the subgroup \(H\) itself)
may be replaced by the neutral element \(1\).

\end{remark}


\begin{lemma}
\label{lem:Transversals}
(See also
\cite[p.1015]{Su},
\cite[(1.5.2), p.11]{Hl},
\cite[Satz 2.6, p.6]{Hp}.)

\begin{enumerate}

\item
If \(G\) is non-abelian and \(H\) is not a normal subgroup of \(G\),
then we can only say that the inverse elements \((\ell_1^{-1},\ldots,\ell_n^{-1})\)
of a left transversal \((\ell_1,\ldots,\ell_n)\)
form a right transversal of \(H\) in \(G\).

\item
However, if \(H\unlhd G\) is a normal subgroup of \(G\),
then any left transversal is also a right transversal of \(H\) in \(G\).

\end{enumerate}

\end{lemma}


\begin{proof}

\begin{enumerate}

\item
Since the mapping \(G\to G\), \(x\mapsto x^{-1}\) is an involution, that is a bijection which is its own inverse, 
we see that\\
\(G=\dot{\bigcup}_{i=1}^n\,\ell_iH\) \quad implies \quad
\(G=G^{-1}=\dot{\bigcup}_{i=1}^n\,(\ell_iH)^{-1}
=\dot{\bigcup}_{i=1}^n\,H^{-1}\ell_i^{-1}=\dot{\bigcup}_{i=1}^n\,H\ell_i^{-1}\).

\item
For a normal subgroup \(H\unlhd G\), we have \(xH=Hx\) for each \(x\in G\).

\end{enumerate}

\end{proof}


\noindent
Let \(\phi:\,G\to T\) be a group homomorphism and
\((\ell_1,\ldots,\ell_n)\) be a left transversal of a subgroup \(H\) in \(G\)
with finite index \(n=(G:H)\ge 1\).
We must check whether the image of this transversal under the homomorphism is again a transversal.

\begin{proposition}
\label{prp:TransvUnderHom}

The following two conditions are equivalent.

\begin{enumerate}

\item
\((\phi(\ell_1),\ldots,\phi(\ell_n))\) is a left transversal
of the subgroup \(\phi(H)\) in the image \(\mathrm{im}(\phi)=\phi(G)\)
with finite index \((\phi(G):\phi(H))=n\).

\item
\(\ker(\phi)\le H\).

\end{enumerate}

\end{proposition}


\noindent
We emphasize this important equivalence in a formula:

\begin{equation}
\label{eqn:TransvUnderHom}
\phi(G)=\dot{\bigcup}_{i=1}^n\,\phi(\ell_i)\phi(H)\text{ and } (\phi(G):\phi(H))=n
\quad\Longleftrightarrow\quad
\ker(\phi)\le H
\end{equation}


\begin{proof}
By assumption,
we have the disjoint left coset decomposition \(G=\dot{\bigcup}_{i=1}^n\,\ell_iH\)
which comprises two statements simultaneously.\\
Firstly,
the group \(G=\bigcup_{i=1}^n\,\ell_iH\) is a union of cosets,\\
and secondly,
any two distinct cosets have an empty intersection \(\ell_iH\bigcap \ell_jH=\emptyset\), for \(i\ne j\).

Due to the properties of the set mapping associated with \(\phi\),
the homomorphism \(\phi\) maps the union to another union
\[\phi(G)=\phi(\bigcup_{i=1}^n\,\ell_iH)=\bigcup_{i=1}^n\,\phi(\ell_iH)=\bigcup_{i=1}^n\,\phi(\ell_i)\phi(H),\]
but weakens the equality for the intersection to a trivial inclusion
\[\emptyset=\phi(\emptyset)=\phi(\ell_iH\bigcap \ell_jH)\subseteq\phi(\ell_iH)\bigcap\phi(\ell_jH)
=\phi(\ell_i)\phi(H)\bigcap\phi(\ell_j)\phi(H), \text{ for } i\ne j.\]
To show that the images of the cosets remain disjoint
we need the property \(\ker(\phi)\le H\) of the homomorphism \(\phi\).

Suppose that \(\phi(\ell_i)\phi(H)\bigcap\phi(\ell_j)\phi(H)\ne\emptyset\) for some \(1\le i\le j\le n\),\\
then we have \(\phi(\ell_i)\phi(h_i)=\phi(\ell_j)\phi(h_j)\) for certain elements \(h_i,h_j\in H\).\\
Multiplying by \(\phi(\ell_j)^{-1}\) from the left and by \(\phi(h_j)^{-1}\) from the right, we obtain
\[\phi(\ell_j^{-1}\ell_ih_ih_j^{-1})=\phi(\ell_j)^{-1}\phi(\ell_i)\phi(h_i)\phi(h_j)^{-1}=1, \text{ that is, }
\ell_j^{-1}\ell_ih_ih_j^{-1}\in\ker(\phi)\le H.\]
Since \(h_ih_j^{-1}\in H\), this implies \(\ell_j^{-1}\ell_i\in H\), resp. \(\ell_iH=\ell_jH\), and thus \(i=j\).
(This part of the proof is also covered by
\cite[Thm.X.21, p.340]{Is}
and, in the context of normal subgroups instead of homomorphisms, by
\cite[Thm.2.3.4, p.29]{Hl}
and
\cite[Satz 3.10, p.16]{Hp}.)

Conversely, we use contraposition.\\
If the kernel \(\ker(\phi)\) of \(\phi\) is not contained in the subgroup \(H\),
then there exists an element \(x\in G\setminus H\) such that \(\phi(x)=1\).\\
But then the homomorphism \(\phi\) maps the disjoint cosets \(xH\bigcap 1\cdot H=\emptyset\)\\
to equal cosets \(\phi(x)\phi(H)\bigcap\phi(1)\phi(H)=1\cdot\phi(H)\bigcap 1\cdot\phi(H)=\phi(H)\).

\end{proof}



\normalsize

\subsection{Permutation representation}
\label{ss:PermRepr}

Suppose \((\ell_1,\ldots,\ell_n)\) is a left transversal of a subgroup \(H\le G\) of finite index \(n=(G:H)\ge 1\) in a group \(G\).
A fixed element \(x\in G\) gives rise to a unique permutation \(\lambda_x\in S_n\) of the left cosets of \(H\) in \(G\)
by left multiplication such that

\begin{equation}
\label{eqn:LeftPerm}
\begin{array}{ccccc}
x\ell_iH=\ell_{\lambda_x(i)}H, & \text{ resp. } & x\ell_i\in\ell_{\lambda_x(i)}H, & \text{ resp. } & u_x(i):=\ell_{\lambda_x(i)}^{-1}x\ell_i\in H,
\end{array}
\end{equation}

\noindent
for each \(1\le i\le n\).

Similarly, if \((r_1,\ldots,r_n)\) is a right transversal of \(H\) in \(G\), then
a fixed element \(x\in G\) gives rise to a unique permutation \(\rho_x\in S_n\) of the right cosets of \(H\) in \(G\)
by right multiplication such that

\begin{equation}
\label{eqn:RightPerm}
\begin{array}{ccccc}
Hr_ix=Hr_{\rho_x(i)}, & \text{ resp. } & r_ix\in Hr_{\rho_x(i)}, & \text{ resp. } & w_x(i):=r_ixr_{\rho_x(i)}^{-1}\in H,
\end{array}
\end{equation}

\noindent
for each \(1\le i\le n\).

The elements \(u_x(i)\), resp. \(w_x(i)\), \(1\le i\le n\), of the subgroup \(H\)
are called the \textit{monomials} associated with \(x\)
with respect to \((\ell_1,\ldots,\ell_n)\), resp. \((r_1,\ldots,r_n)\).


\begin{definition}
\label{dfn:PermRepr}
(See also \cite[Hauptsatz 6.2, p.28]{Hp}.)

The mapping \(G\to S_n,\ x\mapsto\lambda_x\), resp. \(G\to S_n,\ x\mapsto\rho_x\),
is called the \textit{permutation representation} of \(G\) in \(S_n\)
with respect to \((\ell_1,\ldots,\ell_n)\), resp. \((r_1,\ldots,r_n)\).

\end{definition}


\begin{lemma}
\label{lem:PermRepr}

For the special right transversal \((\ell_1^{-1},\ldots,\ell_n^{-1})\)
associated to the left transversal \((\ell_1,\ldots,\ell_n)\),
we have the following relations between
the monomials and permutations corresponding to an element \(x\in G\):

\begin{equation}
\label{eqn:PermRepr}
\begin{array}{ccc}
w_{x^{-1}}(i)=u_x(i)^{-1} \text{ for } 1\le i\le n & \text{ and } & \rho_{x^{-1}}=\lambda_x.
\end{array}
\end{equation}

\end{lemma}


\begin{proof}
For the right transversal \((\ell_1^{-1},\ldots,\ell_n^{-1})\), we have
\(w_x(i)=\ell_i^{-1}x\ell_{\rho_x(i)}\), for each \(1\le i\le n\).\\
On the other hand,
for the left transversal \((\ell_1,\ldots,\ell_n)\), we have

\begin{center}
\(u_x(i)^{-1}=(\ell_{\lambda_x(i)}^{-1}x\ell_i)^{-1}=\ell_i^{-1}x^{-1}\ell_{\lambda_x(i)}
=\ell_i^{-1}x^{-1}\ell_{\rho_{x^{-1}}(i)}=w_{x^{-1}}(i)\),
for each \(1\le i\le n\).
\end{center}

\noindent
This relation simultaneously shows that, for any \(x\in G\),
the permutation representations and the associated monomials are connected by

\begin{equation*}
\begin{array}{ccc}
\rho_{x^{-1}}=\lambda_x & \text{ and } & w_{x^{-1}}(i)=u_x(i)^{-1},
\end{array}
\end{equation*}

\noindent
for each \(1\le i\le n\).
\end{proof}



\section{Artin transfer}
\label{s:ArtinTransfer}

Let \(G\) be a group and \(H\le G\) be a subgroup of finite index \(n=(G:H)\ge 1\).
Assume that \((\ell_1,\ldots,\ell_n)\), resp. \((r_1,\ldots,r_n)\),
is a left, resp. right, transversal of \(H\) in \(G\)
with associated permutation representation \(G\to S_n\),
\(x\mapsto\lambda_x\), resp. \(\rho_x\), such that
\(u_x(i):=\ell_{\lambda_x(i)}^{-1}x\ell_i\in H\), resp. \(w_x(i):=r_ixr_{\rho_x(i)}^{-1}\in H\),
for \(1\le i\le n\).


\begin{definition}
\label{dfn:ArtinTransfer}
(See also
\cite[p.1014]{Su},
\cite[\S 2, p.50]{Ar2},
\cite[(14.2.2--4), p.202]{Hl},
\cite[p.413]{Hp},
\cite[p.248]{Gs},
\cite[p.197]{Ab},
\cite[Dfn.(17.1), p.60]{DkHw},
\cite[p.154]{SmTb}
\cite[p.149]{Is},
\cite[p.2]{Ol}.)

The \textit{Artin transfer} \(T_{G,H}:\ G\to H/H^\prime\) from \(G\) to the abelianization \(H/H^\prime\) of \(H\)
with respect to \((\ell_1,\ldots,\ell_n)\), resp. \((r_1,\ldots,r_n)\), is defined by

\begin{equation}
\label{eqn:LeftTransfer}
\begin{array}{ccc}
T_{G,H}^{(\ell)}(x):=\prod_{i=1}^n\,\ell_{\lambda_x(i)}^{-1}x\ell_i\cdot H^\prime & \text{ or briefly } & T_{G,H}(x)=\prod_{i=1}^n\,u_x(i)\cdot H^\prime,
\end{array}
\end{equation}

\noindent
resp.

\begin{equation}
\label{eqn:RightTransfer}
\begin{array}{ccc}
T_{G,H}^{(r)}(x):=\prod_{i=1}^n\,r_ixr_{\rho_x(i)}^{-1}\cdot H^\prime & \text{ or briefly } & T_{G,H}(x)=\prod_{i=1}^n\,w_x(i)\cdot H^\prime,
\end{array}
\end{equation}

\noindent
for \(x\in G\).

\end{definition}



\begin{remark}
\label{rmk:Transfer}
I.M. Isaacs
\cite[p.149]{Is}
calls the mapping \(P:\,G\to H\),
\(x\mapsto\prod_{i=1}^n\,u_x(i)\), resp. \(x\mapsto\prod_{i=1}^n\,w_x(i)\),
the \textit{pre-transfer} from \(G\) to \(H\).
The pre-transfer can be composed with
a homomorphism \(\phi:\,H\to A\) from \(H\) into an abelian group \(A\)
to define a more general version of the \textit{transfer} \((\phi\circ P):\,G\to A\),
\(x\mapsto\prod_{i=1}^n\,\phi(u_x(i))\), resp. \(x\mapsto\prod_{i=1}^n\,\phi(w_x(i))\),
from \(G\) to \(A\) via \(\phi\), which occurs in the book by D. Gorenstein
\cite[p.248]{Gs}.
Taking the natural epimorphism \(\phi:\,H\to H/H^\prime\), \(v\mapsto vH^\prime\),
yields the Definition
\ref{dfn:ArtinTransfer}
of the \textit{Artin transfer} \(T_{G,H}\) in its original form by I. Schur
\cite[p.1014]{Su}
and by E. Artin
\cite[\S 2, p.50]{Ar2},
which has also been dubbed \textit{Verlagerung} by H. Hasse
\cite[\S 27.4, pp.170--171]{Ha}.
Note that, in general, the pre-transfer is
neither independent of the transversal nor a group homomorphism.
\end{remark}



\subsection{Independence of the transversal}
\label{ss:Independence}

Assume that \((g_1,\ldots,g_n)\) is another left transversal of \(H\) in \(G\)
such that \(G=\dot{\cup}_{i=1}^n\,g_iH\).


\begin{proposition}
\label{prp:LeftLeft}
(See also
\cite[p.1014]{Su},
\cite[Thm.14.2.1, p.202]{Hl},
\cite[Hilfssatz 1.5, p.414]{Hp},
\cite[Thm.3.2, p.246]{Gs},
\cite[(37.1), p.198]{Ab},
\cite[Thm.(17.2), p.61]{DkHw},
\cite[p.154]{SmTb},
\cite[Thm.5.1, p.149]{Is},
\cite[Prop.2, p.2]{Ol}.)

The Artin transfers with respect to \((g)\) and \((\ell)\) coincide, \(T_{G,H}^{(g)}=T_{G,H}^{(\ell)}\).

\end{proposition}


\begin{proof}
There exists a unique permutation \(\sigma\in S_n\) such that \(\ell_iH=g_{\sigma(i)}H\), for all \(1\le i\le n\).
Consequently, \(h_i:=\ell_i^{-1}g_{\sigma(i)}\in H\), resp. \(g_{\sigma(i)}=\ell_ih_i\) with \(h_i\in H\),
for all \(1\le i\le n\).
For a fixed element \(x\in G\), there exists a unique permutation \(\gamma_x\in S_n\) such that we have

\begin{center}
\(g_{\gamma_x(\sigma(i))}H=xg_{\sigma(i)}H=x\ell_ih_iH=x\ell_iH
=\ell_{\lambda_x(i)}H=\ell_{\lambda_x(i)}h_{\lambda_x(i)}H=g_{\sigma(\lambda_x(i))}H\),
\end{center}

\noindent
for all \(1\le i\le n\).
Therefore, the permutation representation of \(G\) with respect to \((g_1,\ldots,g_n)\) is given by
\(\gamma_x\circ\sigma=\sigma\circ\lambda_x\), resp. \(\gamma_x=\sigma\circ\lambda_x\circ\sigma^{-1}\in S_n\),
for \(x\in G\).
Furthermore, for the connection between the elements
\(v_x(i):=g_{\gamma_x(i)}^{-1}xg_i\in H\) and \(u_x(i):=\ell_{\lambda_x(i)}^{-1}x\ell_i\in H\), we obtain

\begin{align*}
v_x(\sigma(i)) &= g_{\gamma_x(\sigma(i))}^{-1}xg_{\sigma(i)}
=g_{\sigma(\lambda_x(i))}^{-1}x\ell_ih_i=(\ell_{\lambda_x(i)}h_{\lambda_x(i)})^{-1}x\ell_ih_i\\
               &= h_{\lambda_x(i)}^{-1}\ell_{\lambda_x(i)}^{-1}x\ell_ih_i=h_{\lambda_x(i)}^{-1}u_x(i)h_i,
\end{align*}

\noindent
for all \(1\le i\le n\).
Finally, due to the commutativity of the quotient group \(H/H^\prime\)
and the fact that \(\sigma\) and \(\lambda_x\) are permutations,
the Artin transfer turns out to be independent of the left transversal: 

\begin{align*}
T_{G,H}^{(g)}(x) &= \prod_{i=1}^n\,v_x(\sigma(i))\cdot H^\prime
=\prod_{i=1}^n\,h_{\lambda_x(i)}^{-1}u_x(i)h_i\cdot H^\prime\\
                    &= \prod_{i=1}^n\,u_x(i)\prod_{i=1}^n\,h_{\lambda_x(i)}^{-1}\prod_{i=1}^n\,h_i\cdot H^\prime\\
                    &= \prod_{i=1}^n\,u_x(i)\cdot 1\cdot H^\prime=\prod_{i=1}^n\,u_x(i)\cdot H^\prime
=T_{G,H}^{(\ell)}(x),
\end{align*}

\noindent
as prescribed in Definition
\ref{dfn:ArtinTransfer},
equation
(\ref{eqn:LeftTransfer}).
\end{proof}


It is clear that a similar proof shows that the Artin transfer is independent of
the choice between two different right transversals. 
It remains to show that the Artin transfer with respect to a right transversal
coincides with the Artin transfer with respect to a left transversal.


For this purpose,
we select the special right transversal \((\ell_1^{-1},\ldots,\ell_n^{-1})\)
associated to the left transversal \((\ell_1,\ldots,\ell_n)\),
as explained in Remark
\ref{rmk:Transversals}
and Lemma
\ref{lem:PermRepr}.

\begin{proposition}
\label{prp:RightLeft}

The Artin transfers with respect to \((\ell^{-1})\) and \((\ell)\) coincide,  \(T_{G,H}^{(\ell^{-1})}=T_{G,H}^{(\ell)}\).

\end{proposition}


\begin{proof}
Using
(\ref{eqn:PermRepr})
in Lemma
\ref{lem:PermRepr}
and the commutativity of \(H/H^\prime\), we consider the expression

\begin{align*}
T_{G,H}^{(\ell^{-1})}(x) &= \prod_{i=1}^n\,\ell_i^{-1}x\ell_{\rho_x(i)}\cdot H^\prime
=\prod_{i=1}^n\,w_x(i)\cdot H^\prime=\prod_{i=1}^n\,u_{x^{-1}}(i)^{-1}\cdot H^\prime \\
                      &= (\prod_{i=1}^n\,u_{x^{-1}}(i)\cdot H^\prime)^{-1}
=(T_{G,H}^{(\ell)}(x^{-1}))^{-1}=T_{G,H}^{(\ell)}(x).
\end{align*}

\noindent
The last step is justified by the fact that the Artin transfer is a homomorphism.
This will be shown in the following subsection
\ref{ss:Hom}.
\end{proof}



\subsection{Artin transfers as homomorphisms}
\label{ss:Hom}

Let \((\ell_1,\ldots,\ell_n)\) be a left transversal of \(H\) in \(G\).

\begin{theorem}
\label{thm:Hom}
(See also
\cite[(2), p.1014]{Su},
\cite[Thm.14.2.1, p.202]{Hl},
\cite[Hauptsatz 1.4, p.413]{Hp},
\cite[Thm.3.2, p.246]{Gs},
\cite[(37.2), p.198]{Ab},
\cite[Thm.(17.2), p.61]{DkHw},
\cite[p.155]{SmTb},
\cite[Thm.5.2, p.150]{Is},
\cite[Prop.1, p.2]{Ol}.)

The Artin transfer \(T_{G,H}:\ G\to H/H^\prime\),
\(x\mapsto\prod_{i=1}^n\,\ell_{\lambda_x(i)}^{-1}x\ell_i\cdot H^\prime\)
and the permutation representation \(G\to S_n,\ x\mapsto\lambda_x\) are group homomorphisms:

\begin{equation}
\label{eqn:Hom}
T_{G,H}(xy)=T_{G,H}(x)\cdot T_{G,H}(y) \text{ and } \lambda_{xy}=\lambda_x\circ\lambda_y
\text{ for } x,y\in G.
\end{equation}

\end{theorem}


\begin{proof}
Let \(x,y\in G\) be two elements with transfer images
\(T_{G,H}(x)=\prod_{i=1}^n\,\ell_{\lambda_x(i)}^{-1}x\ell_i\cdot H^\prime\) and
\(T_{G,H}(y)=\prod_{j=1}^n\,\ell_{\lambda_y(j)}^{-1}y\ell_j\cdot H^\prime\).
Since \(H/H^\prime\) is abelian and \(\lambda_y\) is a permutation,
we can change the order of the factors in the following product:

\begin{align*}
T_{G,H}(x)\cdot T_{G,H}(y) &= \prod_{i=1}^n\,\ell_{\lambda_x(i)}^{-1}x\ell_i\cdot H^\prime
\cdot\prod_{j=1}^n\,\ell_{\lambda_y(j)}^{-1}y\ell_j\cdot H^\prime \\
                           &= \prod_{j=1}^n\,\ell_{\lambda_x(\lambda_y(j))}^{-1}x\ell_{\lambda_y(j)}\cdot H^\prime
\cdot\prod_{j=1}^n\,\ell_{\lambda_y(j)}^{-1}y\ell_j\cdot H^\prime \\
&= \prod_{j=1}^n\,\ell_{\lambda_x(\lambda_y(j))}^{-1}x\ell_{\lambda_y(j)}\ell_{\lambda_y(j)}^{-1}y\ell_j\cdot H^\prime \\
                           &= \prod_{j=1}^n\,\ell_{(\lambda_x\circ\lambda_y)(j))}^{-1}xy\ell_j\cdot H^\prime
=T_{G,H}(xy).
\end{align*}

\noindent
This relation simultaneously shows that the Artin transfer \(T_{G,H}\)
and the permutation representation \(G\to S_n,\ x\mapsto\lambda_x\) are homomorphisms,
since \(T_{G,H}(xy)=T_{G,H}(x)\cdot T_{G,H}(y)\) and \(\lambda_{xy}=\lambda_x\circ\lambda_y\),
in a \textit{covariant} way.
\end{proof}



\subsection{Monomial representation}
\label{ss:MonRepr}

Let \((\ell_1,\ldots,\ell_n)\), resp. \((r_1,\ldots,r_n)\),
be a left, resp. right, transversal of a subgroup \(H\) in a group \(G\).
Using the monomials \(u_x(i)\), resp. \(w_x(i)\),
associated with an element \(x\in G\)
according to Equation
(\ref{eqn:LeftPerm}),
resp.
(\ref{eqn:RightPerm}),
we define.


\begin{definition}
\label{dfn:MonRepr}

The mapping
\(G\to H^n\times S_n,\ x\mapsto(u_x(1),\ldots,u_x(n);\lambda_x)\), respectively
\(G\to H^n\times S_n,\ x\mapsto(w_x(1),\ldots,w_x(n);\rho_x)\),
is called the \textit{monomial representation} of \(G\) in \(H^n\times S_n\)
with respect to \((\ell_1,\ldots,\ell_n)\), resp. \((r_1,\ldots,r_n)\).
\end{definition}


\noindent
It is illuminating to restate the homomorphism property of the Artin transfer
in terms of the monomial representation.
The images of the factors \(x,y\) are given by
\(T_{G,H}(x)=\prod_{i=1}^n\,u_x(i)\cdot H^\prime\) and
\(T_{G,H}(y)=\prod_{j=1}^n\,u_y(j)\cdot H^\prime\).
In the proof of Theorem
\ref{thm:Hom},
the image of the product \(xy\) turned out to be
\(T_{G,H}(xy)=
\prod_{j=1}^n\,\ell_{\lambda_x(\lambda_y(j))}^{-1}x\ell_{\lambda_y(j)}\ell_{\lambda_y(j)}^{-1}y\ell_j\cdot H^\prime
=\prod_{j=1}^n\,u_x(\lambda_y(j))\cdot u_y(j)\cdot H^\prime\),
which is a very peculiar \textit{law of composition} discussed in more detail in the sequel.

The law reminds of the \textit{crossed homomorphisms} \(x\mapsto u_x\)
in the \textit{first cohomology group} \(\mathrm{H}^1(G,M)\)
of a \(G\)-module \(M\), which have the property
\(u_{xy}=u_x^y\cdot u_y\), for \(x,y\in G\).

These peculiar structures can also be interpreted by endowing the cartesian product \(H^n\times S_n\)
with a special law of composition known as the
\textit{wreath product} \(H\wr S_n\) of the groups \(H\) and \(S_n\)
with respect to the set \(\lbrace 1,\ldots,n\rbrace\).


\begin{definition}
\label{dfn:WreathProduct}
For \(x,y\in G\),
the \textit{wreath product} of the associated monomials and permutations is given by

\begin{equation}
\label{eqn:MonRepr}
\begin{array}{cl}
(u_x(1),\ldots,u_x(n);\lambda_x)\cdot (u_y(1),\ldots,u_y(n);\lambda_y) &
:=(u_x(\lambda_y(1))\cdot u_y(1),\ldots,u_x(\lambda_y(n))\cdot u_y(n);\lambda_x\circ\lambda_y) \\
 & =(u_{xy}(1),\ldots,u_{xy}(n);\lambda_{xy}),
\end{array}
\end{equation}

\end{definition}


\begin{theorem}
\label{thm:MonRepr}
(See also
\cite[Thm.14.1, p.200]{Hl},
\cite[Hauptsatz 1.4, p.413]{Hp}.)

This law of composition on \(H^n\times S_n\) causes the monomial representation
\(G\to H\wr S_n,\ x\mapsto(u_x(1),\ldots,u_x(n);\lambda_x)\)
also to be a homomorphism.
In fact, it is a \(\mathrm{faithful}\) representation,
that is an injective homomorphism, also called a monomorphism or embedding,
in contrast to the permutation representation.

\end{theorem}


\begin{proof}
The homomorphism property has been shown above already.
For a homomorphism to be injective, it suffices to show the triviality of its kernel.
The neutral element of the group \(H^n\times S_n\) endowed with the wreath product
is given by \((1,\ldots,1;1)\), where the last \(1\) means the identity permutation.
If \((u_x(1),\ldots,u_x(n);\lambda_x)=(1,\ldots,1;1)\), for some \(x\in G\),
then \(\lambda_x=1\) and consequently \(1=u_x(i)=g_{\lambda_x(i)}^{-1}xg_i=g_i^{-1}xg_i\),
for all \(1\le i\le n\).
Finally, an application of the inverse inner automorphism with \(g_i\) yields \(x=1\), as required for injectivity.

The permutation representation
cannot be injective if \(G\) is infinite or at least of an order bigger than \(n!\),
the factorial of \(n\).
\end{proof}


\begin{remark}
\label{rmk:MonRepr}
Formula
(\ref{eqn:MonRepr})
is an example for
the \textit{left-sided variant} of the wreath product on \(H^n\times S_n\).
However, we point out that the wreath product
with respect to a right transversal \((r_1,\ldots,r_n)\) of \(H\) in \(G\)
appears in its \textit{right-sided variant}

\begin{equation}
\label{eqn:MonReprRight}
\begin{array}{cl}
(w_x(1),\ldots,w_x(n);\rho_x)\cdot (w_y(1),\ldots,w_y(n);\rho_y) &
:=(w_x(1)\cdot w_y(\rho_x(1)),\ldots,w_x(n)\cdot w_y(\rho_x(n));\rho_y\circ\rho_x) \\
 & =(w_{xy}(1),\ldots,w_{xy}(n);\rho_{xy}),
\end{array}
\end{equation}

\noindent
which implies that the permutation representation \(G\to S_n\), \(x\mapsto\rho_x\)
is a homomorphism with respect to the \textit{opposite} law of composition
\(\rho_{xy}=\rho_y\circ\rho_x\) on \(S_n\), in a \textit{contravariant} manner.

It can be shown that the left-sided and the right-sided variant of the
wreath product lead to isomorphic group structures on \(H^n\times S_n\).


A related viewpoint is taken by M. Hall
\cite[p.200]{Hl},
who uses the multiplication of \textit{monomial matrices} to describe the wreath product.
Such a matrix can be represented in the form
\(M_x=\mathrm{diag}(w_x(1),\ldots,w_x(n))\cdot P_{\rho_x}\)
as the product of an invertible diagonal matrix over the group ring \(K\lbrack H\rbrack\),
where \(K\) denotes a field,
and the permutation matrix \(P_{\rho_x}\)
associated with the permutation \(\rho_x\in S_n\).
Multiplying two such monomial matrices
yields a law of composition identical to the wreath product,
\(M_x\cdot M_y
=\mathrm{diag}(w_x(1),\ldots,w_x(n))\cdot P_{\rho_x}
\cdot\mathrm{diag}(w_y(1),\ldots,w_y(n))\cdot P_{\rho_y}\)\\
\(=\mathrm{diag}(w_x(1)\cdot w_y(\rho_x(1)),\ldots,w_x(n)\cdot w_y(\rho_x(n)))\cdot P_{\rho_x\circ\rho_y}\),
in the right-sided variant.


Whereas B. Huppert
\cite[p.413]{Hp}
uses the monomial representation for \textit{defining} the Artin transfer
by composition with the unsigned determinant,
we prefer to give the immediate Definition
\ref{dfn:ArtinTransfer}
and to merely \textit{illustrate} the homomorphism property of the Artin transfer
with the aid of the monomial representation.
\end{remark}



\subsection{Composition of Artin transfers}
\label{ss:Composition}

Let \(G\) be a group with nested subgroups \(K\le H\le G\) such that the indices
\((G:H)=n\), \((H:K)=m\) and \((G:K)=(G:H)\cdot (H:K)=n\cdot m\) are finite.


\begin{theorem}
\label{thm:Composition}
(See also
\cite[Thm.14.2.1, p.202]{Hl},
\cite[Satz 1.6, p.415]{Hp},
\cite[Lem.(17.3), p.61]{DkHw},
\cite[Thm.10.8, p.301]{Is},
\cite[Prop.3, p.3]{Ol}.)

Then the Artin transfer \(T_{G,K}\) is the compositum of
the \(\mathrm{induced transfer}\) \(\tilde{T}_{H,K}:\ H/H^\prime\to K/K^\prime\)
(in the sense of Corollary
\ref{cor:FctThrQtn}
or Corollary
\ref{cor:Abelianization}
in the appendix)
and the Artin transfer \(T_{G,H}\), i.e.,

\begin{equation}
\label{eqn:Composition}
T_{G,K} = \tilde{T}_{H,K}\circ T_{G,H}.
\end{equation}

\end{theorem}

\noindent
This can be seen in the following manner.


\begin{proof}
If \((\ell_1,\ldots,\ell_n)\) is a left transversal of \(H\) in \(G\)
and \((h_1,\ldots,h_m)\) is a left transversal of \(K\) in \(H\),
that is \(G=\dot{\cup}_{i=1}^n\,\ell_iH\) and \(H=\dot{\cup}_{j=1}^m\,h_jK\),
then \(G=\dot{\cup}_{i=1}^n\,\dot{\cup}_{j=1}^m\,\ell_ih_jK\) is a
disjoint left coset decomposition of \(G\) with respect to \(K\).
(See also
\cite[Thm.1.5.3, p.12]{Hl},
\cite[Satz 2.6, p.6]{Hp}.)
Given two elements \(x\in G\) and \(y\in H\),
there exist unique permutations \(\lambda_x\in S_n\), and \(\sigma_y\in S_m\),
such that the associated monomials are given by

\begin{center}
\(u_x(i):=\ell_{\lambda_x(i)}^{-1}x\ell_i\in H\), for each \(1\le i\le n\), and
\(v_y(j):=h_{\sigma_y(j)}^{-1}yh_j\in K\), for each \(1\le j\le m\).
\end{center}

\noindent
Then, using Corollary
\ref{cor:Abelianization},
we have

\begin{center}
\(T_{G,H}(x)=\prod_{i=1}^n\,u_x(i)\cdot H^\prime\), \quad
and \quad \(\tilde{T}_{H,K}(y\cdot H^\prime)=T_{H,K}(y)=\prod_{j=1}^m\,v_y(j)\cdot K^\prime\).
\end{center}

\noindent
For each pair of subscripts \(1\le i\le n\) and \(1\le j\le m\), we put \(y_i:=u_x(i)\in H\) and obtain

\begin{align*}
x\ell_ih_j &= \ell_{\lambda_x(i)}\ell_{\lambda_x(i)}^{-1}x\ell_ih_j=\ell_{\lambda_x(i)}u_x(i)h_j
=\ell_{\lambda_x(i)}y_ih_j \\
        &= \ell_{\lambda_x(i)}h_{\sigma_{y_i}(j)}h_{\sigma_{y_i}(j)}^{-1}y_ih_j
=\ell_{\lambda_x(i)}h_{\sigma_{y_i}(j)}v_{y_i}(j),
\end{align*}

\noindent
resp. \(h_{\sigma_{y_i}(j)}^{-1}\ell_{\lambda_x(i)}^{-1}x\ell_ih_j=v_{y_i}(j)\).
Thus, the image of \(x\) under the Artin transfer \(T_{G,K}\) is given by

\begin{align*}
T_{G,K}(x) &= \prod_{i=1}^n\,\prod_{j=1}^m\,v_{y_i}(j)\cdot K^\prime
=\prod_{i=1}^n\,\prod_{j=1}^m\,h_{\sigma_{y_i}(j)}^{-1}\ell_{\lambda_x(i)}^{-1}x\ell_ih_j\cdot K^\prime \\
           &= \prod_{i=1}^n\,\prod_{j=1}^m\,h_{\sigma_{y_i}(j)}^{-1}u_x(i)h_j\cdot K^\prime
=\prod_{i=1}^n\,\prod_{j=1}^m\,h_{\sigma_{y_i}(j)}^{-1}y_ih_j\cdot K^\prime \\
           &= \prod_{i=1}^n\,\tilde{T}_{H,K}(y_i\cdot H^\prime)
=\tilde{T}_{H,K}(\prod_{i=1}^n\,y_i\cdot H^\prime)=\tilde{T}_{H,K}(\prod_{i=1}^n\,u_x(i)\cdot H^\prime) \\
           &= \tilde{T}_{H,K}(T_{G,H}(x)).
\end{align*}
\end{proof}



\subsection{Wreath product of \(S_m\) and \(S_n\)}
\label{ss:WreathProd}

Motivated by the proof of Theorem
\ref{thm:Composition},
we want to emphasize the structural peculiarity of the monomial representation
\[G\to K^{n\cdot m}\times S_{n\cdot m},\
x\mapsto (k_x(1,1),\ldots,k_x(n,m);\gamma_x),\]
which corresponds to the compositum of Artin transfers,
defining
\[k_x(i,j):=((\ell h)_{\gamma_x(i,j)})^{-1}x(\ell h)_{(i,j)}\in K\]
for a permutation
\(\gamma_x\in S_{n\cdot m}\),
and using the symbolic notation \((\ell h)_{(i,j)}:=\ell_ih_j\)
for all pairs of subscripts \(1\le i\le n\), \(1\le j\le m\).


The preceding proof has shown that
\(k_x(i,j)=h_{\sigma_{y_i}(j)}^{-1}\ell_{\lambda_x(i)}^{-1}x\ell_ih_j\).
Therefore, the action of the permutation \(\gamma_x\) on the set
\(\lbrack 1,n\rbrack\times\lbrack 1,m\rbrack\) is given by
\(\gamma_x(i,j)=(\lambda_x(i),\sigma_{u_x(i)}(j))\).
The action on the second component \(j\) depends on the first component \(i\)
(via the permutation \(\sigma_{u_x(i)}\in S_m\)),
whereas the action on the first component \(i\) is independent of the second component \(j\).
Therefore, the permutation \(\gamma_x\in S_{n\cdot m}\) can be identified with the multiplet
\((\lambda_x;\sigma_{u_x(1)},\ldots,\sigma_{u_x(n)})\in S_n\times S_m^n\),
which will be written in twisted form in the sequel.

The permutations \(\gamma_x\), which arise as second components of the monomial representation
\[G\to K\wr S_{n\cdot m},\
x\mapsto (k_x(1,1),\ldots,k_x(n,m);\gamma_x),\]
are of a very special kind.
They belong to the \textit{stabilizer} of the natural equipartition of the set
\(\lbrack 1,n\rbrack\times\lbrack 1,m\rbrack\)
into the \(n\) rows of the corresponding matrix (rectangular array).
Using the peculiarities of the composition of Artin transfers in the previous section,
we show that this stabilizer is isomorphic to the \textit{wreath product}
\(S_m\wr S_n\) of the symmetric groups \(S_m\) and \(S_n\) with respect to \(\lbrace 1,\ldots,n\rbrace\),
whose underlying set \(S_m^n\times S_n\)
is endowed with the following \textit{law of composition}
in the left-sided variant.

\begin{equation}
\label{eqn:SmWreathSn}
\begin{aligned}
\gamma_{x}\cdot\gamma_{z}
&= (\sigma_{u_x(1)},\ldots,\sigma_{u_x(n)};\lambda_x)\cdot (\sigma_{u_z(1)},\ldots,\sigma_{u_z(n)};\lambda_z) \\
&= (\sigma_{u_x(\lambda_z(1))}\circ\sigma_{u_z(1)},\ldots,\sigma_{u_x(\lambda_z(n))}\circ\sigma_{u_z(n)};\lambda_x\circ \lambda_z) \\
&= (\sigma_{u_{xz}(1)},\ldots,\sigma_{u_{xz}(n)};\lambda_{xz})
=\gamma_{xz}
\end{aligned}
\end{equation}

\noindent
for all \(x,z\in G\).

This law reminds of the \textit{chain rule}
\(D(g\circ f)(x)=D(g)(f(x))\circ D(f)(x)\)
for the \textit{Fr\'echet derivative} in \(x\in E\) of the compositum of \textit{differentiable functions}
\(f:\,E\to F\) and \(g:\,F\to G\) between \textit{complete normed spaces}.


The above considerations establish a third representation, the \textit{stabilizer representation},
\[G\to S_m\wr S_n,\ x\mapsto(\sigma_{u_x(1)},\ldots,\sigma_{u_x(n)};\lambda_x)\]
of the group \(G\) in the wreath product \(S_m\wr S_n\),
similar to the permutation representation and the monomial representation.
As opposed to the latter, the stabilizer representation cannot be injective, in general.
For instance, certainly not, if \(G\) is infinite.

Formula
(\ref{eqn:SmWreathSn})
proves the following statement.

\begin{theorem}
\label{thm:StabilizerRep}
The stabilizer representation
\(G\to S_m\wr S_n,\ x\mapsto\gamma_{x}=(\sigma_{u_x(1)},\ldots,\sigma_{u_x(n)};\lambda_x)\)
of the group \(G\) in the wreath product \(S_m\wr S_n\) of symmetric groups
is a group homomorphism.
\end{theorem}



\subsection{Cycle decomposition}
\label{ss:CycDec}

Let \((\ell_1,\ldots,\ell_n)\) be a left transversal of a subgroup \(H\le G\) of finite index \(n=(G:H)\ge 1\) in a group \(G\).
Suppose the element \(x\in G\) gives rise to the permutation \(\lambda_x\in S_n\) of the left cosets of \(H\) in \(G\) such that
\(x\ell_iH=\ell_{\lambda_x(i)}H\), resp. \(\ell_{\lambda_x(i)}^{-1}x\ell_i=:u_x(i)\in H\), for each \(1\le i\le n\).


\begin{theorem}
\label{lem:CycDec}
(See also
\cite[\S 2, p.50]{Ar2},
\cite[\S 27.4, p.170]{Ha},
\cite[Hilfssatz 1.7, p.415]{Hp},
\cite[Thm.3.3, p.249]{Gs},
\cite[(37.3), p.198]{Ab},
\cite[p.154]{SmTb},
\cite[Lem.5.5, p.153]{Is},
\cite[p.5]{Ol}.)

If the permutation \(\lambda_x\) has the decomposition \(\lambda_x=\prod_{j=1}^t\,\zeta_j\) into pairwise disjoint (and thus commuting) cycles \(\zeta_j\in S_n\) of lengths \(f_j\ge 1\),
which is unique up to the ordering of the cycles, more explicitly, if

\begin{equation}
\label{eqn:CycleAction}
(\ell_jH,\ell_{\zeta_j(j)}H,\ell_{\zeta_j^2(j)}H,\ldots,\ell_{\zeta_j^{f_j-1}(j)}H)=(\ell_jH,x\ell_jH,x^2\ell_jH,\ldots,x^{f_j-1}\ell_jH),
\end{equation}

\noindent
for \(1\le j\le t\), and \(\sum_{j=1}^t\,f_j=n\),
then the image of \(x\in G\) under the Artin transfer \(T_{G,H}\) is given by

\begin{equation}
\label{eqn:TransferCycleForm}
T_{G,H}(x)=\prod_{j=1}^t\,\ell_j^{-1}x^{f_j}\ell_j\cdot H^\prime.
\end{equation}

\end{theorem}


\begin{proof}
The reason for this fact is that we obtain another left transversal of \(H\) in \(G\) by putting
\(g_{j,k}:=x^k\ell_j\) for \(0\le k\le f_j-1\) and \(1\le j\le t\), since

\begin{equation}
\label{eqn:LeftTransvOfCycles}
G=\dot{\cup}_{j=1}^t\,\dot{\cup}_{k=0}^{f_j-1}\,x^k\ell_jH
\end{equation}

\noindent
is a disjoint decomposition of \(G\) into left cosets of \(H\).

Let us fix a value of \(1\le j\le t\).
For \(0\le k\le f_j-2\), we have
\[xg_{j,k}=xx^k\ell_j=x^{k+1}\ell_j=g_{j,k+1}\in g_{j,k+1}H,
\text{ resp. } u_x(j,k):=g_{j,k+1}^{-1}xg_{j,k}=1\in H.\]
However, for \(k=f_j-1\), we obtain
\[xg_{j,f_j-1}=xx^{f_j-1}\ell_j=x^{f_j}\ell_j\in\ell_jH=g_{j,0}H,
\text{ resp. } u_x(j,f_j-1):=g_{j,0}^{-1}xg_{j,f_j-1}=\ell_j^{-1}x^{f_j}\ell_j\in H.\]
Consequently,
\[T_{G,H}(x)=\prod_{j=1}^t\,\prod_{k=0}^{f_j-1}\,u_x(j,k)\cdot H^\prime
=\prod_{j=1}^t\,(\prod_{k=0}^{f_j-2}\,1)\cdot u_x(j,f_j-1)\cdot H^\prime
=\prod_{j=1}^t\,\ell_j^{-1}x^{f_j}\ell_j\cdot H^\prime.\]
\end{proof}


The cycle decomposition corresponds to a
double coset decomposition \(G=\dot{\cup}_{j=1}^t\,\langle x\rangle \ell_jH\) of the group \(G\)
modulo the cyclic group \(\langle x\rangle\) and modulo the subgroup \(H\).
It was actually this cycle decomposition form of the transfer homomorphism
which was given by E. Artin in his original 1929 paper
\cite[\S 2, p.50]{Ar2}.



\subsection{Transfer to a normal subgroup}
\label{ss:NrmSbg}

Now let \(H\unlhd G\) be a \textit{normal} subgroup of finite index \(n=(G:H)\ge 1\) in a group \(G\).
Then we have \(xH=Hx\), for all \(x\in G\), and there exists the quotient group \(G/H\) of order \(n\).
For an element \(x\in G\), we let \(f:=\mathrm{ord}(xH)\) denote the order of the coset \(xH\) in \(G/H\),
and we let \((\ell_1,\ldots,\ell_t)\) be a left transversal of the subgroup \(\langle x,H\rangle\) in \(G\),
where \(t=n/f\).

\begin{theorem}
\label{thm:NrmSbg}
(See also \cite[\S 27.4, VII, p.171]{Ha}.)

Then the image of \(x\in G\) under the Artin transfer \(T_{G,H}\) is given by

\begin{equation}
\label{eqn:TransferToNormalSbg}
T_{G,H}(x)=\prod_{j=1}^t\,\ell_j^{-1}x^f\ell_j\cdot H^\prime.
\end{equation}

\end{theorem}

\begin{proof}
\(\langle xH\rangle\) is a cyclic subgroup of order \(f\) in \(G/H\),
and a left transversal \((\ell_1,\ldots,\ell_t)\) of the subgroup \(\langle x,H\rangle\) in \(G\),
where \(t=n/f\) and
\(G=\dot{\cup}_{j=1}^t\,\ell_j\langle x,H\rangle\) is the corresponding disjoint left coset decomposition,
can be refined to a left transversal
\(\ell_jx^k\) \((1\le j\le t,\ 0\le k\le f-1)\)
with  disjoint left coset decomposition

\begin{equation}
\label{eqn:LeftTransvOfNormalSbg}
G=\dot{\cup}_{j=1}^t\,\dot{\cup}_{k=0}^{f-1}\,\ell_jx^kH
\end{equation}

\noindent
of \(H\) in \(G\).
Hence, the formula for the image of \(x\) under the Artin transfer \(T_{G,H}\) in the previous section
takes the particular shape
\[T_{G,H}(x)=\prod_{j=1}^t\,\ell_j^{-1}x^f\ell_j\cdot H^\prime\]
with exponent \(f\) independent of \(j\).
\end{proof}


\begin{corollary}
\label{cor:NrmSbg}
(See also
\cite[Lem.10.6, p.300]{Is}
for a special case.)

In particular,
the \(\mathrm{inner\ transfer}\) of an element \(x\in H\) is given as a symbolic power

\begin{equation}
\label{eqn:InnerTransfer}
T_{G,H}(x)=x^{\mathrm{Tr}_G(H)}\cdot H^\prime
\end{equation}

\noindent
with the \(\mathrm{trace\ element}\)

\begin{equation}
\label{eqn:TraceElement}
\mathrm{Tr}_G(H)=\sum_{j=1}^t\,\ell_j\in\mathbb{Z}\lbrack G\rbrack
\end{equation}

\noindent
of \(H\) in \(G\) as symbolic exponent.\\
The other extreme is the \(\mathrm{outer\ transfer}\) of an element \(x\in G\setminus H\) which generates \(G\) modulo \(H\),
that is \(G=\langle x,H\rangle\).
It is simply an \(n\)th power

\begin{equation}
\label{eqn:OuterTransfer}
T_{G,H}(x)=x^n\cdot H^\prime.
\end{equation}

\end{corollary}


\begin{proof}
The inner transfer of an element \(x\in H\),
whose coset \(xH=H\) is the principal set in \(G/H\) of order \(f=1\),
is given as the symbolic power

\begin{center}
\(T_{G,H}(x)=\prod_{j=1}^t\,\ell_j^{-1}x\ell_j\cdot H^\prime=\prod_{j=1}^t\,x^{\ell_j}\cdot H^\prime
=x^{\sum_{j=1}^t\,\ell_j}\cdot H^\prime\)
\end{center}

\noindent
with the \(\mathrm{trace\ element}\)

\begin{center}
\(\mathrm{Tr}_G(H)=\sum_{j=1}^t\,\ell_j\in\mathbb{Z}\lbrack G\rbrack\)
\end{center}

\noindent
of \(H\) in \(G\) as symbolic exponent.\\
The outer transfer of an element \(x\in G\setminus H\) which generates \(G\) modulo \(H\),
that is \(G=\langle x,H\rangle\),
whose coset \(xH\) is generator of \(G/H\) with order \(f=n\),
is given as the \(n\)th power

\begin{center}
\(T_{G,H}(x)=\prod_{j=1}^1\,1^{-1}\cdot x^n\cdot 1\cdot H^\prime=x^n\cdot H^\prime\).
\end{center}
\end{proof}


Transfers to normal subgroups will be the most important cases in the sequel,
since the central concept of this article, the \textit{Artin pattern},
which endows descendant trees with additional structure,
consists of targets and kernels (\S\
\ref{s:TransferTargetKernel})
of Artin transfers from a group \(G\) to intermediate groups \(G^\prime\le H\le G\)
between \(G\) and its commutator subgroup \(G^\prime\).
For these intermediate groups we have the following lemma.

\begin{lemma}
\label{lem:HeadSbg}
All subgroups \(H\le G\) of a group \(G\) which contain the commutator subgroup \(G^\prime\)
are normal subgroups \(H\unlhd G\).
\end{lemma}

\begin{proof}
Let \(G^\prime\le H\le G\).
If \(H\) were not a normal subgroup of \(G\),
then we had \(x^{-1}Hx\not\subseteq H\) for some element \(x\in G\setminus H\).
This would imply the existence of elements \(h\in H\) and \(y\in G\setminus H\)
such that \(x^{-1}hx=y\), and consequently the commutator
\(\lbrack h,x\rbrack=h^{-1}x^{-1}hx=h^{-1}y\) would be an element in \(G\setminus H\)
in contradiction to \(G^\prime\le H\).
\end{proof}


Explicit implementations of Artin transfers in the simplest situations are presented in the following section.



\normalsize

\section{Computational implementation}
\label{s:CompImpl}

\subsection{Abelianization of type \((p,p)\)}
\label{ss:TypePePe}

Let \(G\) be a pro-\(p\) group with abelianization \(G/G^\prime\) of elementary abelian type \((p,p)\).
Then \(G\) has \(p+1\) maximal subgroups \(H_i<G\) \((1\le i\le p+1)\) of index \((G:H_i)=p\).
In this particular case,
the Frattini subgroup \(\Phi(G):=\bigcap_{i=1}^{p+1}\,H_i\),
which is defined as the intersection of all maximal subgroups,
coincides with the commutator subgroup \(G^\prime=\lbrack G,G\rbrack\),
since the latter contains all \(p\)th powers \(G^\prime\ge G^p\),
and thus we have \(\Phi(G)=G^p\cdot G^\prime=G^\prime\).

For each \(1\le i\le p+1\),
let \(T_i:\,G\to H_i/H_i^\prime\) be the Artin transfer homomorphism from \(G\) to the abelianization of \(H_i\).
According to Burnside's basis theorem, the group \(G\) has generator rank \(d(G)=2\)
and can therefore be generated as \(G=\langle x,y\rangle\) by two elements \(x,y\) such that \(x^p,y^p\in G^\prime\).
For each of the normal subgroups \(H_i\lhd G\), we need
a generator \(h_i\) with respect to \(G^\prime\),
and a generator \(t_i\) of a transversal \((1,t_i,t_i^2,\ldots,t_i^{p-1})\) such that
\(H_i=\langle h_i,G^\prime\rangle\) and \(G=\langle t_i,H_i\rangle=\dot{\bigcup}_{j=0}^{p-1}\,t_i^jH_i\).

A convenient selection is given by

\begin{equation}
\label{eqn:GenAndTrvPxP}
h_1=y,\ t_1=x, \text{ and } h_i=xy^{i-2},\ t_i=y, \text{ for all } 2\le i\le p+1.
\end{equation}

\noindent
Then, for each \(1\le i\le p+1\), it is possible to implement the \textit{inner transfer} by

\begin{equation}
\label{eqn:InnerATPxP}
T_i(h_i)=h_i^{\mathrm{Tr}_G(H_i)}\cdot H_i^\prime=h_i^{1+t_i+t_i^2+\ldots +t_i^{p-1}}\cdot H_i^\prime,
\end{equation}

\noindent
according to equation
(\ref{eqn:InnerTransfer})
of Corollary
\ref{cor:NrmSbg},
which can also be expressed by a product of two \(p\)th powers,

\begin{equation}
\label{eqn:InnerATVarPxP}
T_i(h_i)
=h_i\cdot t_i^{-1}h_it_i\cdot t_i^{-2}h_it_i^2\cdots t_i^{-p+1}h_it_i^{p-1}\cdot H_i^\prime
=(h_it_i^{-1})^pt_i^p\cdot H_i^\prime,
\end{equation}

\noindent
and to implement the \textit{outer transfer} as a complete \(p\)th power by

\begin{equation}
\label{eqn:OuterATPxP}
T_i(t_i)=t_i^p\cdot H_i^\prime,
\end{equation}

\noindent
according to equation
(\ref{eqn:OuterTransfer})
of Corollary
\ref{cor:NrmSbg}.
The reason is that \(\mathrm{ord}(h_iH_i)=1\) and \(\mathrm{ord}(t_iH_i)=p\) in the quotient group \(G/H_i\).

It should be pointed out that the complete specification of the Artin transfers \(T_i\)
also requires explicit knowledge of the derived subgroups \(H_i^\prime\).
Since \(G^\prime\) is a normal subgroup of index \(p\) in \(H_i\),
a certain general reduction is possible by
\(H_i^\prime=\lbrack H_i,H_i\rbrack=\lbrack G^\prime,H_i\rbrack=(G^\prime)^{h_i-1}\)
\cite[Lem.2.1, p.52]{Bl},
but an explicit pro-\(p\) presentation of \(G\) must be known
for determining generators of \(G^\prime=\langle s_1,\ldots,s_n\rangle\),
whence

\begin{equation}
\label{eqn:DrvSbgOfMaxSbg}
H_i^\prime=(G^\prime)^{h_i-1}=\langle\lbrack s_1,h_i\rbrack,\ldots,\lbrack s_n,h_i\rbrack\rangle.
\end{equation}



\begin{figure}[ht]
\caption{Layers of subgroups \(G^\prime\le H_{i,j}\le G\) for \(G/G^\prime=\langle x,y,G^\prime\rangle\simeq(p^2,p)\)}
\label{fig:LayersTypeP2xP}

\setlength{\unitlength}{1.0cm}
\begin{picture}(15,10)(-4,1)

\put(-4,10.3){\makebox(0,0)[cb]{Order \(p^n\)}}
\put(-4,8){\vector(0,1){2}}
\put(-3.8,8){\makebox(0,0)[lc]{\(p^3\)}}
\put(-3.8,6){\makebox(0,0)[lc]{\(p^2\)}}
\put(-3.8,4){\makebox(0,0)[lc]{\(p\)}}
\put(-3.8,2){\makebox(0,0)[lc]{\(1\)}}
\multiput(-4.1,2)(0,2){4}{\line(1,0){0.2}}
\put(-4,2){\line(0,1){6}}

\put(-2.5,8){\makebox(0,0)[lb]{Layer \(0\)}}
\put(-2.5,8){\makebox(0,0)[lt]{(Top)}}
\put(6.5,8){\makebox(0,0)[lb]{\(\mathrm{Lyr}_0(G)=\lbrace H_{0,1}\rbrace\)}}
\put(1.8,8){\makebox(0,0)[rb]{\(G=H_{0,1}\)}}
\put(2,8){\circle*{0.1}}
\multiput(2,8)(1.5,-2){2}{\line(-3,-4){1.5}}
\put(2,8){\line(1,-4){0.5}}
\put(2.5,6){\line(-1,-4){0.5}}
\multiput(2,8)(-1.5,-2){2}{\line(3,-4){1.5}}

\put(0.3,6){\makebox(0,0)[rb]{\(H_{1,1}\)}}
\put(0.3,5.8){\makebox(0,0)[rt]{\(x\)}}
\put(1.5,6){\makebox(0,0)[cb]{\(\ldots\)}}
\multiput(0.5,6)(2,0){2}{\circle*{0.1}}
\put(3.5,6){\circle*{0.2}}
\put(2.6,6){\makebox(0,0)[lb]{\(H_{1,p}\)}}
\put(3.7,6){\makebox(0,0)[lb]{\(H_{1,p+1}=\tilde{\Phi}\)}}
\put(-2.5,6){\makebox(0,0)[lb]{Layer \(1\)}}
\put(6.5,6){\makebox(0,0)[lb]{\(\mathrm{Lyr}_1(G)=\lbrace H_{1,1},\ldots,H_{1,p+1}\rbrace\)}}
\put(6.5,6){\makebox(0,0)[lt]{Maximal subgroups}}

\put(1.8,4){\makebox(0,0)[rb]{\(\Phi=H_{2,p+1}\)}}
\put(1.8,4){\makebox(0,0)[rt]{Frattini subgroup}}
\put(2.9,4){\makebox(0,0)[rb]{\(H_{2,p}\)}}
\put(2,4){\circle*{0.2}}
\multiput(3,4)(2,0){2}{\circle*{0.1}}
\put(4,4){\makebox(0,0)[cb]{\(\ldots\)}}
\put(5.2,4){\makebox(0,0)[lb]{\(H_{2,1}\)}}
\put(5.2,3.8){\makebox(0,0)[lt]{\(y\)}}
\put(-2.5,4){\makebox(0,0)[lb]{Layer \(2\)}}
\put(6.5,4){\makebox(0,0)[lb]{\(\mathrm{Lyr}_2(G)=\lbrace H_{2,1},\ldots,H_{2,p+1}\rbrace\)}}

\multiput(3.5,6)(1.5,-2){2}{\line(-3,-4){1.5}}
\put(3.5,6){\line(-1,-4){0.5}}
\put(3,4){\line(1,-4){0.5}}
\multiput(3.5,6)(-1.5,-2){2}{\line(3,-4){1.5}}
\put(3.5,2){\circle*{0.1}}
\put(3.7,2){\makebox(0,0)[lt]{\(H_{3,1}=G^\prime\)}}
\put(-2.5,2){\makebox(0,0)[lb]{Layer \(3\)}}
\put(-2.5,2){\makebox(0,0)[lt]{(Bottom)}}
\put(6.5,2){\makebox(0,0)[lb]{\(\mathrm{Lyr}_3(G)=\lbrace H_{3,1}\rbrace\)}}
\put(6.5,2){\makebox(0,0)[lt]{Commutator subgroup}}

\end{picture}

\end{figure}
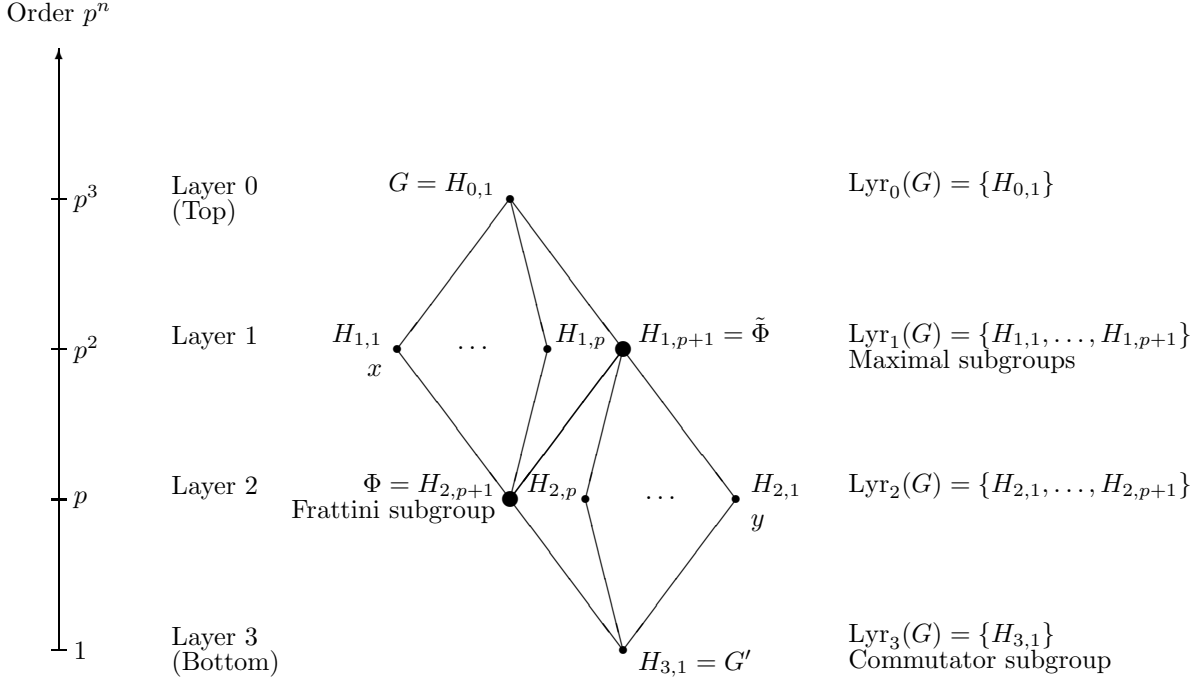



\subsection{Abelianization of type \((p^2,p)\)}
\label{ss:TypePe2Pe}

Let \(G\) be a pro-\(p\) group with abelianization \(G/G^\prime\) of non-elementary abelian type \((p^2,p)\).
Then \(G\) has \(p+1\) maximal subgroups \(H_{1,i}<G\) \((1\le i\le p+1)\) of index \((G:H_{1,i})=p\)
and \(p+1\) subgroups \(H_{2,i}<G\) \((1\le i\le p+1)\) of index \((G:H_{2,i})=p^2\).

Figure
\ref{fig:LayersTypeP2xP}
visualizes this smallest non-trivial example of a \textit{multi-layered} abelianization \(G/G^\prime\)
\cite[Dfn.3.1--3, p.288]{Ma6}.

\noindent
For each \(1\le i\le p+1\), let \(T_{1,i}:\,G\to H_{1,i}/H_{1,i}^\prime\), resp. \(T_{2,i}:\,G\to H_{2,i}/H_{2,i}^\prime\),
be the Artin transfer homomorphism from \(G\) to the abelianization of \(H_{1,i}\), resp. \(H_{2,i}\).
Burnside's basis theorem asserts that the group \(G\) has generator rank \(d(G)=2\)
and can therefore be generated as \(G=\langle x,y\rangle\) by two elements \(x,y\)
such that \(x^{p^2},y^p\in G^\prime\).

We begin by considering the \textit{first layer} of subgroups.
For each of the normal subgroups \(H_{1,i}\triangleleft G\) \((1\le i\le p)\),
we select a generator

\begin{equation}
\label{eqn:GenLyr1P2xP}
h_i=xy^{i-1} \text{ such that } H_{1,i}=\langle h_i,G^\prime\rangle.
\end{equation}

\noindent
These are the cases where the factor group \(H_{1,i}/G^\prime\) is cyclic of order \(p^2\).
However, for the \textit{distinguished maximal subgroup} \(H_{1,p+1}\),
for which the factor group \(H_{1,p+1}/G^\prime\) is bicyclic of type \((p,p)\),
we need two generators

\begin{equation}
\label{eqn:GenLyr1VarP2xP}
h_{p+1}=y,\ h_0=x^p\text{ such that } H_{1,p+1}=\langle h_{p+1},h_0,G^\prime\rangle.
\end{equation}

\noindent
Further, a generator \(t_i\) of a transversal must be given
such that \(G=\langle t_i,H_{1,i}\rangle\), for each \(1\le i\le p+1\).
It is convenient to define

\begin{equation}
\label{eqn:TrvLyr1P2xP}
t_i=y, \text{ for } 1\le i\le p, \text{ and } t_{p+1}=x.
\end{equation}

\noindent
Then, for each \(1\le i\le p+1\), we have the \textit{inner transfer}

\begin{equation}
\label{eqn:InnerATLyr1P2xP}
T_{1,i}(h_i)=h_i^{\mathrm{Tr}_G(H_{1,i})}\cdot H_{1,i}^\prime=h_i^{1+t_i+t_i^2+\ldots +t_i^{p-1}}\cdot H_{1,i}^\prime,
\end{equation}

\noindent
which equals \((h_it_i^{-1})^pt_i^p\cdot H_{1,i}^\prime\), and the \textit{outer transfer}

\begin{equation}
\label{eqn:OuterATLyr1P2xP}
T_{1,i}(t_i)=t_i^p\cdot H_{1,i}^\prime,
\end{equation}

\noindent
since \(\mathrm{ord}(h_iH_{1,i})=1\) and \(\mathrm{ord}(t_iH_{1,i})=p\).

Now we continue by considering the \textit{second layer} of subgroups.
For each of the normal subgroups \(H_{2,i}\triangleleft G\) \((1\le i\le p+1)\),
we select a generator

\begin{equation}
\label{eqn:GenLyr2P2xP}
u_1=y,\ u_i=x^py^{i-1} \text{ for } 2\le i\le p, \text{ and } u_{p+1}=x^p,
\end{equation}

\noindent
such that \(H_{2,i}=\langle u_i,G^\prime\rangle\).
Among these subgroups, the \textit{Frattini subgroup} \(H_{2,p+1}=\langle x^p,G^\prime\rangle=G^p\cdot G^\prime\)
is particularly distinguished.
A uniform way of defining generators \(t_i,w_i\) of a transversal such that \(G=\langle t_i,w_i,H_{2,i}\rangle\),
is to set

\begin{equation}
\label{eqn:TrvLyr2P2xP}
t_i=x,\ w_i=x^p,\text{ for } 1\le i\le p, \text{ and } t_{p+1}=x,\ w_{p+1}=y.
\end{equation}

\noindent
Since \(\mathrm{ord}(u_iH_{2,i})=1\), but on the other hand \(\mathrm{ord}(t_iH_{2,i})=p^2\) and \(\mathrm{ord}(w_iH_{2,i})=p\),
for \(1\le i\le p+1\),
with the single exception that \(\mathrm{ord}(t_{p+1}H_{2,p+1})=p\),
we obtain the following expressions for the \textit{inner transfer}

\begin{equation}
\label{eqn:InnerATLyr2P2xP}
T_{2,i}(u_i)=u_i^{\mathrm{Tr}_G(H_{2,i})}\cdot H_{2,i}^\prime=u_i^{\sum_{j=0}^{p-1}\,\sum_{k=0}^{p-1}\,w_i^jt_i^k}\cdot H_{2,i}^\prime
=\prod_{j=0}^{p-1}\,\prod_{k=0}^{p-1}\,(w_i^jt_i^k)^{-1}u_iw_i^jt_i^k\cdot H_{2,i}^\prime,
\end{equation}

\noindent
and for the \textit{outer transfer}

\begin{equation}
\label{eqn:OuterATLyr2P2xP}
T_{2,i}(t_i)=t_i^{p^2}\cdot H_{2,i}^\prime,
\end{equation}

\noindent
exceptionally

\begin{equation}
\label{eqn:OuterATLyr2Var1P2xP}
T_{2,p+1}(t_{p+1})=(t_{p+1}^p)^{1+w_{p+1}+w_{p+1}^2+\ldots +w_{p+1}^{p-1}}\cdot H_{2,p+1}^\prime,
\end{equation}

\noindent
and

\begin{equation}
\label{eqn:OuterATLyr2Var2P2xP}
T_{2,i}(w_i)=(w_i^p)^{1+t_i+t_i^2+\ldots +t_i^{p-1}}\cdot H_{2,i}^\prime,
\end{equation}

\noindent
for \(1\le i\le p+1\).
Again, it should be emphasized that
the structure of the derived subgroups \(H_{1,i}^\prime\) and \(H_{2,i}^\prime\)
must be known explicitly
to specify the action of the Artin transfers completely.



\section{Transfer targets and kernels}
\label{s:TransferTargetKernel}

After our thorough treatment of the general theory of Artin transfers in \S\S\
\ref{s:TransversalsPermutations}
and
\ref{s:ArtinTransfer},
and their computational implementation for some simple cases in \S\
\ref{s:CompImpl},
we are now in the position to introduce \textit{Artin transfer patterns},
which form the central concept of this article.
They provide an incredibly powerful tool for
classifying finite and infinite pro-\(p\) groups
and for identifying a finite \(p\)-group \(G\)
with sufficiently many assigned components of its Artin pattern
by the \textit{strategy of pattern recognition}.
This is done in a search through the descendant tree with root \(G/G^\prime\)
by means of recursive applications of the \(p\)-group generation algorithm
by Newman
\cite{Nm}
and O'Brien
\cite{Ob}.

An Artin transfer pattern consists of two \textit{families}
of transfer \textit{targets}, resp. \textit{kernels},
which are also called \textit{multiplets},
whereas their individual components are referred to as \textit{singulets}.



\subsection{Singulets of transfer targets}
\label{ss:TargetSing}

\begin{theorem}
\label{thm:TargetSing}

Let \(G\) and \(T\) be groups.
Suppose that
\(H=\mathrm{im}(\varphi)=\varphi(G)\le T\)
is the image of \(G\) under a homomorphism \(\varphi:\,G\to T\),
and \(V=\varphi(U)\) is the image of an arbitrary subgroup \(U\le G\).
Then the following claims hold
without any further necessary assumptions.

\begin{enumerate}

\item
The commutator subgroup of \(V\) is the image of the commutator subgroup of \(U\),
that is

\begin{equation}
\label{eqn:ImgDrvSbg}
V^\prime=\varphi(U^\prime).
\end{equation}

\item
The restriction \(\varphi\vert_U:\,U\to V\) is an epimorphism which
induces a unique epimorphism

\begin{equation}
\label{eqn:IndEpi}
\tilde{\varphi}:\,U/U^\prime\to V/V^\prime,\ xU^\prime\mapsto\varphi(x)V^\prime.
\end{equation}

\noindent
Thus, the abelianization of \(V\),

\begin{equation}
\label{eqn:ImgIndEpi}
V/V^\prime\simeq (U/U^\prime)/\ker(\tilde{\varphi}),
\end{equation}

\noindent
is an epimorphic image of the abelianization of \(U\),
namely the quotient of \(U/U^\prime\) by the kernel of \(\tilde{\varphi}\),
which is given by

\begin{equation}
\label{eqn:KerIndEpi}
\ker(\tilde{\varphi})=\Bigl(U^\prime\cdot\ker(\varphi)\cap U\Bigr)/U^\prime.
\end{equation}

\item
Moreover,
the map \(\tilde{\varphi}\) is an isomorphism,
and the quotients \(V/V^\prime\simeq U/U^\prime\) are isomorphic,
if and only if

\begin{equation}
\label{eqn:IndIso}
\ker(\varphi)\le U^\prime.
\end{equation}

\end{enumerate}

\end{theorem}

See Figure
\ref{fig:HomAndQtn}
for a visualization of this situation.



\begin{figure}[ht]
\caption{Induced homomorphism of derived quotients}
\label{fig:HomAndQtn}


\setlength{\unitlength}{1cm}
\begin{picture}(6,5)(-3,-6.5)

\put(-2,-2){\makebox(0,0)[cc]{\(V\)}}
\put(-2,-5.5){\vector(0,1){3}}
\put(-2.1,-4){\makebox(0,0)[rc]{\(\varphi\vert_U\)}}
\put(-2,-6){\makebox(0,0)[cc]{\(U\)}}

\put(0,-1.9){\makebox(0,0)[cb]{\(\omega_{V^\prime}\)}}
\put(-1.5,-2){\vector(1,0){3}}
\put(-1.5,-6){\vector(1,0){3}}
\put(0,-6.1){\makebox(0,0)[ct]{\(\omega_{U^\prime}\)}}

\put(2,-2){\makebox(0,0)[cc]{\(V/V^\prime\)}}
\put(2,-5.5){\vector(0,1){3}}
\put(2.1,-4){\makebox(0,0)[lc]{\(\tilde{\varphi}\)}}
\put(2,-6){\makebox(0,0)[cc]{\(U/U^\prime\)}}

\end{picture}

\end{figure}



\begin{proof}
The statements can be seen in the following manner.
The image of the commutator subgroup is given by
\[\varphi(U^\prime)=\varphi(\lbrack U,U\rbrack)=\varphi(\langle\lbrack u,v\rbrack\mid u,v\in U\rangle)
=\langle\lbrack \varphi(u),\varphi(v)\rbrack\mid u,v\in U\rangle=\lbrack \varphi(U),\varphi(U)\rbrack
=\varphi(U)^\prime=V^\prime.\]
The homomorphism \(\varphi\) can be restricted to an epimorphism \(\varphi\vert_U:\ U\to\varphi(U)=V\).
According to Theorem
\ref{thm:IndHomQtn},
in particular, by the Formulas
(\ref{eqn:CritIndHom})
and
(\ref{eqn:KerImCoKer})
in the appendix,
the condition \(\varphi(U^\prime)=V^\prime\) implies the existence of a uniquely determined epimorphism
\(\tilde{\varphi}:\ U/U^\prime\to V/V^\prime\) such that
\(\tilde{\varphi}\circ\omega_{U^\prime}=\omega_{V^\prime}\circ\varphi\vert_U\).
The Isomorphism Theorem in Formula
(\ref{eqn:IsomThm})
in the appendix shows that \(V/V^\prime\simeq (U/U^\prime)/\ker(\tilde{\varphi})\).
Furthermore, by the Formulas
(\ref{eqn:KerImCoKer})
and
(\ref{eqn:InvAfterDir}),
the kernel of \(\tilde{\varphi}\) is given explicitly by
\[\ker(\tilde{\varphi})=\Bigl(\varphi^{-1}(\varphi(U^\prime))\cap U\Bigr)/U^\prime
=\Bigl(U^\prime\cdot\ker(\varphi)\cap U\Bigr)/U^\prime.\]
Thus, \(\tilde{\varphi}\) is an isomorphism if and only if \(\ker(\varphi)\unlhd U^\prime(\unlhd U)\).
\end{proof}



\noindent
\textbf{Functor of derived quotients.}
In analogy to section \S\
\ref{ss:FunctorialProps}
in the appendix,
a \textit{covariant functor}
\(F:\,\varphi\mapsto F(\varphi)=\tilde{\varphi}\)
can be used to map a morphism \(\varphi\) of one category
to an induced morphism \(\tilde{\varphi}\) of another category.

In the present situation,
we denote by \(\mathcal{G}\) the category of groups and
we define the domain of the functor \(F\) as the following \textit{category} \(\mathcal{G}_s\).
The objects of the category are pairs \((G,U)\)
consisting of a group \(G\) and a \textit{subgroup} \(U\le G\),

\begin{equation}
\label{eqn:ObjGSub}
\mathrm{Obj}(\mathcal{G}_s)=\lbrace (G,U)\mid G\in\mathrm{Obj}(\mathcal{G}),\ U\le G\rbrace.
\end{equation}

\noindent
For two objects \((G,U),(H,V)\in\mathrm{Obj}(\mathcal{G}_s)\),
the set of morphisms \(\mathrm{Mor}_{\mathcal{G}_s}((G,U),(H,V))\)
consists of \textit{epimorphisms} \(\varphi:\,G\to H\) such that \(\varphi(G)=H\) and \(\varphi(U)=V\),
briefly written as arrows \(\varphi:\,(G,U)\to (H,V)\),

\begin{equation}
\label{eqn:MorGSub}
\mathrm{Mor}_{\mathcal{G}_s}((G,U),(H,V))=\lbrace\varphi\in\mathrm{Mor}_{\mathcal{G}}(G,H)\mid\varphi(G)=H,\ \varphi(U)=V\rbrace.
\end{equation}

\noindent
The functor \(F:\,\mathcal{G}_s\to\mathcal{A}\)
from this category \(\mathcal{G}_s\) to the category \(\mathcal{A}\) of abelian groups
maps a pair \((G,U)\in\mathrm{Obj}(\mathcal{G}_s)\) 
to the commutator quotient group \(F((G,U)):=U/U^\prime\in\mathrm{Obj}(\mathcal{A})\) of the subgroup \(U\),
and it maps a morphism \(\varphi\in\mathrm{Mor}_{\mathcal{G}_s}((G,U),(H,V))\)
to the induced epimorphism \(F(\varphi):=\tilde{\varphi}\in\mathrm{Mor}_{\mathcal{A}}(U/U^\prime,V/V^\prime)\)
of the restriction \(\varphi\vert_U:\,U\to V\),

\begin{equation}
\label{eqn:FunctorGSub}
F:\,\mathcal{G}_s\to\mathcal{A},\ F((G,U)):=U/U^\prime,\ F(\varphi):=\tilde{\varphi}.
\end{equation}

\noindent
Existence and uniqueness of \(F(\varphi):=\tilde{\varphi}\) have been proved in Theorem
\ref{thm:TargetSing}
under the assumption that \(\varphi(U)=V\),
which is satisfied according to the definition of the arrow \(\varphi:\,(G,U)\to (H,V)\)
and automatically implies \(\varphi(U^\prime)=V^\prime\).


\begin{definition}
\label{dfn:poTarget}

Due to the results in Theorem
\ref{thm:TargetSing},
it makes sense to define a \textit{pre-order of transfer targets}
on the image \(F(\mathrm{Obj}(\mathcal{G}_s))\) of the functor \(F\)
in the object class \(\mathrm{Obj}(\mathcal{A})\)
of the category \(\mathcal{A}\) of abelian groups
in the following manner.

For two objects \((G,U),(H,V)\in\mathrm{Obj}(\mathcal{G}_s)\),
a morphism \(\varphi\in\mathrm{Mor}_{\mathcal{G}_s}((G,U),(H,V))\),
and the images
\(F((G,U))=U/U^\prime,\ F((H,V))=V/V^\prime\in\mathrm{Obj}(\mathcal{A})\),
and \(F(\varphi)=\tilde{\varphi}\in\mathrm{Mor}_{\mathcal{A}}(U/U^\prime,V/V^\prime)\),\\
let (non-strict) \textit{precedence} be defined by

\begin{equation}
\label{eqn:TargetPrecedence}
V/V^\prime\preceq U/U^\prime
:\Longleftrightarrow
V/V^\prime\simeq (U/U^\prime)/\ker(\tilde{\varphi}),
\end{equation}

\noindent
and let \textit{equality} be defined by

\begin{equation}
\label{eqn:TargetEquality}
V/V^\prime=U/U^\prime
:\Longleftrightarrow
V/V^\prime\simeq U/U^\prime,
\end{equation}

\noindent
if the induced epimorphism
\(\tilde{\varphi}:\,U/U^\prime\to V/V^\prime\)
is an isomorphism.

\end{definition}


\begin{corollary}
\label{cor:poTarget}
If both components of the pairs \((G,U),(H,V)\in\mathrm{Obj}(\mathcal{G}_s)\)
are restricted to Hopfian groups,
then the pre-order of transfer targets
\(V/V^\prime\preceq U/U^\prime\) is actually a partial order.
\end{corollary}

\begin{proof}
We use the functorial properties of the functor \(F\).
The \textit{reflexivity} of the partial order follows from the functorial identity in Formula
(\ref{eqn:FunctorialIdentity}),
and the \textit{transitivity} is a consequence of the functorial compositum in Formula
(\ref{eqn:FunctorialCompositum}),
given in the appendix.
The \textit{antisymmetry} might be a problem for infinite groups,
since it is known that there exist so-called \textit{non-Hopfian} groups.
However, for finite groups, and more generally for \textit{Hopfian} groups,
it is due to the implication
\(V/V^\prime\simeq (U/U^\prime)/\ker(\tilde{\varphi_1})
\simeq\Bigl((V/V^\prime)/\ker(\tilde{\varphi_2})\Bigr)/\ker(\tilde{\varphi_1})\)
\(\Longrightarrow\)
\(\ker(\tilde{\varphi_1})=\ker(\tilde{\varphi_2})=1\).
\end{proof}



\subsection{Singulets of transfer kernels}
\label{ss:KernelSing}

Suppose that \(G\) and \(T\) are groups, \(H=\mathrm{im}(\varphi)=\varphi(G)\le T\) is the image of \(G\)
under a homomorphism \(\varphi:\ G\to T\),
and \(V=\varphi(U)\) is the image of a subgroup \(U\le G\) of finite index \(1\le n:=(G:U)<\infty\).
Let \(T_{G,U}\) be the Artin transfer from \(G\) to \(U/U^\prime\).

\begin{theorem}
\label{thm:KernelSing}

If \(\ker(\varphi)\le U\),
then the image \((\varphi(\ell_1),\ldots,\varphi(\ell_n))\)
of a left transversal \((\ell_1,\ldots,\ell_n)\) of \(U\) in \(G\)
is a left transversal  of \(V\) in \(H\),
the index \((H:V)=(G:U)=n<\infty\) remains the same and is therefore finite,
and the Artin transfer \(T_{H,V}\) from \(H\) to \(V/V^\prime\) exists.

\begin{enumerate}

\item
The following connections exist between the two Artin transfers:\\
the required condition for the composita
of the commutative diagram in Figure
\ref{fig:EpiAndArtTfe},

\begin{equation}
\label{eqn:FunctorialMorphism}
\tilde{\varphi}\circ T_{G,U}=T_{H,V}\circ\varphi,
\end{equation}

\noindent
and, consequently, the inclusion of the kernels,

\begin{equation}
\label{eqn:KernelInclusion}
\varphi(\ker(T_{G,U}))\le\ker(T_{H,V}).
\end{equation}

\item
A sufficient (but not necessary) condition for the equality of the kernels is:

\begin{equation}
\label{eqn:KernelCoincidence}
\ker(T_{H,V})=\varphi(\ker(T_{G,U}))
\qquad \text{ if } \qquad
\ker(\varphi)\le U^\prime.
\end{equation}

\end{enumerate}

\end{theorem}

See Figure
\ref{fig:EpiAndArtTfe}
for a visualization of this scenario.



\begin{figure}[ht]
\caption{Epimorphism and Artin transfer}
\label{fig:EpiAndArtTfe}


\setlength{\unitlength}{1cm}
\begin{picture}(6,5)(-3,-6.5)

\put(-2,-2){\makebox(0,0)[cc]{\(H\)}}
\put(-2,-5.5){\vector(0,1){3}}
\put(-2.1,-4){\makebox(0,0)[rc]{\(\varphi\)}}
\put(-2,-6){\makebox(0,0)[cc]{\(G\)}}

\put(0,-1.9){\makebox(0,0)[cb]{\(T_{H,V}\)}}
\put(-1.5,-2){\vector(1,0){3}}
\put(-1.5,-6){\vector(1,0){3}}
\put(0,-6.1){\makebox(0,0)[ct]{\(T_{G,U}\)}}

\put(2,-2){\makebox(0,0)[cc]{\(V/V^\prime\)}}
\put(2,-5.5){\vector(0,1){3}}
\put(2.1,-4){\makebox(0,0)[lc]{\(\tilde{\varphi}\)}}
\put(2,-6){\makebox(0,0)[cc]{\(U/U^\prime\)}}

\end{picture}

\end{figure}



\begin{proof}
The truth of these statements can be justified in the following way.
The first part has been proved in Proposition
\ref{prp:TransvUnderHom}
already:
Let \((\ell_1,\ldots,\ell_n)\) be a left transversal of \(U\) in \(G\).
Then \(G=\dot{\cup}_{i=1}^n\,\ell_iU\) is a disjoint union
but the union \(\varphi(G)=\cup_{i=1}^n\,\varphi(\ell_i)\varphi(U)\) is not necessarily disjoint.
For \(1\le j,k\le n\),
we have \(\varphi(\ell_j)\varphi(U)=\varphi(\ell_k)\varphi(U)\)
\(\Leftrightarrow\) \(\varphi(U)=\varphi(\ell_j)^{-1}\varphi(\ell_k)\varphi(U)=\varphi(\ell_j^{-1}\ell_k)\varphi(U)\)
\(\Leftrightarrow\) \(\varphi(\ell_j^{-1}\ell_k)=\varphi(u)\) for some element \(u\in U\)
\(\Leftrightarrow\) \(\varphi(u^{-1}\ell_j^{-1}\ell_k)=1\)
\(\Leftrightarrow\) \(u^{-1}\ell_j^{-1}\ell_k=:x\in\ker(\varphi)\).
However, if the condition \(\ker(\varphi)\unlhd U\) is satisfied, then we are able to conclude that\\
\(\ell_j^{-1}\ell_k=ux\in U\), and thus \(j=k\).

Let \(\tilde{\varphi}:\ U/U^\prime\to V/V^\prime\)
be the epimorphism obtained in the manner indicated in the proof of Theorem
\ref{thm:TargetSing}
and Formula
(\ref{eqn:IndEpi}).
For the image of \(x\in G\) under the Artin transfer, we obtain

\begin{equation*}
\begin{aligned}
\tilde{\varphi}(T_{G,U}(x))
&= \tilde{\varphi}\Bigl(\prod_{i=1}^n\,\ell_{\lambda_x(i)}^{-1}x\ell_i\cdot U^\prime\Bigr)
=\prod_{i=1}^n\,\tilde{\varphi}\Bigl(\ell_{\lambda_x(i)}^{-1}x\ell_i\cdot U^\prime\Bigr) \\
&=\prod_{i=1}^n\,\varphi(\ell_{\lambda_x(i)}^{-1}x\ell_i)\cdot\varphi(U^\prime)
=\prod_{i=1}^n\,\varphi(\ell_{\lambda_x(i)})^{-1}\varphi(x)\varphi(\ell_i)\cdot\varphi(U^\prime).
\end{aligned}
\end{equation*}

\noindent
Since \(\varphi(U^\prime)=\varphi(U)^\prime=V^\prime\),
the right hand side equals \(T_{H,V}(\varphi(x))\),
provided that \((\varphi(\ell_1),\ldots,\varphi(\ell_n))\) is a left transversal of \(V\) in \(H\),
which is correct when \(\ker(\varphi)\unlhd U\).
This shows that the diagram in Figure
\ref{fig:EpiAndArtTfe}
is commutative, that is,
\(\tilde{\varphi}\circ T_{G,U}=T_{H,V}\circ\varphi\).
It also yields the connection between the permutations \(\lambda_{\varphi(x)}=\lambda_x\) and
the monomials \(u_{\varphi(x)}(i)=\varphi(u_x(i))\) for all \(1\le i\le n\). 
As a consequence, we obtain the inclusion
\(\varphi(\ker(T_{G,U}))\le\ker(T_{H,V})\), if \(\ker(\varphi)\unlhd U\).
Finally, if \(\ker(\varphi)\unlhd U^\prime\),
then the previous section has shown that \(\tilde{\varphi}\) is an isomorphism.
Using the inverse isomorphism, we get
\(T_{G,U}=\tilde{\varphi}^{-1}\circ T_{H,V}\circ\varphi\),
which proves the equation \(\varphi(\ker(T_{G,U}))=\ker(T_{H,V})\).
More explicitly, we have the following chain of equivalences and implications:
\[x\in\ker(T_{G,U}) \Longleftrightarrow T_{G,U}(x)=1
\Longrightarrow 1=\tilde{\varphi}(T_{G,U}(x))=T_{H,V}(\varphi(x))
\Longleftrightarrow \varphi(x)\in\ker(T_{H,V}).\]
Conversely, \(\varphi(x)\in\ker(T_{H,V})\) only implies
\(T_{G,U}(x)\in\ker(\tilde{\varphi})=\Bigl(U^\prime\cdot\ker(\varphi)\cap U\Bigr)/U^\prime\).
Therefore, we certainly have \(\ker(T_{H,V})=\varphi(\ker(T_{G,U}))\)
if \(\ker(\varphi)\le U^\prime\), which is, however, not necessary.
\end{proof}



\noindent
\textbf{Artin transfers as natural transformations.}
Artin transfers \(T_{G,U}\) can be viewed as components of a \textit{natural transformation} \(T\)
between two functors \(?\) and \(F\) from the following category \(\mathcal{G}_{f}\)
to the usual category \(\mathcal{G}\) of groups.

The objects of the category \(\mathcal{G}_{f}\) are pairs \((G,U)\)
consisting of a group \(G\) and a subgroup \(U\le G\) of \textit{finite index} \((G:U)<\infty\),

\begin{equation}
\label{eqn:ObjGFinInd}
\mathrm{Obj}(\mathcal{G}_{f})=\lbrace (G,U)\mid G\in\mathrm{Obj}(\mathcal{G}),\ U\le G,\ (G:U)<\infty\rbrace.
\end{equation}

For two objects \((G,U),(H,V)\in\mathrm{Obj}(\mathcal{G}_{f})\),
the set of morphisms \(\mathrm{Mor}_{\mathcal{G}_{f}}((G,U),(H,V))\)
consists of \textit{epimorphisms} \(\varphi:\,G\to H\) satisfying \(\varphi(G)=H\), \(\varphi(U)=V\),
and the additional condition \(\ker(\varphi)\le U\) for their kernels,
briefly written as arrows \(\varphi:\,(G,U)\to (H,V)\),

\begin{equation}
\label{eqn:MorGFinInd}
\mathrm{Mor}_{\mathcal{G}_{f}}((G,U),(H,V))=\lbrace\varphi\in\mathrm{Mor}_{\mathcal{G}}(G,H)\mid\varphi(G)=H,\ \varphi(U)=V,\ \ker(\varphi)\le U\rbrace.
\end{equation}

The \textit{forgetful functor} \(?:\,\mathcal{G}_{f}\to\mathcal{G}\)
from this category \(\mathcal{G}_{f}\) to the category \(\mathcal{G}\) of groups
maps a pair \((G,U)\in\mathrm{Obj}(\mathcal{G}_{f})\) 
to its first component \(?((G,U)):=G\in\mathrm{Obj}(\mathcal{G})\),
and it maps a morphism \(\varphi\in\mathrm{Mor}_{\mathcal{G}_{f}}((G,U),(H,V))\)
to the underlying epimorphism \(?(\varphi):=\varphi\in\mathrm{Mor}_{\mathcal{G}}(G,H)\).

\begin{equation}
\label{eqn:ForgetGFinInd}
?:\,\mathcal{G}_{f}\to\mathcal{G},\ ?((G,U)):=G,\ ?(\varphi):=\varphi.
\end{equation}

The functor \(F:\,\mathcal{G}_{f}\to\mathcal{G}\)
from \(\mathcal{G}_{f}\) to the category \(\mathcal{G}\) of groups
maps a pair \((G,U)\in\mathrm{Obj}(\mathcal{G}_{f})\) 
to the commutator quotient group \(F((G,U)):=U/U^\prime\in\mathrm{Obj}(\mathcal{G})\) of the subgroup \(U\),
and it maps a morphism \(\varphi\in\mathrm{Mor}_{\mathcal{G}_{f}}((G,U),(H,V))\)
to the induced epimorphism \(F(\varphi):=\tilde{\varphi}\in\mathrm{Mor}_{\mathcal{G}}(U/U^\prime,V/V^\prime)\)
of the restriction \(\varphi\vert_U:\,U\to V\).
Note that we must abstain here from letting \(F\) map into the subcategory \(\mathcal{A}\) of abelian groups.

\begin{equation}
\label{eqn:FunctorGAster}
F:\,\mathcal{G}_{f}\to\mathcal{G},\ F((G,U)):=U/U^\prime,\ F(\varphi):=\tilde{\varphi}.
\end{equation}

\noindent
The system \(T\) of all Artin transfers fulfils the requirements for a natural transformation
\(T:\,?\to F\) between these two functors, since we have

\begin{equation}
\label{eqn:NatTransform}
F(\varphi)\circ T_{G,U}=\tilde{\varphi}\circ T_{G,U}=T_{H,V}\circ\varphi=T_{H,V}\circ ?(\varphi),
\end{equation}

\noindent
for every morphism \(\varphi:\,(G,U)\to (H,V)\) of the category \(\mathcal{G}_{f}\).


\begin{definition}
\label{dfn:poKernel}

Due to the results in Theorem
\ref{thm:KernelSing},
it makes sense to define a \textit{pre-order of transfer kernels}
on the kernels \(\ker(T_{G,U})\) of the components \(T_{G,U}\) of the natural transformation \(T\)
in the object class \(\mathrm{Obj}(\mathcal{G})\)
of the category \(\mathcal{G}\) of groups
in the following manner.

For two objects \((G,U),(H,V)\in\mathrm{Obj}(\mathcal{G}_{f})\),
a morphism \(\varphi\in\mathrm{Mor}_{\mathcal{G}_{f}}((G,U),(H,V))\),
and the images
\(F((G,U))=U/U^\prime,\ F((H,V))=V/V^\prime\in\mathrm{Obj}(\mathcal{G})\),
and \(F(\varphi)=\tilde{\varphi}\in\mathrm{Mor}_{\mathcal{G}}(U/U^\prime,V/V^\prime)\),\\
let (non-strict) \textit{precedence} be defined by

\begin{equation}
\label{eqn:KernelPrecedence}
\ker(T_{G,U})\preceq\ker(T_{H,V})
:\Longleftrightarrow
\varphi(\ker(T_{G,U}))\le\ker(T_{H,V}),
\end{equation}

\noindent
and let \textit{equality} be defined by

\begin{equation}
\label{eqn:KernelEquality}
\ker(T_{G,U})=\ker(T_{H,V})
:\Longleftrightarrow
\varphi(\ker(T_{G,U}))=\ker(T_{H,V}),
\end{equation}

\noindent
if the induced epimorphism
\(\tilde{\varphi}:\,U/U^\prime\to V/V^\prime\)
is an isomorphism.

\end{definition}


\begin{corollary}
\label{cor:poKernel}
If both components of the pairs \((G,U),(H,V)\in\mathrm{Obj}(\mathcal{G}_s)\)
are restricted to Hopfian groups,
then the pre-order of transfer kernels
\(\ker(T_{G,U})\preceq\ker(T_{H,V})\) is actually a partial order.
\end{corollary}

\begin{proof}
Similarly as in the proof of Corollary
\ref{cor:poTarget},
we use the properties of the functor \(F\).
The \textit{reflexivity} is due to the functorial identity in Formula
(\ref{eqn:FunctorialIdentity}).
The \textit{transitivity} is due to the functorial compositum in Formula
(\ref{eqn:FunctorialCompositum}),
where we have to observe the relations
\(\ker(\varphi)\le U\), \(\ker(\psi)\le V\), and Formula
(\ref{eqn:InvAfterDir})
in the appendix
for verifying the kernel relation
\[\ker(\psi\circ\varphi)=(\psi\circ\varphi)^{-1}(1)=\varphi^{-1}(\psi^{-1}(1))
=\varphi^{-1}(\ker(\psi))\le\varphi^{-1}(V)=\varphi^{-1}(\varphi(U))=U\cdot\ker(\varphi)=U\]
additionally to the image relation
\[(\psi\circ\varphi)(U)=\psi(\varphi(U))=\psi(V)=W.\]
The \textit{antisymmetry} is certainly satisfied for finite groups,
and more generally for \textit{Hopfian} groups.
\end{proof}



\subsection{Multiplets of transfer targets and kernels}
\label{ss:KernelTargetMult}

Instead of viewing various pairs \((G,U)\) which share the same first component \(G\)
as distinct objects in the categories \(\mathcal{G}_s\), resp. \(\mathcal{G}_f\),
which we used for describing singulets of transfer targets, resp. kernels,
we now consider a collective accumulation of singulets in \textit{multiplets}.
For this purpose, we shall define a new category \(\mathcal{G}_{(f)}\) of families,
which generalizes the category \(\mathcal{G}_f\), rather than the category \(\mathcal{G}_s\).
However, we have to prepare this definition with a criterion for the
compatibility of a system of subgroups with its image under a homomorphism.


\begin{proposition}
\label{prp:InvSetBij}
(See also
\cite[Thm.2.3.4, p.29]{Hl},
\cite[Satz 3.10, p.16]{Hp},
\cite[Thm.2.4, p.6]{Gs},
\cite[Thm.X.21, p.340]{Is}.)

For an epimorphism \(\phi:\, G\to H\) of groups,
the associated set mappings

\begin{equation}
\label{eqn:InvSetBij}
\phi:\,\mathcal{U}\to\mathcal{V} \text{ and } \phi^{-1}:\,\mathcal{V}\to\mathcal{U}
\end{equation}

\noindent
are inverse bijections between the following systems of subgroups

\begin{equation}
\label{eqn:RstSysSbg}
\mathcal{U}:=\lbrace U\mid\ker(\phi)\le U\le G\rbrace \text{ and } \mathcal{V}:=\lbrace V\mid V\le H\rbrace.
\end{equation}

\end{proposition}


\begin{proof}
The fourth and fifth statement of Lemma
\ref{lem:NrmSbgAndKer}
in the appendix
show that usually the associated set mappings \(\phi^{-1}\) and \(\phi\)
of a homomorphism \(\phi:\,G\to H\)
are not inverse bijections between systems of subgroups of \(G\) and \(H\).
However, if we replace the homomorphism \(\phi\) by an epimorphism with \(H=\phi(G)\),
then the Formula
(\ref{eqn:DirAfterInv})
yields the first desired equality
\[\phi(\phi^{-1}(V))=\phi(G)\cap V=H\cap V=V \text{ for all } V\in\mathcal{V}:=\lbrace V\mid V\le H\rbrace.\]
Guided by the property \(\ker(\phi)=\phi^{-1}(1)\le\phi^{-1}(V)\)
of all pre-images \(\phi^{-1}(V)\) of \(V\in\mathcal{V}\),
we define a restricted system of subgroups of the domain \(G\),
\[\mathcal{U}:=\lbrace U\mid\ker(\phi)\le U\le G\rbrace,\]
and, according to Formula
(\ref{eqn:InvAfterDir}),
we consequently obtain the second required equality
\[\phi^{-1}(\phi(U))=U\cdot\ker(\phi)=U \text{ for all } U\in\mathcal{U},\]
which yields the crucial pair of inverse set bijections
\[\phi^{-1}:\,\mathcal{V}\to\mathcal{U} \text{ and } \phi:\,\mathcal{U}\to\mathcal{V}.\]
\end{proof}


After this preparation, we are able to specify the new category \(\mathcal{G}_{(f)}\).
The objects of the category \(\mathcal{G}_{(f)}\) are pairs \((G,(U_i)_{i\in I})\)
consisting of a group \(G\)
and the \textit{family} of all subgroups \(U_i\le G\) with \textit{finite index} \((G:U_i)<\infty\),

\begin{equation}
\label{eqn:ObjGFam}
\mathrm{Obj}(\mathcal{G}_{(f)})=\lbrace (G,(U_i)_{i\in I})\mid
G\in\mathrm{Obj}(\mathcal{G}),\ U_i\le G,\ (G:U_i)<\infty, \text{ for all } i\in I\rbrace,
\end{equation}

\noindent
where \(I\) denotes a suitable indexing set.
Note that \(G\) itself is one of the subgroups \(U_i\).

The morphisms of the new category are subject to more restrictive conditions,
which concern entire families of subgroups instead of just a single subgroup.

For two objects \((G,(U_i)_{i\in I}),(H,(V_i)_{i\in I})\in\mathrm{Obj}(\mathcal{G}_{(f)})\),
the set \(\mathrm{Mor}_{\mathcal{G}_{(f)}}((G,(U_i)_{i\in I}),(H,(V_i)_{i\in I}))\) of morphisms
consists of \textit{epimorphisms} \(\varphi:\,G\to H\) satisfying \(\varphi(G)=H\),
the image conditions \(\varphi(U_i)=V_i\),
and the kernel conditions \(\ker(\varphi)\le U_i\),
which imply the pre-image conditions \(\varphi^{-1}(V_i)=U_i\), for all \(i\in I\),
briefly written as arrows \(\varphi:\,(G,(U_i)_{i\in I})\to (H,(V_i)_{i\in I})\),

\begin{equation}
\label{eqn:MorGFam}
\mathrm{Mor}_{\mathcal{G}_{(f)}}((G,(U_i)_{i\in I}),(H,(V_i)_{i\in I}))=\lbrace\varphi\in\mathrm{Mor}_{\mathcal{G}}(G,H)\mid\varphi(U_i)=V_i,\ \ker(\varphi)\le U_i, \text{ for all } i\in I\rbrace.
\end{equation}

\noindent
Note that, in view of Proposition
\ref{prp:InvSetBij},
we can always use the same indexing set \(I\) for the domain and for the codomain
of morphisms, provided they satisfy the required kernel condition.

Now we come to the essential definition of \textit{Artin transfer patterns}.


\begin{definition}
\label{dfn:FullArtinPattern}
Let \((G,(U_i)_{i\in I})\in\mathrm{Obj}(\mathcal{G}_{(f)})\)
be an object of the category \(\mathcal{G}_{(f)}\).\\
The \textit{transfer target type} (TTT) of \(G\) is the family

\begin{equation}
\label{eqn:TTT}
\tau(G):=\bigl(U_i/U_i^\prime\bigr)_{i\in I} \text{ of abelian groups in }\mathrm{Obj}(\mathcal{A}).
\end{equation}

\noindent
The \textit{transfer kernel type} (TKT) of \(G\) is the family
\begin{equation}
\label{eqn:TKT}
\varkappa(G):=\bigl(\ker(T_{G,U_i})\bigr)_{i\in I} \text{ of groups in }\mathrm{Obj}(\mathcal{G}).
\end{equation}

\noindent
The \textit{complete Artin pattern} of \(G\) is the pair
\begin{equation}
\label{eqn:FullArtinPattern}
\mathrm{AP}_c(G):=\bigl(\tau(G),\varkappa(G)\bigr).
\end{equation}

\end{definition}


The \textit{natural partial order} on TTTs and TKTs is reduced to
the partial order on the components, according to the Definitions
\ref{dfn:poTarget}
and
\ref{dfn:poKernel}.

\begin{definition}
\label{dfn:poArtinPattern}
Let \((G,(U_i)_{i\in I}),(H,(V_i)_{i\in I})\in\mathrm{Obj}(\mathcal{G}_{(f)})\)
be two objects of the category \(\mathcal{G}_{(f)}\),
where all members of the families \((U_i)_{i\in I}\) and \((V_i)_{i\in I}\) are Hopfian groups.\\
Then (non-strict) \textit{precedence} of TTTs is defined by

\begin{equation}
\label{eqn:TTTPrecedence}
\tau(H)\preceq\tau(G)
:\Longleftrightarrow
V_i/V_i^\prime\preceq U_i/U_i^\prime, \text{ for all } i\in I,
\end{equation}

\noindent
and \textit{equality} of TTTs is defined by

\begin{equation}
\label{eqn:TTTEquality}
\tau(H)=\tau(G)
:\Longleftrightarrow
V_i/V_i^\prime=U_i/U_i^\prime, \text{ for all } i\in I.
\end{equation}

\noindent
(Non-strict) \textit{precedence} of TKTs is defined by

\begin{equation}
\label{eqn:TKTPrecedence}
\varkappa(G)\preceq\varkappa(H)
:\Longleftrightarrow
\ker(T_{G,U_i})\preceq\ker(T_{H,V_i}), \text{ for all } i\in I,
\end{equation}

\noindent
and \textit{equality} of TKTs is defined by

\begin{equation}
\label{eqn:TKTEquality}
\varkappa(G)=\varkappa(H)
:\Longleftrightarrow
\ker(T_{G,U_i})=\ker(T_{H,V_i}), \text{ for all } i\in I.
\end{equation}

\end{definition}


\noindent
We partition the indexing set \(I\) in two disjoint components,
according to whether components of the Artin pattern remain fixed or change
under an epimorphism.

\begin{definition}
\label{dfn:StbPol}
Let \((G,(U_i)_{i\in I}),(H,(V_i)_{i\in I})\in\mathrm{Obj}(\mathcal{G}_{(f)})\)
be two objects of the category \(\mathcal{G}_{(f)}\),
and let \(\varphi\in\mathrm{Mor}_{\mathcal{G}_{(f)}}((G,(U_i)_{i\in I}),(H,(V_i)_{i\in I}))\)
be a morphism between these objects.\\
The \textit{stable part} and the \textit{polarized part} of the Artin pattern \(\mathrm{AP}(G)\) of \(G\)
with respect to \(\varphi\) are defined by

\begin{equation}
\label{eqn:StbPol}
\begin{aligned}
\mathrm{Stb}_{\varphi}(G) &:= \lbrace i\in I\mid\ker(\varphi)\le U_i^\prime\rbrace, \\
\mathrm{Pol}_{\varphi}(G) &:= \lbrace i\in I\mid\ker(\varphi)\not\le U_i^\prime\rbrace.
\end{aligned}
\end{equation}

\end{definition}

\noindent
Accordingly, we have

\begin{equation}
\label{eqn:FixedChanged}
\begin{aligned}
V_i/V_i^\prime=U_i/U_i^\prime \text{ and } \ker(T_{G,U_i})=\ker(T_{H,V_i}), \text{ for all } i\in\mathrm{Stb}_{\varphi}(G), \\
V_i/V_i^\prime\prec U_i/U_i^\prime \text{ and } \ker(T_{G,U_i})\preceq\ker(T_{H,V_i}), \text{ for all } i\in\mathrm{Pol}_{\varphi}(G). 
\end{aligned}
\end{equation}

\noindent
Note that the precedence of polarized targets is strict as opposed to polarized kernels.



\subsection{The Artin pattern on a descendant tree}
\label{ss:APDescTree}

Before we specialize to the usual kinds of descendant trees of finite \(p\)-groups
\cite[\S\ 4, pp.163--164]{Ma5}
we consider an \textit{abstract form} of a \textit{rooted directed tree} \(\mathcal{T}\),
which is characterized by two relations.\\
Firstly, a basic relation
\(\pi(G)<G\)
between \textit{parent} and \textit{child} (also called \textit{immediate descendant}),
corresponding to a directed edge \(G\to\pi(G)\) of the tree,
for any vertex \(G\in\mathcal{T}\setminus\lbrace R\rbrace\)
which is different from the root \(R\) of the tree.\\
Secondly, an induced non-strict partial order relation,
\(\pi^n(G)\le G\) for some integer \(n\ge 0\),
between \textit{ancestor} and \textit{descendant},
corresponding to a path \(G=\pi^0(G)\to\pi^1(G)\to\cdots\to\pi^n(G)\) of directed edges,
for an arbitrary vertex \(G\in\mathcal{T}\), that is,
the ancestor \(\pi^n(G)\) is an iterated parent of the descendant.
Note that only an empty path with \(n=0\) starts from the root \(R\) of the tree,
which has no parent.\\
A brief justification of the partial order:
Reflexivity is due to the relation \(G=\pi^0(G)\).
Transitivity follows from the rule \(\pi^m(\pi^n(G))=\pi^{m+n}(G)\).
Antisymmetry is a consequence of the absence of cycles,
that is, \(H=\pi^m(G)=\pi^m(\pi^n(H))=\pi^{m+n}(H)\) implies \(m=n=0\) and thus \(H=G\).

\noindent
\textbf{The category of a tree.}
Now let \(\mathcal{T}\) be a rooted directed tree
whose vertices are groups \(G\in\mathrm{Obj}(\mathcal{G})\).
Then we define \(\mathcal{G}_{\mathcal{T}}\), the \textit{category associated with} \(\mathcal{T}\),
as a subcategory of the category \(\mathcal{G}_{(f)}\)
which was introduced in the Formulas
(\ref{eqn:ObjGFam})
and
(\ref{eqn:MorGFam}).

The objects of the category \(\mathcal{G}_{\mathcal{T}}\) are those pairs \((G,(U_i)_{i\in I})\)
in the object class of the category \(\mathcal{G}_{(f)}\)
whose first component is a vertex of the tree \(\mathcal{T}\),

\begin{equation}
\label{eqn:ObjGTree}
\mathrm{Obj}(\mathcal{G}_{\mathcal{T}})=\lbrace (G,(U_i)_{i\in I})\in\mathrm{Obj}(\mathcal{G}_{(f)})\mid
G\in\mathcal{T}\rbrace.
\end{equation}

\noindent
The morphisms of the category \(\mathcal{G}_{\mathcal{T}}\)
are selected along the paths of the tree \(\mathcal{T}\) only.\\
For two objects \((G,(U_i)_{i\in I}),(H,(V_i)_{i\in I})\in\mathrm{Obj}(\mathcal{G}_{\mathcal{T}})\),
the set \(\mathrm{Mor}_{\mathcal{G}_{\mathcal{T}}}((G,(U_i)_{i\in I}),(H,(V_i)_{i\in I}))\) of morphisms
is either empty or consists only of a single element,

\begin{equation}
\label{eqn:MorGTree}
\mathrm{Mor}_{\mathcal{G}_{\mathcal{T}}}((G,(U_i)_{i\in I}),(H,(V_i)_{i\in I}))=
\begin{cases}
\emptyset & \text{ if } H \text{ is not an ancestor of } G, \\
\lbrace\pi^n\rbrace & \text{ if } H=\pi^n(G) \text{ for some integer } n\ge 0.
\end{cases}
\end{equation}

\noindent
In the case of an ancestor-descendant relation between \(H\) and \(G\),
the specification of the supercategory \(\mathcal{G}_{(f)}\) enforces the following constraints
on the unique morphism \(\pi^n\):
the image relations \(\pi^n(U_i)=V_i\) and the kernel relations \(\ker(\pi^n)\le U_i\), for all \(i\in I\).


At this position, we must start to be more concrete.
In the descendant tree \(\mathcal{T}=\mathcal{T}(R)\) of a group \(R\), which is the root of the tree,
the formal \textit{parent operator} \(\pi\) gets a second meaning as a \textit{natural projection}
\(\pi:\,G\to\pi(G)=G/N\), \(x\mapsto\pi(x)=xN\),
from the child \(G\) to its parent \(\pi(G)\),
which is always the quotient of \(G\) by a suitable normal subgroup \(N\unlhd G\).
To be precise, the epimorphism \(\pi=\pi_G\) with kernel \(\ker(\pi)=N\) is actually dependent on its domain \(G\).
Therefore, the formal power \(\pi^n\) is only a convenient abbreviation for
the compositum \(\pi_{\pi^{n-1}(G)}\circ\cdots\circ\pi_{\pi^2(G)}\circ\pi_{\pi(G)}\circ\pi_{G}\).

As described in
\cite{Ma5},
there are several possible selections of the normal subgroup \(N\) in the parent definition \(G/N\).
Here, we would like to emphasize the following three choices
of \textit{characteristic subgroups} \(N\) of the child \(G\).
If \(p\) denotes a prime number and \(\mathcal{T}(R)\) is the descendant tree of a finite \(p\)-group \(R\),
then it is usual to take for \(N\unlhd G\)

\begin{enumerate}

\item
either the last non-trivial member \(N:=\gamma_c(G)\) of the lower central series of \(G\)

\item
or the last non-trivial member \(N:=P_{\tilde{c}-1}(G)\) of the lower exponent-\(p\) central series of \(G\)

\item
or the last non-trivial member \(N:=G^{(d-1)}\) of the derived series of \(G\),

\end{enumerate}

\noindent
where \(c=\mathrm{cl}(G)\) denotes the nilpotency class,
\(\tilde{c}=\mathrm{cl}_p(G)\) the lower exponent \(p\)-class,
and \(d=\mathrm{dl}(G)\) the derived length of \(G\), respectively.

Note that every descendant tree of finite \(p\)-groups
is subtree of a descendant tree with abelian root.
Therefore, it is no loss of generality to restrict our attention
to descendant trees with abelian roots.

\begin{theorem}
\label{thm:Comparability}
A uniform warranty for the \(\mathrm{comparability}\)
of the Artin patterns \((\tau(G),\varkappa(G))\)
of all vertices \(G\) of a descendant tree \(\mathcal{T}=\mathcal{T}(R)\) of finite \(p\)-groups
with abelian root \(R\),
in the sense of the natural partial order,
is given by the following restriction
of the family of subgroups \((U_i)_{i\in I}\) in the corresponding object \((G,(U_i)_{i\in I})\)
of the category \(\mathcal{G}_{\mathcal{T}}\).
The restriction depends on the definition of a parent \(\pi(G)\) in the descendant tree.

\begin{enumerate}

\item
\(\gamma_2(G)\le U_i\) for all \(i\in I\), when \(\pi(G)=G/\gamma_c(G)\) with \(c=\mathrm{cl}(G)\).

\item
\(P_1(G)\le U_i\) for all \(i\in I\), when \(\pi(G)=G/P_{\tilde{c}-1}(G)\) with \(\tilde{c}=\mathrm{cl}_p(G)\).

\item
\(G^{(1)}\le U_i\) for all \(i\in I\), when \(\pi(G)=G/G^{(d-1)}\) with \(d=\mathrm{dl}(G)\).

\end{enumerate}

\end{theorem}

\begin{proof}
If parents are defined by \(\pi(G)=G/\gamma_c(G)\) with \(c=\mathrm{cl}(G)\), then
we have \(\ker(\pi)=\gamma_c(G)\) and \(\ker(\pi^n)=\gamma_{c+1-n}(G)\) for any \(0\le n<c\).
The largest of these kernels arises for \(n=c-1\).
Therefore, uniform comparability of Artin patterns is warranted by the restriction
\(\ker(\pi^{c-1})=\gamma_{2}(G)\le U_i\) for all \(i\in I\).

The parent definition \(\pi(G)=G/P_{c-1}(G)\) with \(c=\mathrm{cl}_p(G)\)
implies \(\ker(\pi)=P_{c-1}(G)\) and \(\ker(\pi^n)=P_{c-n}(G)\) for any \(0\le n<c\).
The largest of these kernels arises for \(n=c-1\).
Consequently, a uniform comparability of Artin patterns is guaranteed by the restriction
\(\ker(\pi^{c-1})=P_{1}(G)\le U_i\) for all \(i\in I\).

Finally, in the case of the parent definition \(\pi(G)=G/G^{(d-1)}\) with \(d=\mathrm{dl}(G)\),
we have \(\ker(\pi)=G^{(d-1)}\) and \(\ker(\pi^n)=G^{(d-n)}\) for any \(0\le n<d\).
The largest of these kernels arises for \(n=d-1\).
Consequently, a uniform comparability of Artin patterns is guaranteed by the condition
\(\ker(\pi^{d-1})=G^{(1)}\le U_i\) for all \(i\in I\).
\end{proof}

Note that the first and third condition coincide
since both, \(G^{(1)}\) and \(\gamma_2(G)\), denote the \textit{commutator subgroup} \(G^\prime\).
So the family \((U_i)_{i\in I}\) is restricted to the \textit{normal subgroups}
which contain \(G^\prime\),
as announced in the paragraph preceding Lemma
\ref{lem:HeadSbg}.

The second condition restricts the family \((U_i)_{i\in I}\)
to the \textit{maximal subgroups} of \(G\)
inclusively the group \(G\) itself and the \textit{Frattini subgroup} \(P_1(G)=\Phi(G)\).



Since we shall mainly be concerned with the first and third parent definition
for descendant trees, that is, either with respect to the lower central series
or to the derived series, the comparability condition in Theorem
\ref{thm:Comparability}
suggests the definition of a category \(\mathcal{G}_{(h)}\)
whose objects are subject to more severe conditions than those in Formula
(\ref{eqn:ObjGFam}),

\begin{equation}
\label{eqn:ObjGHead}
\mathrm{Obj}(\mathcal{G}_{(h)})=\lbrace (G,(U_i)_{i\in I})\mid
G\in\mathrm{Obj}(\mathcal{G}),\ (G:G^\prime)<\infty,\ G^\prime\le U_i\le G, \text{ for all } i\in I\rbrace,
\end{equation}

\noindent
but whose morphism are defined exactly as in Formula
(\ref{eqn:MorGFam}).
The new viewpoint leads to a
corresponding modification of Artin transfer patterns.


\begin{definition}
\label{dfn:RstrArtinPattern}
Let \((G,(U_i)_{i\in I})\in\mathrm{Obj}(\mathcal{G}_{(h)})\)
be an object of the category \(\mathcal{G}_{(h)}\).

\noindent
The \textit{Artin pattern}, more precisely the \textit{restricted Artin pattern}, of \(G\) is the pair

\begin{equation}
\label{eqn:RstrArtinPattern}
\mathrm{AP}(G):=\mathrm{AP}_r(G):=\bigl(\tau(G),\varkappa(G)\bigr),
\end{equation}

\noindent
whose components, the TTT and the TKT of \(G\), are defined as in the Formulas
(\ref{eqn:TTT})
and
(\ref{eqn:TKT}),
but now with respect to the smaller system of subgroups of \(G\).

\end{definition}



The following \textit{Main Theorem} shows that any non-metabelian group \(G\)
with derived length \(\mathrm{dl}(G)\ge 3\) and finite abelianization \(G/G^\prime\)
shares its Artin transfer pattern \(\mathrm{AP}_r(G)=(\tau(G),\varkappa(G))\),
in the \textit{restricted sense},
with its \textit{metabelianization}, that is the second derived quotient \(G/G^{\prime\prime}\).

\begin{theorem}
\label{thm:DrvSer}
\textbf{(Main Theorem.)}
Let \(G\) be a (non-metabelian) group with finite abelianization \(G/G^\prime\),
and denote by \(G^{(n)}\), \(n\ge 0\), the terms of the derived series of \(G\),
that is \(G^{(0)}:=G\) and \(G^{(n+1)}:=\lbrack G^{(n)},G^{(n)}\rbrack\) for \(n\ge 0\),
in particular, \(G^{(1)}=G^\prime\) and \(G^{(2)}=G^{\prime\prime}\),
then

\begin{enumerate}

\item
every subgroup \(U\le G\) which contains the commutator subgroup \(G^\prime\)
is a normal subgroup \(U\unlhd G\) of finite index \(1\le m:=(G:U)<\infty\),

\item
for each \(G^\prime\unlhd U\unlhd G\),
there is a chain of normal subgroups

\begin{equation}
\label{eqn:InclChain}
G^{\prime\prime}\unlhd U^\prime\unlhd G^\prime\unlhd U\unlhd G,
\text{ and } \Bigl(U/G^{\prime\prime}\Bigr)^\prime=U^\prime/G^{\prime\prime},
\end{equation}

\item
for each \(G^\prime\unlhd U\unlhd G\),
the targets of the transfers
\(T_{G,U}:\,G\to U/U^\prime\) and\\
\(T_{G/G^{\prime\prime},U/G^{\prime\prime}}:\,
G/G^{\prime\prime}\to\Bigl(U/G^{\prime\prime}\Bigr)\Bigm/\Bigl(U^\prime/G^{\prime\prime}\Bigr)\)
are equal in the sense of the natural order,

\begin{equation}
\label{eqn:Tau}
U/U^\prime=\Bigl(U/G^{\prime\prime}\Bigr)\Bigm/\Bigl(U^\prime/G^{\prime\prime}\Bigr),
\end{equation}

\item
for each \(G^\prime\unlhd U\unlhd G\),
the kernels of the transfers
\(T_{G,U}:\,G\to U/U^\prime\) and\\
\(T_{G/G^{\prime\prime},U/G^{\prime\prime}}:\,
G/G^{\prime\prime}\to\Bigl(U/G^{\prime\prime}\Bigr)\Bigm/\Bigl(U^\prime/G^{\prime\prime}\Bigr)\)
are equal in the sense of the natural order,

\begin{equation}
\label{eqn:Kappa}
\ker(T_{G,U})=\ker(T_{G/G^{\prime\prime},U/G^{\prime\prime}}).
\end{equation}

\end{enumerate}

\end{theorem}

\begin{proof}
We use the natural epimorphism \(\omega:\,G\to G/G^{\prime\prime}\), \(x\mapsto xG^{\prime\prime}\).

\begin{enumerate}

\item
If \(U\) is an intermediate group \(G^\prime\le U\le G\),
then \(U\unlhd G\) is a normal subgroup of \(G\), according to Lemma
\ref{lem:HeadSbg}.
The assumption \(\infty>(G:G^\prime)=(G:U)\cdot (U:G^\prime)\) implies
that \(m:=(G:U)\) is a divisor of the integer \((G:G^\prime)\).
Therefore, the Artin transfer \(T_{G,U}\) exists.

\item
Firstly, \(U\le G\) implies \(U^\prime=\lbrack U,U\rbrack\le\lbrack G,G\rbrack=G^\prime\).
Since \(U^\prime\) is characteristic in \(U\), we also have \(U^\prime\unlhd G^\prime\).
Similarly, \(G^{\prime\prime}\) is characteristic in \(G^\prime\) and thus normal in \(U^\prime\).
Finally, we obtain
\(\Bigl(U/G^{\prime\prime}\Bigr)^\prime=\omega(U)^\prime=\omega(U^\prime)=U^\prime/G^{\prime\prime}\).

\item
The mapping \(f:\,U/G^{\prime\prime}\to U/U^\prime\), \(xG^{\prime\prime}\mapsto xU^\prime\),
is an epimorphism with kernel \(\ker(f)=U^\prime/G^{\prime\prime}\).
Consequently, the isomorphism theorem in Remark
\ref{rmk:FctThrQtn}
of the appendix yields the isomorphism
\(U/U^\prime=\mathrm{im}(f)\simeq\Bigl(U/G^{\prime\prime}\Bigr)\Bigm/\ker(f)
=\Bigl(U/G^{\prime\prime}\Bigr)\Bigm/\Bigl(U^\prime/G^{\prime\prime}\Bigr)\).

\item
Firstly, the restriction \(\omega\vert_U:\,U\to U/G^{\prime\prime}\) is an epimorphism
which induces an isomorphism
\(\tilde{\omega}:\,U/U^\prime\to\Bigl(U/G^{\prime\prime}\Bigr)\Bigm/\Bigl(U^\prime/G^{\prime\prime}\Bigr)\),
since \(\omega(U^\prime)=U^\prime/G^{\prime\prime}\) and
\(\ker(\tilde{\omega})=\omega^{-1}(\omega(U^\prime))/U^\prime
=\Bigl(U^\prime\cdot\ker(\omega)\cap U\Bigr)\Bigm/U^\prime=U^\prime/U^\prime\simeq 1\),
according to Theorem
\ref{thm:TargetSing}.
Secondly, according to Theorem
\ref{thm:KernelSing},
the condition \(\ker(\omega)=G^{\prime\prime}\le U\) implies that the index
\(\Bigl((G/G^{\prime\prime}):(U/G^{\prime\prime})\bigr)=(G:U)=m\) is finite,
the Artin transfer \(T_{G/G^{\prime\prime},U/G^{\prime\prime}}\) exists,
the composite mappings
\(\tilde{\omega}\circ T_{G,U}=T_{G/G^{\prime\prime},U/G^{\prime\prime}}\circ\omega\) commute, and,
since we even have \(\ker(\omega)=G^{\prime\prime}\le U^\prime\),
the transfer kernels satisfy the relation
\(\ker(T_{G/G^{\prime\prime}})=\omega(\ker(T_{G,U}))=\ker(T_{G,U})/G^{\prime\prime}\).
In the sense of the natural partial order on transfer kernels this means equality
\(\ker(T_{G/G^{\prime\prime}})=\ker(T_{G,U})\),
since \(G^\prime\le\ker(T_{G,U})\le G\) and thus,
similarly as in Proposition
\ref{prp:InvSetBij},
the map \(\omega\) establishes a set bijection between the
systems of subgroups \(\mathcal{U}:=\lbrace U\mid G^\prime\le U\le G\rbrace\)
and \(\mathcal{V}:=\lbrace V\mid G^\prime/G^{\prime\prime}\le V\le G/G^{\prime\prime}\rbrace\),
where \(G^\prime/G^{\prime\prime}=\Bigl(G/G^{\prime\prime}\Bigr)^{\prime}\).
\end{enumerate}
\end{proof}



\begin{remark}
\label{rmk:DrvSer}
At this point it is adequate to emphasize
how similar concepts in previous publications
are related to the concept of Artin patterns.
The \textit{restricted} Artin pattern \(\mathrm{AP}_r(G)\) in Definition
\ref{dfn:RstrArtinPattern}
was essentially introduced in
\cite[Dfn.1.1, p.403]{Ma4},
for a special case already earlier in
\cite[\S\ 1, p.417]{Ma3}.
The name \textit{Artin pattern} appears in
\cite[Dfn.3.1, p.747]{Ma7}
for the first time.
The \textit{complete} Artin pattern \(\mathrm{AP}_c(G)\) in Definition
\ref{dfn:FullArtinPattern}
is new in the present article,
but we should point out that it includes the
iterated IPADs (index-\(p\) abelianization data) in
\cite[Dfn.3.5, p.289]{Ma6}
and the iterated IPODs (index-\(p\) obstruction data) in
\cite[Dfn.4.5]{Ma8}.
\end{remark}



\noindent
In a second remark, we emphasize the importance of the preceding Main Theorem
for arithmetical applications.

\begin{remark}
\label{rmk:ClsTwr}
In algebraic number theory, Theorem
\ref{thm:DrvSer}
has striking consequences for the determination of
the length \(\ell_p(K)\ge 2\) of the \(p\)-class tower \(\mathrm{F}_p^\infty(K)\),
that is the maximal unramified pro-\(p\) extension,
of an algebraic number field \(K\)
with respect to a given prime number \(p\).
It shows the impossibility of deciding,
exclusively with the aid of the \textit{restricted} Artin pattern \(\mathrm{AP}_r(G)\),
which of several assigned candidates \(G\) with distinct derived lengths \(\mathrm{dl}(G)\ge 2\)
is the actual \(p\)-class tower group \(\mathrm{Gal}(\mathrm{F}_p^\infty(K)\vert K)\).
(In contrast, \(\ell_p(K)=1\) can always be recognized with \(\mathrm{AP}_r(G)\).)

This is the point where the \textit{complete} Artin pattern \(\mathrm{AP}_c(G)\) enters the stage.
Most recent investigations by means of iterated IPADs of \(2^{\mathrm{nd}}\) order,
whose components are contained in \(\mathrm{AP}_c(G)\),
enabled decisions between \(2\le\ell_p(K)\le 3\) in
\cite{Ma6,Ma8}.

Another successful method is to employ cohomological results by I.R. Shafarevich
on the relation rank \(d_2(G)=\dim_{\mathbb{F}_p}\mathrm{H}^2(G,\mathbb{F}_p)\)
for selecting among several candidates \(G\) for the \(p\)-class tower group,
in dependence on the torsion-free unit rank of the base field \(K\),
for instance in
\cite{BuMa,Ma7}.
\end{remark}

\noindent
Important examples for the concepts in \S\
\ref{s:TransferTargetKernel}
are provided in the following subsections.



\subsection{Abelianization of type \((p,p)\)}
\label{ss:DrvQtnTypePePe}

Let \(G\) be a \(p\)-group with abelianization \(G/G^\prime\) of elementary abelian type \((p,p)\).
Then \(G\) has \(p+1\) maximal subgroups \(H_i<G\) \((1\le i\le p+1)\) of index \((G:H_i)=p\).
For each \(1\le i\le p+1\), let \(T_i:\,G\to H_i/H_i^\prime\) be
the Artin transfer homomorphism from \(G\) to the abelianization of \(H_i\).

\begin{definition}
\label{dfn:DrvQtnTypePePe}

The family of normal subgroups \(\varkappa_H(G)=(\ker(T_i))_{1\le i\le p+1}\) is called
the \textit{transfer kernel type} (TKT) of \(G\) \textit{with respect to} \(H_1,\ldots,H_{p+1}\).

\end{definition}


\begin{remark}
\label{rmk:OrbitTypePePe}

For brevity, the TKT is identified with the multiplet \((\varkappa(i))_{1\le i\le p+1}\),
whose integer components are given by

\begin{equation}
\label{eqn:TKTPxP}
\varkappa(i)=
\begin{cases}
0 & \text{ if } \ker(T_i)=G,\\
j & \text{ if } \ker(T_i)=H_j \text{ for some } 1\le j\le p+1.
\end{cases}
\end{equation}

\noindent
Here, we take into consideration that each transfer kernel \(\ker(T_i)\)
must contain the commutator subgroup \(G^\prime\) of \(G\),
since the transfer target \(H_i/H_i^\prime\) is abelian.
However, the minimal case \(\ker(T_i)=G^\prime\) cannot occur,
according to Hilbert's Theorem \(94\).

A \textit{renumeration} of the maximal subgroups \(K_i=H_{\pi(i)}\)
and of the transfers \(V_i=T_{\pi(i)}\) by means of a permutation \(\pi\in S_{p+1}\)
gives rise to a new TKT \(\lambda_K(G)=(\ker(V_i))_{1\le i\le p+1}\) with respect to \(K_1,\ldots,K_{p+1}\), identified with \((\lambda(i))_{1\le i\le p+1}\), where

\[\lambda(i)=
\begin{cases}
0 & \text{ if } \ker(V_i)=G,\\
j & \text{ if } \ker(V_i)=K_j \text{ for some } 1\le j\le p+1.
\end{cases}\]

\noindent
It is adequate to view the TKTs \(\lambda_K(G)\sim\varkappa_H(G)\) as \textit{equivalent}. 
Since we have

\begin{center}
\(K_{\lambda(i)}=\ker(V_i)=\ker(T_{\pi(i)})=H_{\varkappa(\pi(i))}=K_{\tilde{\pi}^{-1}(\varkappa(\pi(i)))}\),
\end{center}

\noindent
the relation between \(\lambda\) and \(\varkappa\) is given by
\(\lambda=\tilde{\pi}^{-1}\circ\varkappa\circ\pi\).
Therefore,
\(\lambda\) is another representative of the \textit{orbit} \(\varkappa^{S_{p+1}}\) of \(\varkappa\)
under the operation \((\pi,\mu)\mapsto\tilde{\pi}^{-1}\circ\mu\circ\pi\) of the symmetric group \(S_{p+1}\)
on the set of all mappings from \(\lbrace 1,\ldots,p+1\rbrace\) to \(\lbrace 0,\ldots,p+1\rbrace\),
where the extension \(\tilde{\pi}\in S_{p+2}\) of the permutation \(\pi\in S_{p+1}\)
is defined by \(\tilde{\pi}(0)=0\),
and we formally put \(H_0=G\), \(K_0=G\).

\end{remark}


\begin{definition}
\label{dfn:OrbitTypePePe}

The orbit \(\varkappa(G)=\varkappa^{S_{p+1}}\) of any representative \(\varkappa\)
is an invariant of the \(p\)-group \(G\) and is called its \textit{transfer kernel type}, briefly TKT.

\end{definition}


\begin{remark}
\label{dfn:CounterTypePePe}

This definition of \(\varkappa(G)\) goes back to the origins of the \textit{capitulation theory}
and was introduced by Scholz and Taussky for \(p=3\) in \(1934\)
\cite{SoTa}.
Several other authors used this original definition and investigated capitulation problems further.
In historical order, Chang and Foote in \(1980\)
\cite{CgFt},
Heider and Schmithals in \(1982\)
\cite{HeSm},
Brink in \(1984\)
\cite{Br},
Brink and Gold in \(1987\)
\cite{BrGo},
Nebelung in \(1989\)
\cite{Ne},
and ourselves in \(1991\)
\cite{Ma0}
and in \(2012\)
\cite{Ma2}.\\
In the brief form of the TKT \(\varkappa(G)\),
the natural order is expressed by \(i\prec 0\) for \(1\le i\le p+1\).

Let \(\#\mathcal{H}_0(G):=\#\lbrace 1\le i\le p+1\mid\varkappa(i)=0\rbrace\) denote
the \textit{counter of total transfer kernels} \(\ker(T_i)=G\),
which is an invariant of the group \(G\).
In \(1980\), Chang and Foote
\cite{CgFt}
proved that, for any odd prime \(p\) and for any integer \(0\le n\le p+1\),
there exist metabelian \(p\)-groups \(G\) having abelianization \(G/G^\prime\) of type \((p,p)\)
such that \(\#\mathcal{H}_0(G)=n\).
However, for \(p=2\), there do not exist non-abelian \(2\)-groups \(G\) with \(G/G^\prime\simeq (2,2)\),
such that \(\#\mathcal{H}_0(G)\ge 2\).
Such groups must be metabelian of maximal class.
Only the elementary abelian \(2\)-group \(G=C_2\times C_2\) has \(\#\mathcal{H}_0(G)=3\).

\end{remark}


In the following concrete examples for the counters \(\#\mathcal{H}_0(G)\),
and also in the remainder of this article,
we use identifiers of finite \(p\)-groups in the SmallGroups Library
by Besche, Eick and O'Brien
\cite{BEO1,BEO2}.

\begin{example}
\label{exm:CounterTypePePe}

For \(p=3\), we have the following TKTs \(\varkappa=\varkappa(G)\):

\begin{itemize}

\item
\(\#\mathcal{H}_0(G)=0\) for the extra special group \(G=\langle 27,4\rangle\) of exponent \(9\)
with \(\varkappa=(1,1,1,1)\),

\item
\(\#\mathcal{H}_0(G)=1\) for the two groups \(G\in\lbrace\langle 243,6\rangle,\langle 243,8\rangle\rbrace\)
with \(\varkappa\in\lbrace (0,1,2,2),(2,0,3,4)\rbrace\),

\item
\(\#\mathcal{H}_0(G)=2\) for the group \(G=\langle 243,3\rangle\) with \(\varkappa=(0,0,4,3)\),

\item
\(\#\mathcal{H}_0(G)=3\) for the group \(G=\langle 81,7\rangle\) with \(\varkappa=(2,0,0,0)\),

\item
\(\#\mathcal{H}_0(G)=4\) for the extra special group \(G=\langle 27,3\rangle\) of exponent \(3\)
with \(\varkappa=(0,0,0,0)\).

\end{itemize}

\end{example}



\subsection{Abelianization of type \((p^2,p)\)}
\label{ss:DrvQtnPe2Pe}

Let \(G\) be a \(p\)-group with abelianization \(G/G^\prime\) of non-elementary abelian type \((p^2,p)\).
Then \(G\) possesses \(p+1\) maximal subgroups \(H_{1,i}<G\) \((1\le i\le p+1)\) of index \((G:H_{1,i})=p\),
and \(p+1\) subgroups \(H_{2,i}<G\) \((1\le i\le p+1)\) of index \((G:H_{2,i})=p^2\).

\noindent
\textbf{Convention.}\\
Suppose that \(H_{1,p+1}=\prod_{j=1}^{p+1}\,H_{2,j}\) is the \textit{distinguished maximal subgroup}
which is the product of all subgroups of index \(p^2\),
and \(H_{2,p+1}=\cap_{j=1}^{p+1}\,H_{1,j}\) is the distinguished subgroup of index \(p^2\)
which is the intersection of all maximal subgroups,
that is the \textit{Frattini subgroup} \(\Phi(G)\) of \(G\).


\noindent
\textbf{First layer.}\\
For each \(1\le i\le p+1\),
let \(T_{1,i}:\,G\to H_{1,i}/H_{1,i}^\prime\) be the Artin transfer homomorphism
from \(G\) to the abelianization of \(H_{1,i}\).

\begin{definition}
\label{dfn:DrvQtnTypePe2PeLyr1}

The family \(\varkappa_{1,H}(G)=(\ker(T_{1,i}))_{1\le i\le p+1}\) is called
the \textit{first layer transfer kernel type} of \(G\)
with respect to \(H_{1,1},\ldots,H_{1,p+1}\) and \(H_{2,1},\ldots,H_{2,p+1}\),
and is identified with \((\varkappa_1(i))_{1\le i\le p+1}\), where

\begin{equation}
\label{eqn:TKTP2xPLyr1}
\varkappa_1(i)=
\begin{cases}
0 & \text{ if } \ker(T_{1,i})=H_{1,p+1},\\
j & \text{ if } \ker(T_{1,i})=H_{2,j} \text{ for some } 1\le j\le p+1.
\end{cases}
\end{equation}

\end{definition}

\begin{remark}
\label{rmk:DrvQtnTypePe2Pe}

Here, we observe that each first layer transfer kernel
is of exponent \(p\) with respect to \(G^\prime\)
and consequently cannot coincide with \(H_{1,j}\) for any \(1\le j\le p\),
since \(H_{1,j}/G^\prime\) is cyclic of order \(p^2\),
whereas \(H_{1,p+1}/G^\prime\) is bicyclic of type \((p,p)\).

\end{remark}


\noindent
\textbf{Second layer.}\\
For each \(1\le i\le p+1\),
let \(T_{2,i}:\,G\to H_{2,i}/H_{2,i}^\prime\) be the Artin transfer homomorphism
from \(G\) to the abelianization of \(H_{2,i}\).

\begin{definition}
\label{dfn:DrvQtnTypePe2PeLyr2}

The family \(\varkappa_{2,H}(G)=(\ker(T_{2,i}))_{1\le i\le p+1}\) is called
the \textit{second layer transfer kernel type} of \(G\)
with respect to \(H_{2,1},\ldots,H_{2,p+1}\) and \(H_{1,1},\ldots,H_{1,p+1}\),
and is identified with \((\varkappa_2(i))_{1\le i\le p+1}\), where

\begin{equation}
\label{eqn:TKTP2xPLyr2}
\varkappa_2(i)=
\begin{cases}
0 & \text{ if } \ker(T_{2,i})=G,\\
j & \text{ if } \ker(T_{2,i})=H_{1,j} \text{ for some } 1\le j\le p+1.
\end{cases}
\end{equation}

\end{definition}

\noindent
\textbf{Transfer kernel type.}\\
Combining the information on the two layers,
we obtain the (complete) \textit{transfer kernel type}
\[\varkappa_{H}(G)=(\varkappa_{1,H}(G);\varkappa_{2,H}(G))\]
of the \(p\)-group \(G\) with respect to \(H_{1,1},\ldots,H_{1,p+1}\) and \(H_{2,1},\ldots,H_{2,p+1}\).

\begin{remark}
\label{rmk:OrbitTypePe2Pe}

The distinguished subgroups \(H_{1,p+1}\) and \(H_{2,p+1}=\Phi(G)\) are unique invariants of \(G\)
and should not be renumerated.
However, \textit{independent renumerations}
of the remaining maximal subgroups \(K_{1,i}=H_{1,\tau(i)}\) \((1\le i\le p)\)
and the transfers \(V_{1,i}=T_{1,\tau(i)}\)
by means of a permutation \(\tau\in S_p\),
and of the remaining subgroups \(K_{2,i}=H_{2,\sigma(i)}\) \((1\le i\le p)\) of index \(p^2\)
and the transfers \(V_{2,i}=T_{2,\sigma(i)}\)
by means of a permutation \(\sigma\in S_p\),
give rise to new TKTs
\(\lambda_{1,K}(G)=(\ker(V_{1,i}))_{1\le i\le p+1}\)
with respect to \(K_{1,1},\ldots,K_{1,p+1}\) and \(K_{2,1},\ldots,K_{2,p+1}\),
identified with \((\lambda_1(i))_{1\le i\le p+1}\), where
\[\lambda_1(i)=
\begin{cases}
0 & \text{ if } \ker(V_{1,i})=K_{1,p+1},\\
j & \text{ if } \ker(V_{1,i})=K_{2,j} \text{ for some } 1\le j\le p+1,
\end{cases}\]
and
\(\lambda_{2,K}(G)=(\ker(V_{2,i}))_{1\le i\le p+1}\)
with respect to \(K_{2,1},\ldots,K_{2,p+1}\) and \(K_{1,1},\ldots,K_{1,p+1}\),
identified with \((\lambda_2(i))_{1\le i\le p+1}\), where
\[\lambda_2(i)=
\begin{cases}
0 & \text{ if } \ker(V_{2,i})=G,\\
j & \text{ if } \ker(V_{2,i})=K_{1,j} \text{ for some } 1\le j\le p+1.
\end{cases}\]
It is adequate to view the TKTs
\(\lambda_{1,K}(G)\sim\varkappa_{1,H}(G)\) and \(\lambda_{2,K}(G)\sim\varkappa_{2,H}(G)\)
as \textit{equivalent}.
Since we have
\[K_{2,\lambda_1(i)}=\ker(V_{1,i})=\ker(T_{1,\hat{\tau}(i)})=
H_{2,\varkappa_1(\hat{\tau}(i))}=K_{2,\tilde{\sigma}^{-1}(\varkappa_1(\hat{\tau}(i)))},\]
resp.
\[K_{1,\lambda_2(i)}=\ker(V_{2,i})=\ker(T_{2,\hat{\sigma}(i)})=
H_{1,\varkappa_2(\hat{\sigma}(i))}=K_{1,\tilde{\tau}^{-1}(\varkappa_2(\hat{\sigma}(i)))},\]
the relations between \(\lambda_1\) and \(\varkappa_1\), resp. \(\lambda_2\) and \(\varkappa_2\),
are given by \(\lambda_1=\tilde{\sigma}^{-1}\circ\varkappa_1\circ\hat{\tau}\),
resp. \(\lambda_2=\tilde{\tau}^{-1}\circ\varkappa_2\circ\hat{\sigma}\).
Therefore, \(\lambda=(\lambda_1,\lambda_2)\) is another representative
of the \textit{orbit} \(\varkappa^{S_p\times S_p}\) of \(\varkappa=(\varkappa_1,\varkappa_2)\)
under the \textit{operation}
\[((\sigma,\tau),(\mu_1,\mu_2))\mapsto
(\tilde{\sigma}^{-1}\circ\mu_1\circ\hat\tau,\tilde{\tau}^{-1}\circ\mu_2\circ\hat\sigma)\]
of the product of two symmetric groups \(S_p\times S_p\)
on the set of all pairs of mappings from \(\lbrace 1,\ldots,p+1\rbrace\) to \(\lbrace 0,\ldots,p+1\rbrace\),
where the extensions \(\hat{\pi}\in S_{p+1}\) and \(\tilde{\pi}\in S_{p+2}\) of a permutation \(\pi\in S_p\) are defined by \(\hat{\pi}(p+1)=\tilde{\pi}(p+1)=p+1\) and \(\tilde{\pi}(0)=0\),
and formally\\
\(H_{1,0}=K_{1,0}=G\), \(K_{1,p+1}=H_{1,p+1}\), \(H_{2,0}=K_{2,0}=H_{1,p+1}\),
and \(K_{2,p+1}=H_{2,p+1}=\Phi(G)\).

\end{remark}

\begin{definition}
\label{dfn:OrbitTypePe2Pe}

The orbit \(\varkappa(G)=\varkappa^{S_p\times S_p}\)
of any representative \(\varkappa=(\varkappa_1,\varkappa_2)\)
is an invariant of the \(p\)-group \(G\) and is called
its \textit{transfer kernel type}, briefly TKT.

\end{definition}


\noindent
\textbf{Connections between layers.}\\
The Artin transfer \(T_{2,i}:\,G\to H_{2,i}/H_{2,i}^\prime\)
from \(G\) to a subgroup \(H_{2,i}\) of index \((G:H_{2,i})=p^2\) (\(1\le i\le p+1\))
is the compositum \(T_{2,i}=\tilde{T}_{H_{1,j},H_{2,i}}\circ T_{1,j}\)
of the \textit{induced transfer}
\(\tilde{T}_{H_j,U_i}:\,H_{1,j}/H_{1,j}^\prime\to H_{2,i}/H_{2,i}^\prime\) from \(H_{1,j}\) to \(H_{2,i}\)
(in the sense of Corollary
\ref{cor:FctThrQtn}
or Corollary
\ref{cor:Abelianization}
in the appendix)
and the Artin transfer \(T_{1,j}:\,G\to H_{1,j}/H_{1,j}^\prime\) from \(G\) to \(H_{1,j}\),
for any intermediate subgroup \(H_{2,i}<H_{1,j}<G\) of index \((G:H_{1,j})=p\) (\(1\le j\le p+1\)).
There occur two situations:

\begin{itemize}

\item
For the subgroups \(H_{2,1},\ldots,H_{2,p}\)
only the distinguished maximal subgroup \(H_{1,p+1}\) is an intermediate subgroup.

\item
For the Frattini subgroup \(H_{2,p+1}=\Phi(G)\)
all maximal subgroups \(H_{1,1},\ldots,H_{1,p+1}\) are intermediate subgroups.

\end{itemize}

\noindent
This causes restrictions for the transfer kernel type \(\varkappa_2(G)\) of the second layer,\\
since \(\ker(T_{2,i})=\ker(\tilde{T}_{H_{1,j},H_{2,i}}\circ T_{1,j})\supset\ker(T_{1,j})\), and thus

\begin{itemize}

\item
\(\ker(T_{2,i})\supset\ker(T_{1,p+1})\), for all \(1\le i\le p\),

\item
but even \(\ker(T_{2,p+1})\supset\langle\cup_{j=1}^{p+1}\,\ker(T_{1,j})\rangle\).

\end{itemize}
\noindent
Furthermore, when \(G=\langle x,y\rangle\) with \(x^p\notin G^\prime\) and \(y^p\in G^\prime\),
an element \(xy^{k-1}\) (\(1\le k\le p\)) which is of order \(p^2\) with respect to \(G^\prime\),
can belong to the transfer kernel \(\ker(T_{2,i})\)
only if its \(p\)th power \(x^p\) is contained in \(\ker(T_{1,j})\),
for all intermediate subgroups \(H_{2,i}<H_{1,j}<G\), and thus:

\begin{itemize}

\item
\(xy^{k-1}\in\ker(T_{2,i})\), for certain \(1\le i,k\le p\),
enforces the first layer TKT singulet \(\varkappa_1(p+1)=p+1\),

\item
but \(xy^{k-1}\in\ker(T_{2,p+1})\), for some \(1\le k\le p\),
even specifies the complete first layer TKT multiplet \(\varkappa_1=((p+1)^{p+1})\),
that is \(\varkappa_1(j)=p+1\), for all \(1\le j\le p+1\).

\end{itemize}



\section{Stabilization and polarization in descendant trees}
\label{s:StbAndPol}

Theorem
\ref{thm:DrvSer}
has proved that it suffices to get an overview of the \textit{restricted} Artin patterns
of \textit{metabelian} groups \(G\) with \(\mathrm{dl}(G)=2\),
since groups \(G\) of derived length \(\mathrm{dl}(G)\ge 3\) will certainly reveal
exactly the same patterns as their metabelianizations \(G/G^{\prime\prime}\).

In this section, we present the \textit{complete theory}
of stabilization and polarization of the restricted Artin patterns
for an extensive exemplary case, namely
for all metabelian \(3\)-groups \(G\) with abelianization \(G/G^\prime\) of type \((3,3)\).

Since the bottom layer, resp. the top layer, of the restricted Artin pattern
will be considered in Theorem
\ref{thm:TypeCommLCAndDrvSer}
on the commutator subgroup \(G^\prime\),
resp. Theorem
\ref{thm:TypeAbLCAndDrvSer}
on the entire group \(G\),
we first focus on the \textit{intermediate layer} \((\tau_1,\varkappa_1)\)
of the maximal subgroups \(G^\prime<U_i<G\).


\subsection{\(3\)-groups of non-maximal class}
\label{ss:NonMaxClass}

We begin with groups \(G\) of \textit{non-maximal class}.
Denoting by \(m\) the index of nilpotency of \(G\),
we let \(\chi_j(G)\) with \(2\le j\le m-1\)
be the centralizers
of two-step factor groups \(\gamma_j(G)/\gamma_{j+2}(G)\)
of the lower central series, that is,
the biggest subgroups of \(G\) with the property
\(\lbrack\chi_j(G),\gamma_j(G)\rbrack\le\gamma_{j+2}(G)\).
They form an ascending chain of characteristic subgroups of \(G\),
\(G^\prime\le\chi_2(G)\le\ldots\le\chi_{m-2}(G)<\chi_{m-1}(G)=G\),
which contain the commutator subgroup.
\(\chi_j(G)\) coincides with \(G\) if and only if \(j\ge m-1\).
We characterize the smallest \textit{two-step centralizer}
different from the commutator group
by an \textit{isomorphism invariant}
\(s=s(G)=\min\lbrace 2\le j\le m-1\mid\chi_j(G)>G^\prime\rbrace\).
According to Nebelung
\cite{Ne},
we can assume that \(G\) has order \(\lvert G\rvert=3^n\),
class \(\mathrm{cl}(G)=c=m-1\ge 3\), and coclass \(\mathrm{cc}(G)=r=e-1=n-c\ge 2\),
where \(4\le m<n\le 2m-3\).
Let generators of \(G=\langle x,y\rangle\) be selected such that
\(\gamma_3(G)=\langle y^3,x^3,\gamma_4(G)\rangle\),
\(x\in G\setminus\chi_s(G)\), if \(s<m-1\),
and \(y\in\chi_s(G)\setminus G^\prime\).
Suppose that a fixed ordering of the four maximal subgroups of \(G\) is defined by
\(U_i=\langle g_i,G^\prime\rangle\) with
\(g_1=y\), \(g_2=x\), \(g_3=xy\), and \(g_4=xy^{-1}\).
Let the \textit{main commutator} of \(G\) be declared by
\(s_2=t_2=\lbrack y,x\rbrack\in\gamma_2(G)=G^\prime\)
and \textit{higher commutators} recursively by
\(s_j=\lbrack s_{j-1},x\rbrack\), \(t_j=\lbrack t_{j-1},y\rbrack\in\gamma_j(G)\)
for \(j\ge 3\).
Starting with the powers \(\sigma_3=y^3\), \(\tau_3=x^3\in\gamma_3(G)\), let
\(\sigma_j=\lbrack\sigma_{j-1},x\rbrack\), \(\tau_j=\lbrack\tau_{j-1},y\rbrack\in\gamma_j(G)\)
for \(j\ge 4\),
and put \(\Sigma_4=\langle\sigma_4,\ldots,\sigma_{m-1}\rangle\),
\(T_4=\langle\tau_4,\ldots\tau_{e+1}\rangle\).

\begin{theorem}
\label{thm:NonMaxClass}
\textbf{(Non-maximal class.)}
Let \(G\) be a metabelian \(3\)-group
of nilpotency class \(\mathrm{cl}(G)=c=m-1\ge 4\) and coclass \(\mathrm{cc}(G)=r=e-1\ge 2\)
with abelianization \(G/G^\prime\simeq (3,3)\).
With respect to the projection \(\pi(G)=G/\gamma_c(G)\) onto the parent,
the restricted Artin pattern \(\mathrm{AP}_r(G)=(\tau(G),\varkappa(G))\) of \(G\) reveals

\begin{enumerate}

\item
a \(\mathrm{bipolarization}\) and partial stabilization,
if \(G\) is an \(\mathrm{interface\ group}\) with
bicyclic last lower central equal to the bicyclic first upper central,
more precisely

\begin{equation}
\label{eqn:BiPol}
\mathrm{Pol}_{\pi}(G)=\lbrace 1,2\rbrace,\ \mathrm{Stb}_{\pi}(G)=\lbrace 3,4\rbrace
\quad\Longleftrightarrow\quad
\gamma_{m-1}(G)=\langle\sigma_{m-1},\tau_{e}\rangle=\zeta_1(G),
\end{equation}

\item
a \(\mathrm{unipolarization}\) and partial stabilization,
if \(G\) is a \(\mathrm{core\ group}\) with
cyclic last lower central and bicyclic first upper central,
more precisely

\begin{equation}
\label{eqn:UniPol}
\mathrm{Pol}_{\pi}(G)=\lbrace 1\rbrace,\ \mathrm{Stb}_{\pi}(G)=\lbrace 2,3,4\rbrace
\quad\Longleftrightarrow\quad
\gamma_{m-1}(G)=\langle\sigma_{m-1}\rangle<\zeta_1(G)=\langle\sigma_{m-1},\tau_{e}\rangle,
\end{equation}

\item
a \(\mathrm{nilpolarization}\) and \(\mathrm{total\ stabilization}\)
if \(G\) is a \(\mathrm{core\ group}\) with
cyclic last lower central equal to the cyclic first upper central.

\begin{equation}
\label{eqn:NilPol}
\mathrm{Pol}_{\pi}(G)=\emptyset,\ \mathrm{Stb}_{\pi}(G)=\lbrace 1,2,3,4\rbrace
\quad\Longleftrightarrow\quad
\gamma_{m-1}(G)=\langle\sigma_{m-1}\rangle=\langle\tau_{e+1}\rangle=\zeta_1(G).
\end{equation}

\end{enumerate}

\end{theorem}


\begin{proof}
Theorems
\ref{thm:TargetSing}
and
\ref{thm:KernelSing}
tell us that
for detecting whether stabilization occurs from parent \(\pi(G)\) to child \(G\),
we have to compare the projection kernel \(\ker(\pi)\)
with the commutator subgroups \(U_i^\prime\)
of the four maximal normal subgroups \(U_i\lhd G\), \(1\le i\le 4\).
According to
\cite[Cor.3.2, p.480]{Ma1}
these derived subgroups are given by

\begin{equation}
\label{eqn:CommSbgOfMaxSbg}
\begin{aligned}
U_1^\prime &= \langle t_3,\tau_4,\ldots,\tau_{e+1}\rangle,\\
U_2^\prime &= \langle s_3,\sigma_4,\ldots,\sigma_{m-1}\rangle,\\
U_3^\prime &= \langle s_3t_3,\gamma_4(G)\rangle,\\
U_4^\prime &= \langle s_3t_3^{-1},\gamma_4(G)\rangle,
\end{aligned}
\end{equation}

\noindent
provided the generators of \(G\) are selected as indicated above.
On the other hand, the projection kernel \(\ker(\pi)=\gamma_{m-1}(G)\) is given by

\begin{equation}
\label{eqn:PrjKer}
\ker(\pi)=
\begin{cases}
\langle\sigma_{m-1},\tau_{e}\rangle & \text{ if } e=m-1, \\
\langle\sigma_{m-1}\rangle & \text{ if } e<m-1,\ \zeta_1(G)=\langle\sigma_{m-1},\tau_{e}\rangle, \\
\langle\sigma_{m-1}\rangle=\langle\tau_{e+1}\rangle & \text{ if } e<m-1,\ \zeta_1(G)=\langle\sigma_{m-1}\rangle.
\end{cases}
\end{equation}

\noindent
Combining this information with \(m-1=c\), we obtain the following results.

\begin{itemize}

\item
\(\ker(\pi)=\gamma_{c}(G)\le\gamma_4(G)\subset U_i^\prime\)
for \(3\le i\le 4\) if \(c\ge 4\), independently of \(3\le e\le m-1\).

\item
\(\ker(\pi)=\langle\sigma_{m-1},\tau_{e}\rangle\not\subseteq U_i^\prime\)
for \(1\le i\le 2\) if \(e=m-1\), which implies \(\Sigma_4\cap T_4=1\).

\item
\(\ker(\pi)=\langle\sigma_{m-1}\rangle\subset U_2^\prime\) but
\(\ker(\pi)=\langle\sigma_{m-1}\rangle\not\subseteq U_1^\prime\)
if \(e<m-1\), \(\zeta_1(G)=\langle\sigma_{m-1},\tau_{e}\rangle\),
which also implies \(\Sigma_4\cap T_4=1\).

\item
\(\ker(\pi)=\langle\sigma_{m-1}\rangle\subset U_i^\prime\)
for \(1\le i\le 2\) if \(e<m-1\), \(\zeta_1(G)=\langle\sigma_{m-1}\rangle\),
which implies \(\Sigma_4\cap T_4=\langle\sigma_{m-1}\rangle=\langle\tau_{e+1}\rangle\).

\end{itemize}

\noindent
Taken together, these results justify all claims.
\end{proof}


\begin{example}
\label{exm:NonMaxClass}
Generally, the parent \(\pi(G)\) of an \textit{interface} group \(G\)
\cite[Dfn.3.3, p.430]{Ma4}
with bicyclic last non-trivial lower central
is a vertex of a different coclass graph with lower coclass.
In the case of a \textit{bipolarization}
\cite[Dfn.3.2, p.430]{Ma4},
which is now also characterized via its Artin pattern by Formula
(\ref{eqn:BiPol})
for \(p=3\),
we can express the membership in coclass graphs by the implication:
If \(G\in\mathcal{G}(p,r)\) with \(r\ge 3\), then \(\pi(G)\in\mathcal{G}(p,r-1)\).
A typical example is the group \(G=\langle 2187,77\rangle\) of coclass \(3\)
with parent \(\pi(G)=\langle 243,3\rangle\prec G\) of coclass \(2\), where
\[\tau(\pi(G))=
\lbrack (9,3),(9,3),(3,3,3),(3,3,3)\rbrack\prec\tau(G)=\lbrack (9,9),(9,9),(3,3,3),(3,3,3)\rbrack\]
and
\[\varkappa(G)=(3,4,4,3)\prec\varkappa(\pi(G))=(0,0,4,3).\]

\noindent
In contrast, a \textit{core} group \(G\)
\cite[Dfn.3.3, p.430]{Ma4}
with cyclic last non-trivial lower central
and its parent \(\pi(G)\)
are vertices of the same coclass graph.
In dependence on the \(p\)-rank of its centre \(\zeta_1(G)\),
the Artin pattern either shows a \textit{unipolarization} as in Formula
(\ref{eqn:UniPol}),
if the centre is bicyclic, or a \textit{total stabilization} as in Formula
(\ref{eqn:NilPol}),
if the centre is cyclic.
Typical examples are 
the group \(G=\langle 2187,304\rangle\)
with parent \(\pi(G)=\langle 729,54\rangle\prec G\) both of coclass \(2\),
where the Artin pattern shows a unipolarization
\[\tau(\pi(G))=\lbrack (9,9),(9,3),(9,3),(9,3)\rbrack\prec\tau(G)=\lbrack (27,9),(9,3),(9,3),(9,3)\rbrack\]
and
\[\varkappa(G)=(1,2,3,1)\prec\varkappa(\pi(G))=(0,2,3,1),\]
and the group \(G=\langle 729,45\rangle\)
with parent \(\pi(G)=\langle 243,4\rangle\prec G\) both of coclass \(2\),
where the Artin pattern shows a total stabilization
\[\tau(\pi(G))=\lbrack (3,3,3),(3,3,3),(9,3),(3,3,3)\rbrack=\tau(G)\]
and
\[\varkappa(G)=(4,4,4,3)=\varkappa(\pi(G)).\]
\end{example}


\subsection{\(p\)-groups of maximal class}
\label{ss:MaxClass}

Next we consider \(p\)-groups of \textit{maximal class}, that is, of coclass \(\mathrm{cc}(G)=1\),
but now for an arbitrary prime number \(p\ge 2\).
According to Blackburn
\cite{Bl}
and Miech
\cite{Mi},
we can assume that \(G\) is a metabelian \(p\)-group
of order \(\lvert G\rvert=p^n\) and
nilpotency class \(\mathrm{cl}(G)=c=m-1\), where \(n=m\ge 3\).
Then \(G\) is of coclass \(\mathrm{cc}(G)=r=n-c=1\)
and the commutator factor group \(G/G^\prime\) of \(G\) is of type \((p,p)\).
The lower central series of \(G\) is defined
recursively by \(\gamma_1(G)=G\) and
\(\gamma_j(G)=\lbrack\gamma_{j-1}(G),G\rbrack\) for \(j\ge 2\),
in particular \(\gamma_2(G)=G^\prime\).

The \textit{centralizer}
\(\chi_2(G)
=\lbrace g\in G\mid\lbrack g,u\rbrack\in\gamma_4(G)\text{ for all }u\in\gamma_2(G)\rbrace\)
of the two-step factor group \(\gamma_2(G)/\gamma_4(G)\), that is,
\[\chi_2(G)/\gamma_4(G)
=\mathrm{Centralizer}_{G/\gamma_4(G)}(\gamma_2(G)/\gamma_4(G)),\]
is the biggest subgroup of \(G\) such that
\(\lbrack\chi_2(G),\gamma_2(G)\rbrack\le\gamma_4(G)\).
It is characteristic, contains the commutator subgroup \(G^\prime\), and
coincides with \(G\), if and only if \(m=3\).
Let the \textit{isomorphism invariant} \(k=k(G)\) of \(G\) be defined by
\[\lbrack\chi_2(G),\gamma_2(G)\rbrack=\gamma_{m-k}(G),\]
where \(k=0\) for \(3\le m\le 4\), \(0\le k\le m-4\) for \(m\ge 5\),
and \(0\le k\le\min(m-4,p-2)\) for \(m\ge p+1\),
according to Miech
\cite[p.331]{Mi}.

Suppose that generators of \(G=\langle x,y\rangle\) are selected such that
\(x\in G\setminus\chi_2(G)\), if \(m\ge 4\), and \(y\in\chi_2(G)\setminus G^\prime\).
We define the \textit{main commutator}
\(s_2=\lbrack y,x\rbrack\in\gamma_2(G)\)
and the \textit{higher commutators}
\(s_j=\lbrack s_{j-1},x\rbrack=s_{j-1}^{x-1}\in\gamma_j(G)\) for \(j\ge 3\).

The maximal subgroups \(U_i\) of \(G\)
contain the commutator subgroup \(G^\prime\) of \(G\)
as a normal subgroup of index \(p\) and thus
are of the shape \(U_i=\langle g_i,G^\prime\rangle\).
We define a fixed ordering by
\(g_1=y\) and \(g_i=xy^{i-2}\) for \(2\le i\le p+1\).

\begin{theorem}
\label{thm:MaxClass}
\textbf{(Maximal class.)}
Let \(G\) be a metabelian \(p\)-group
of nilpotency class \(\mathrm{cl}(G)=c=m-1\ge 3\) and coclass \(\mathrm{cc}(G)=r=e-1=1\),
which automatically implies an abelianization \(G/G^\prime\) of type \((p,p)\).
With respect to the projection \(\pi(G)=G/\gamma_c(G)\) onto the parent,
the restricted Artin pattern \(\mathrm{AP}_r(G)=(\tau(G),\varkappa(G))\) of \(G\) reveals

\begin{enumerate}

\item
a \(\mathrm{unipolarization}\) and partial stabilization,
if the first maximal subgroup \(U_1\) of \(G\) is abelian,
more precisely

\begin{equation}
\label{eqn:UniPolMax}
\mathrm{Pol}_{\pi}(G)=\lbrace 1\rbrace,\ \mathrm{Stb}_{\pi}(G)=\lbrace 2,\ldots,p+1\rbrace
\quad\Longleftrightarrow\quad
U_1^\prime=1,\ k=0,
\end{equation}

\item
a \(\mathrm{nilpolarization}\) and \(\mathrm{total\ stabilization}\)
if all four maximal subgroups \(U_i\) of \(G\) are non-abelian,
more precisely

\begin{equation}
\label{eqn:NilPolMax}
\mathrm{Pol}_{\pi}(G)=\emptyset,\ \mathrm{Stb}_{\pi}(G)=\lbrace 1,\ldots,p+1\rbrace
\Longleftrightarrow
U_1^\prime=\langle s_{m-k},\ldots,s_{m-1}\rangle=\gamma_{m-k}(G),\ k\ge 1.
\end{equation}

\end{enumerate}

\noindent
In both cases,
the commutator subgroups of the other maximal normal subgroups of \(G\) are given by

\begin{equation}
\label{eqn:}
U_i^\prime=\langle s_3,\ldots,s_{m-1}\rangle=\gamma_{3}(G) \text{ for } 2\le i\le p+1.
\end{equation}

\end{theorem}


\begin{proof}
We proceed in the same way as in the proof of Theorem
\ref{thm:NonMaxClass}
and compare the projection kernel \(\ker(\pi)\)
with the commutator subgroups \(U_i^\prime\)
of the \(p+1\) maximal normal subgroups \(U_i\lhd G\), \(1\le i\le p+1\).
According to
\cite[Cor.3.1, p.476]{Ma1}
they are given by

\begin{equation}
\label{eqn:CommSbgOfMaxSbgMax}
\begin{aligned}
U_1^\prime &= \langle s_{m-k},\ldots,s_{m-1}\rangle=\gamma_{m-k}(G),\ k\ge 0,\\
U_i^\prime &= \langle s_3,\ldots,s_{m-1}\rangle=\gamma_{3}(G) \text{ for } 2\le i\le p+1.
\end{aligned}
\end{equation}

\noindent
if the generators of \(G\) are chosen as indicated previously.
The cyclic projection kernel is given uniformly by

\begin{equation}
\label{eqn:PrjKerMax}
\ker(\pi)=\gamma_{m-1}(G)=\langle s_{m-1}\rangle=\zeta_1(G).
\end{equation}

\noindent
Using the relation \(m-1=c\), we obtain the following results.

\begin{itemize}

\item
\(\ker(\pi)=\gamma_{c}(G)\le\gamma_3(G)=U_i^\prime\)
for \(2\le i\le p+1\) if \(c\ge 3\).

\item
\(\ker(\pi)=\gamma_{m-1}(G)\le\gamma_{m-k}(G)=U_1^\prime\)
if and only if \(m-1\ge m-k\), that is, \(k\ge 1\).

\end{itemize}

\noindent
The claims follow by applying Theorems
\ref{thm:TargetSing}
and
\ref{thm:KernelSing}.
\end{proof}


\begin{example}
\label{exm:MaxClass}
For \(p=5\), typical examples are 
the group \(G=\langle 625,8\rangle\)
with parent \(\pi(G)=\langle 125,3\rangle\prec G\) both of coclass \(1\),
where the Artin pattern shows a \textit{unipolarization}
\cite[Dfn.3.1, p.413]{Ma4}
\[\tau(\pi(G))=\lbrack (5,5,5),(5,5),(5,5),(5,5),(5,5),(5,5)\rbrack
\prec\tau(G)=\lbrack (5,5),(5,5),(5,5),(5,5),(5,5),(5,5)\rbrack\]
and
\[\varkappa(G)=(1,0,0,0,0,0)\prec\varkappa(\pi(G))=(0,0,0,0,0,0),\]
and the group \(G=\langle 3125,33\rangle\)
with parent \(\pi(G)=\langle 625,7\rangle\prec G\) both of coclass \(1\),
where the Artin pattern shows a total stabilization
\[\tau(\pi(G))=\lbrack (5,5,5),(5,5),(5,5),(5,5),(5,5),(5,5)\rbrack=\tau(G)\]
and
\[\varkappa(G)=(0,0,0,0,0,0)=\varkappa(\pi(G)).\]
\end{example}


\subsection{Extreme interfaces of \(p\)-groups}
\label{ss:Extreme}

Finally,
what can be said about the \textit{extreme} cases
of non-abelian \(p\)-groups having the smallest possible nilpotency class
\(c=3\) for coclass \(r\ge 2\) and
\(c=2\) for coclass \(r=1\)?
In these particular situations,
the answers can be given for arbitrary prime numbers \(p\ge 2\).

\begin{theorem}
\label{thm:Extreme}
Let \(G\) be a metabelian \(p\)-group with abelianization \(G/G^\prime\) of type \((p,p)\).

\begin{enumerate}

\item
If \(G\) is of coclass \(\mathrm{cc}(G)\ge 2\) and nilpotency class \(\mathrm{cl}(G)=3\),
then \(p\ge 3\) must be odd and the coclass must be \(\mathrm{cc}(G)=2\) exactly.

\item
If \(G\) is of coclass \(\mathrm{cc}(G)=1\) and nilpotency class \(\mathrm{cl}(G)=2\),
then \(G\) is an extra special \(p\)-group of order \(p^3\) and exponent \(p\) or \(p^2\).

\end{enumerate}

\noindent
In both cases,
there occurs a \(\mathrm{total\ polarization}\) and no stabilization at all,
more explicitly

\begin{equation}
\label{eqn:TotPol}
\mathrm{Pol}_{\pi}(G)=\lbrace 1,\ldots,p+1\rbrace,\ \mathrm{Stb}_{\pi}(G)=\emptyset.
\end{equation}

\end{theorem}

\begin{proof}
Suppose that \(G\) is a metabelian \(p\)-group with \(G/G^\prime\simeq (p,p)\).

\begin{enumerate}

\item
According to O. Taussky
\cite{Ta},
a \(2\)-group \(G\) with abelianization \(G/G^\prime\) of type \((2,2)\)
must be of coclass \(\mathrm{cc}(G)=1\).
Consequently, \(\mathrm{cc}(G)\ge 2\) implies \(p\ge 3\).

Since the minimal nilpotency class \(c\) of a non-abelian group with coclass \(r\ge 1\)
is given by \(c=r+1\),
the case \(\mathrm{cl}(G)=3\) cannot occur for \(\mathrm{cc}(G)\ge 3\).

So we are considering metabelian \(p\)-groups \(G\) with \(G/G^\prime\simeq (p,p)\),
nilpotency class \(\mathrm{cl}(G)=3\) and coclass \(\mathrm{cc}(G)=2\) for odd \(p\ge 3\),
which form the stem of the isoclinism family \(\Phi_6\) in the sense of P. Hall.
According to
\cite[Lem.3.1, p.446]{Ma4},
the commutator subgroups \(U_i^\prime\) of the maximal subgroups \(U_i<G\)
are cyclic of degree \(p\), for such a group \(G\in\Phi_6(0)\).
However, the kernel of the parent projection \(\pi\) is the bicyclic group
\(\ker(\pi)=\gamma_3(G)\) of type \((p,p)\)
\cite[\S\ 3.5, p.445]{Ma4},
which cannot be contained in any of the cyclic \(U_i^\prime\) with \(1\le i\le p+1\).

\item
According to
\cite[Cor.3.1, p.476]{Ma1},
the commutator subgroups \(U_i^\prime\) of all maximal subgroups \(U_i<G\) are trivial,
for a metabelian \(p\)-group \(G\)
of coclass \(\mathrm{cc}(G)=1\) and nilpotency class \(m-1=c=\mathrm{cl}(G)=2\),
which implies \(k=0\).
Thus, the kernel of the parent projection \(\ker(\pi)=\gamma_2(G)=\langle s_2\rangle\)
is not contained in any \(U_i^\prime=1\).

\end{enumerate}

\noindent
In both cases,
the final claim is a consequence of the Theorems
\ref{thm:TargetSing}
and
\ref{thm:KernelSing}.
\end{proof}


\begin{example}
\label{exm:Extreme}
For \(p=3\), a typical example
for the interface between groups of coclass \(2\) and \(1\) is
and the group \(G=\langle 243,9\rangle\) of coclass \(2\)
with parent \(\pi(G)=\langle 27,3\rangle\prec G\) of coclass \(1\),
where the Artin pattern shows a total polarization
\[\tau(\pi(G))=\lbrack (9,3),(9,3),(9,3),(9,3)\rbrack\prec\lbrack (3,3),(3,3),(3,3),(3,3)\rbrack=\tau(G)\]
and
\[\varkappa(G)=(2,1,4,3)\prec\varkappa(\pi(G))=(0,0,0,0).\]
For \(p=2\), a typical example
for the interface between non-abelian and abelian groups is 
the extra special quaternion group \(Q=\langle 8,4\rangle\)
with parent \(\pi(Q)=\langle 4,2\rangle\prec Q\) both of coclass \(1\),
where the Artin pattern shows a total polarization
\[\tau(\pi(Q))=\lbrack (2),(2),(2)\rbrack\prec\tau(Q)=\lbrack (4),(4),(4)\rbrack\]
and
\[\varkappa(Q)=(1,2,3)\prec\varkappa(\pi(Q))=(0,0,0).\]
\end{example}

Summarizing, we can say that the last three Theorems
\ref{thm:NonMaxClass},
\ref{thm:MaxClass},
and
\ref{thm:Extreme}
underpin the fact that
Artin transfer patterns provide a marvellous tool for classifying finite \(p\)-groups.


\subsection{Bottom and top layer of the Artin pattern}
\label{ss:BottomTop}

We conclude this section with supplementary general results concerning
the \textit{bottom layer} and \textit{top layer} of the restricted Artin pattern.

\begin{theorem}
\label{thm:TypeCommLCAndDrvSer}
(Bottom layer.)
The type of the commutator subgroup \(G^\prime\)
can never remain stable for a \(\mathrm{metabelian}\) vertex \(G\in\mathcal{T}\)
of a descendant tree \(\mathcal{T}\) with respect to
the lower central series or the lower exponent-\(p\) central series or the derived series.
\end{theorem}

\begin{proof}
All possible kernels \(\ker(\pi)=\gamma_c(G)>1\),
resp. \(\ker(\pi)=\mathrm{P}_{c-1}(G)>1\),
resp. \(\ker(\pi)=G^{(d-1)}>1\),
of the parent projections \(\pi\) are non trivial,
and can therefore never be contained in the trivial second derived subgroup \(G^{\prime\prime}\).
According to Theorem
\ref{thm:TargetSing},
the type of the commutator subgroup \(G^\prime\) cannot be stable.
\end{proof}


\begin{example}
\label{exm:TypeComm}

In Example
\ref{exm:NonMaxClass},
we pointed out that the group \(G=\langle 729,45\rangle\) with cyclic centre
and its parent \(\pi(G)=\langle 243,4\rangle\prec G\), both of coclass \(2\),
cannot be distinguished by their TTT
\[\tau_1(\pi(G))=\lbrack (3,3,3),(3,3,3),(9,3),(3,3,3)\rbrack=\tau_1(G)\]
and TKT
\[\varkappa_1(G)=(4,4,4,3)=\varkappa_1(\pi(G)),\]
due to a \textit{total stabilization} of the restricted Artin pattern
as in Formula
(\ref{eqn:NilPol}).
However, the type of their commutator subgroup (the second layer of their TTT)
admits a distinction, since
\[\tau_2(\pi(G))=\lbrack (3,3,3)\rbrack\prec\lbrack (9,3,3)\rbrack=\tau_2(G).\]

\end{example}


\begin{theorem}
\label{thm:TypeAbLCAndDrvSer}
(Top layer.)
In a descendant tree \(\mathcal{T}\) with respect to the lower central series or derived series,
the type of the abelianization \(G/G^\prime\)
of the vertices \(G\in\mathcal{T}\) remains stable.
\end{theorem}

\begin{proof}
This follows from Theorem
\ref{thm:TargetSing},
since even the maximal possible kernel \(\ker(\pi)=\gamma_2(G)=G^\prime\),
resp. \(\ker(\pi)=G^{(1)}=G^\prime\),
of the parent projections \(\pi\)
is contained in the commutator subgroup \(G^\prime\) of \(G\).
\end{proof}


\noindent
We briefly emphasize the different behaviour of trees where parents are defined
with the lower exponent-\(p\) central series.

\begin{theorem}
\label{thm:RankAbLpCSer}
In a descendant tree \(\mathcal{T}\) with respect to the lower exponent-\(p\) central series,
only the \(p\)-rank of the abelianization \(r_p(G/G^\prime)\)
of the vertices \(G\in\mathcal{T}\) remains stable.
\end{theorem}

\begin{proof}
Denote by \(\varrho:=r_p(G/G^\prime)\) the \(p\)-rank of the abelianization of \(G\).
According to Theorem
\ref{thm:TargetSing},
the maximal possible kernel \(\ker(\pi)=\mathrm{P}_1(G)=\Phi(G)\)
of the parent projections \(\pi\)
is the Frattini subgroup which is contained in all maximal subgroups \(U_i\) of \(G\).
According to Proposition
\ref{prp:InvSetBij},
the map \(\pi\) induces a bijection between the sets of maximal subgroups
of the child \(G\) and the parent \(\pi(G)\),
whose cardinality is given by \(\frac{p^{\varrho}-1}{p-1}\).
Consequently, we have \(r_p(\pi(G)/\pi(G)^\prime)=\varrho=r_p(G/G^\prime)\).
\end{proof}



\section{Appendix: Induced homomorphism between quotient groups}
\label{s:HomSngSbg}

\noindent
Throughout this appendix, let \(\phi:\,G\to H\) be a homomorphism
from a source group (domain) \(G\) to a target group (codomain) \(H\).


\subsection{Image, pre-image and kernel}
\label{ss:ImPreKer}

First, we recall some basic facts concerning the image and pre-image of normal subgroups
and the kernel of the homomorphism \(\phi\).


\begin{lemma}
\label{lem:NrmSbgAndKer}

Suppose that \(U\le G\) and \(V\le H\) are subgroups, and \(x,y\in G\) are elements.

\begin{enumerate}

\item
If \(U\unlhd G\) is a normal subgroup of \(G\),
then its image \(\phi(U)\unlhd\phi(G)\) is a normal subgroup of the (total) image \(\mathrm{im}(\phi)=\phi(G)\).

\item
If \(V\unlhd\phi(G)\) is a normal subgroup of the image \(\mathrm{im}(\phi)=\phi(G)\),
then the pre-image \(\phi^{-1}(V)\unlhd G\) is a normal subgroup of \(G\).\\
In particular,
the kernel \(\ker(\phi)=\phi^{-1}(1)\unlhd G\) of \(\phi\) is a normal subgroup of \(G\).

\item
If \(\phi(x)=\phi(y)\),
then there exists an element \(k\in\ker(\phi)\) such that \(x=y\cdot k\).

\item
If \(\phi(x)\in\phi(U)\),
then \(x\in U\cdot\ker(\phi)\), i.e., the pre-image of the image satisfies

\begin{equation}
\label{eqn:InvAfterDir}
U\le\phi^{-1}(\phi(U))=U\cdot\ker(\phi)=\ker(\phi)\cdot U.
\end{equation}

\item
Conversely, the image of the pre-image is given by

\begin{equation}
\label{eqn:DirAfterInv}
\phi(\phi^{-1}(V))=\phi(G)\cap V\le V.
\end{equation}

\end{enumerate}

\end{lemma}


\noindent
The situation of Lemma
\ref{lem:NrmSbgAndKer}
is visualized by Figure
\ref{fig:PreImageKernel},
where we briefly write
\(K:=\ker(\phi)=\phi^{-1}(1)\)
and
\(I:=\mathrm{im}(\phi)=\phi(G)\).



{\normalsize

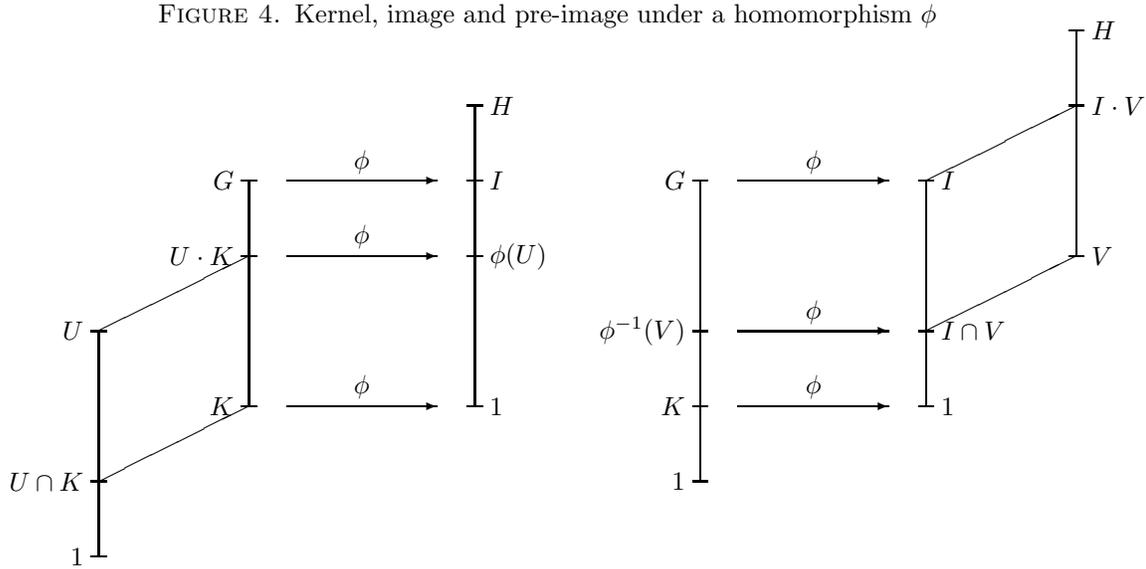
\begin{figure}[ht]
\caption{Kernel, image and pre-image under a homomorphism \(\phi\)}
\label{fig:PreImageKernel}


\setlength{\unitlength}{1cm}
\begin{picture}(15,7)(-8.5,-5)


\put(-7,-2){\line(0,-1){3}}
\multiput(-7.1,-2)(0,-3){2}{\line(1,0){0.2}}
\put(-7.1,-4){\line(1,0){0.2}}
\put(-7.2,-2){\makebox(0,0)[rc]{\(U\)}}
\put(-7.2,-4){\makebox(0,0)[rc]{\(U\cap K\)}}
\put(-7.2,-5){\makebox(0,0)[rc]{\(1\)}}

\multiput(-7,-2)(0,-2){2}{\line(2,1){2}}

\put(-5,0){\line(0,-1){3}}
\multiput(-5.1,0)(0,-3){2}{\line(1,0){0.2}}
\put(-5.1,-1){\line(1,0){0.2}}
\put(-5.2,0){\makebox(0,0)[rc]{\(G\)}}
\put(-5.2,-1){\makebox(0,0)[rc]{\(U\cdot K\)}}
\put(-5.2,-3){\makebox(0,0)[rc]{\(K\)}}

\multiput(-3.5,0.1)(0,-1){2}{\makebox(0,0)[cb]{\(\phi\)}}
\put(-3.5,-2.9){\makebox(0,0)[cb]{\(\phi\)}}
\put(-4.5,0){\vector(1,0){2}}
\multiput(-4.5,-1)(0,-2){2}{\vector(1,0){2}}

\put(-2,1){\line(0,-1){4}}
\multiput(-2.1,1)(0,-1){3}{\line(1,0){0.2}}
\put(-2.1,-3){\line(1,0){0.2}}
\put(-1.8,1){\makebox(0,0)[lc]{\(H\)}}
\put(-1.8,0){\makebox(0,0)[lc]{\(I\)}}
\put(-1.8,-1){\makebox(0,0)[lc]{\(\phi(U)\)}}
\put(-1.8,-3){\makebox(0,0)[lc]{\(1\)}}


\put(1,0){\line(0,-1){4}}
\put(0.9,0){\line(1,0){0.2}}
\multiput(0.9,-2)(0,-1){3}{\line(1,0){0.2}}
\put(0.8,0){\makebox(0,0)[rc]{\(G\)}}
\put(0.8,-2){\makebox(0,0)[rc]{\(\phi^{-1}(V)\)}}
\put(0.8,-3){\makebox(0,0)[rc]{\(K\)}}
\put(0.8,-4){\makebox(0,0)[rc]{\(1\)}}

\multiput(2.5,0.1)(0,-2){2}{\makebox(0,0)[cb]{\(\phi\)}}
\put(2.5,-2.9){\makebox(0,0)[cb]{\(\phi\)}}
\put(1.5,0){\vector(1,0){2}}
\multiput(1.5,-2)(0,-1){2}{\vector(1,0){2}}

\put(4,0){\line(0,-1){3}}
\multiput(3.9,0)(0,-2){2}{\line(1,0){0.2}}
\put(3.9,-3){\line(1,0){0.2}}
\put(4.2,0){\makebox(0,0)[lc]{\(I\)}}
\put(4.2,-2){\makebox(0,0)[lc]{\(I\cap V\)}}
\put(4.2,-3){\makebox(0,0)[lc]{\(1\)}}

\multiput(4,0)(0,-2){2}{\line(2,1){2}}

\put(6,2){\line(0,-1){3}}
\multiput(5.9,2)(0,-3){2}{\line(1,0){0.2}}
\put(5.9,1){\line(1,0){0.2}}
\put(6.2,2){\makebox(0,0)[lc]{\(H\)}}
\put(6.2,1){\makebox(0,0)[lc]{\(I\cdot V\)}}
\put(6.2,-1){\makebox(0,0)[lc]{\(V\)}}

\end{picture}

\end{figure}

}



\begin{remark}
\label{rmk:NrmSbgAndKer}

Note that, in the first statement of Lemma
\ref{lem:NrmSbgAndKer},
we cannot conclude that \(\phi(U)\unlhd H\) is a normal subgroup of the target group \(H\),
and in the second statement of Lemma
\ref{lem:NrmSbgAndKer},
we need not require that \(V\unlhd H\) is a normal subgroup of the target group \(H\).

\end{remark}


\begin{proof}
\begin{enumerate}
\item
If \(U\unlhd G\), then \(x^{-1}Ux=U\) for all \(x\in G\),\\
and thus
\(\phi(x)^{-1}\phi(U)\phi(x)=\phi(x^{-1}Ux)=\phi(U)\) for all \(\phi(x)\in\phi(G)\), i.e., \(\phi(U)\unlhd\phi(G)\).
\item
If \(V\unlhd\phi(G)\), then \((\forall x\in G)\ \phi(x)^{-1}V\phi(x)\subseteq V\),
that is, \((\forall x\in G\,\forall v\in V)\ \phi(x)^{-1}v\phi(x)\in V\).
In particular, we have
\((\forall x\in G\,\forall u\in\phi^{-1}(V))\ \phi(x^{-1}ux)=\phi(x)^{-1}\phi(u)\phi(x)\in V\),
i.e., \((\forall x\in G)\ x^{-1}\phi^{-1}(V)x\subseteq\phi^{-1}(V)\),
and consequently \(\phi^{-1}(V)\unlhd G\).\\
To prove the claim for the kernel, we put \(V:=1\unlhd G\).
\item
If \(\phi(x)=\phi(y)\), then \(\phi(y^{-1}x)=\phi(y)^{-1}\phi(x)=1\), and thus
\(y^{-1}x=k\in\ker(\phi)\).
(See also
\cite[Thm.2.2.1, p.27]{Hl}).
\item
If \(\phi(x)\in\phi(U)\),
then \((\exists u\in U)\ \phi(x)=\phi(u)\),
and thus \((\exists u\in U\,\exists k\in\ker(\phi))\ x=u\cdot k\), by (3).
This shows \(\phi^{-1}(\phi(U))\subseteq U\cdot\ker(\phi)\), and the opposite inclusion is obvious.\\
Finally, since \(\ker(\phi)\) is normal, we have \(U\cdot\ker(\phi)=\ker(\phi)\cdot U\).
\item
This is a consequence of the properties of the set mappings \(\phi^{-1}\) and \(\phi\) associated with the homomorphism \(\phi\).
\end{enumerate}
\end{proof}



\subsection{Criteria for the existence of the induced homomorphism}
\label{ss:CritIndHom}

\noindent
Now we state the central theorem
which provides the foundation for lots of useful applications.
It is the most general version of a series of related theorems,
which is presented in Bourbaki
\cite[Chap.1, Structures alg\'ebriques, Prop.5, p. A I.35]{Bb}.
Weaker versions will be given in the subsequent corollaries.


\begin{theorem}
\label{thm:IndHomQtn}
(Main Theorem)\\
Suppose that \(U\unlhd G\) is a normal subgroup of \(G\) and
\(V\unlhd H\) is a normal subgroup of \(H\).
Let \(\omega_U:\,G\to G/U\) and \(\omega_V:\,H\to H/V\)
denote the canonical projections onto the quotients.

\begin{itemize}
\item
The following three conditions for the homomorphism \(\phi:\,G\to H\) are equivalent.
\begin{enumerate}
\item
There exists an induced homomorphism \(\tilde{\phi}:\,G/U\to H/V\)
such that \(\tilde{\phi}\circ\omega_U=\omega_V\circ\phi\),
that is,

\begin{equation}
\label{eqn:IndHomQtn}
\tilde{\phi}(xU)=\phi(x)V \quad\text{ for all } x\in G.
\end{equation}

\item
\(\phi(U)\subseteq V\).
\item
\(U\subseteq\phi^{-1}(V)\).
\end{enumerate}
\item
If the induced homomorphism \(\tilde{\phi}\) of the quotients exists,
then it is determined uniquely by \(\phi\),
and its kernel, image and cokernel are given by

\begin{equation}
\label{eqn:KerImCoKer}
\begin{aligned}
\ker(\tilde{\phi}) &=& \phi^{-1}(V)/U &\supseteq& (U\cdot\ker(\phi))/U,\\
\mathrm{im}(\tilde{\phi}) &=& (V\cdot\phi(G))/V, && \\
\mathrm{coker}(\tilde{\phi}) &\simeq& H/(V\cdot\phi(G)), && \text{ if } V\cdot\phi(G)\unlhd H.
\end{aligned}
\end{equation}

\noindent
In particular, \(\tilde{\phi}\) is a monomorphism
if and only if \(U=\phi^{-1}(V)\).\\
Moreover, \(\tilde{\phi}\) is an epimorphism if and only if \(H=V\cdot\phi(G)\).\\
In particular, \(\tilde{\phi}\) is certainly an epimorphism if \(\phi\) is onto.
\end{itemize}

\end{theorem}


\noindent
We summarize the criteria for the existence of the induced homomorphism in a formula:

\begin{equation}
\label{eqn:CritIndHom}
(\exists \tilde{\phi}:\,G/U\to H/V)\ \tilde{\phi}\circ\omega_U=\omega_V\circ\phi
\quad\Longleftrightarrow\quad
\phi(U)\subseteq V
\quad\Longleftrightarrow\quad
U\subseteq\phi^{-1}(V).
\end{equation}

\noindent
The situation of Theorem
\ref{thm:IndHomQtn}
is shown in the commutative diagram of Figure
\ref{fig:IndHomQtn}.



{\normalsize

\begin{figure}[ht]
\caption{Induced homomorphism \(\tilde{\phi}\) of quotients}
\label{fig:IndHomQtn}


\setlength{\unitlength}{1cm}
\begin{picture}(6,5)(-3,-6.5)

\put(-2,-2){\makebox(0,0)[cc]{\(G\)}}
\put(-2,-2.5){\vector(0,-1){3}}
\put(-2.1,-4){\makebox(0,0)[rc]{\(\omega_U\)}}
\put(-2,-6){\makebox(0,0)[cc]{\(G/U\)}}

\put(0,-1.9){\makebox(0,0)[cb]{\(\phi\)}}
\put(-1.5,-2){\vector(1,0){3}}
\put(-1.5,-6){\vector(1,0){3}}
\put(0,-6.1){\makebox(0,0)[ct]{\(\tilde{\phi}\)}}

\put(2,-2){\makebox(0,0)[cc]{\(H\)}}
\put(2,-2.5){\vector(0,-1){3}}
\put(2.1,-4){\makebox(0,0)[lc]{\(\omega_V\)}}
\put(2,-6){\makebox(0,0)[cc]{\(H/V\)}}

\end{picture}

\end{figure}

}



\begin{remark}
\label{rmk:IndHomQtn}

If the normal subgroup \(U\unlhd G\) in the assumptions of Theorem
\ref{thm:IndHomQtn}
is taken as \(U:=\phi^{-1}(V)\),
then the induced homomorphism \(\tilde{\phi}\) exists automatically
and is a monomorphism.

Note that \(U=\phi^{-1}(V)\) does not imply \(\phi(U)=V\)
but only \(\phi(U)\subseteq V\),
if \(\phi\) is not an epimorphism.
Similarly, \(\phi(U)=V\) does not imply \(U=\phi^{-1}(V)\)
but only \(U\subseteq\phi^{-1}(V)\),
if \(\phi\) is not a monomorphism.

\end{remark}


\begin{proof}
\begin{itemize}
\item
(1) \(\Longrightarrow\) (2):
If there exists a homomorphism \(\tilde{\phi}:\,G/U\to H/V\)
such that \(\tilde{\phi}(xU)=\phi(x)V\) for all \(x\in G\),
then, for any \(x\in U\), we have \(xU=U=1_{G/U}\),
and thus \(\phi(x)V=\tilde{\phi}(xU)=\tilde{\phi}(1_{G/U})=1_{H/V}=V\),
which means \(\phi(x)\in V\).
It follows that \(\phi(U)\subseteq V\).\\
(2) \(\Longrightarrow\) (1):
If \(\phi(U)\subseteq V\),
then the image \(\tilde{\phi}(xU)\) of the coset \(xU\in G/U\) under \(\tilde{\phi}\)
is independent of the representative \(x\in G\):
If \(xU=yU\) for \(x,y\in G\), then \(y^{-1}x\in U\) and thus
\(\phi(y)^{-1}\phi(x)=\phi(y^{-1}x)\in\phi(U)\subseteq V\).
Consequently, we have \(\tilde{\phi}(xU)=\phi(x)V=\phi(y)V=\tilde{\phi}(yU)\).
Furthermore, \(\tilde{\phi}\) is a homomorphism, since
\[\tilde{\phi}(xU\cdot yU)=\tilde{\phi}(xyU)=\phi(xy)V=
\phi(x)\phi(y)V=\phi(x)V\cdot\phi(y)V=\tilde{\phi}(xU)\cdot\tilde{\phi}(yU).\]
(2) \(\Longrightarrow\) (3):
If \(\phi(U)\subseteq V\), then
\(U\subseteq\phi^{-1}(\phi(U))\subseteq\phi^{-1}(V)\).\\
(3) \(\Longrightarrow\) (2):
If \(U\subseteq\phi^{-1}(V)\), then
\(\phi(U)\subseteq\phi(\phi^{-1}(V))\subseteq V\).
\item
The image of any \(xU\in G/U\) under \(\tilde{\phi}\)
is determined uniquely by \(\phi\), since
\(\tilde{\phi}(xU)=\phi(x)V\).\\
The kernel of \(\tilde{\phi}\) is given by
\(\ker(\tilde{\phi})=\tilde{\phi}^{-1}(1_{H/V})=\tilde{\phi}^{-1}(V)\),
and for \(x\in G\) we have
\[xU\in\tilde{\phi}^{-1}(V)\Longleftrightarrow\phi(x)V=\tilde{\phi}(xU)=V
\Longleftrightarrow\phi(x)\in V\Longleftrightarrow x\in\phi^{-1}(V),\]
that is \(\ker(\tilde{\phi})=\phi^{-1}(V)/U\),
which clearly contains \(\phi^{-1}(\phi(U))/U=(U\cdot\ker(\phi))/U\),
since \(\phi(U)\subseteq V\).\\
The cokernel of \(\tilde{\phi}\) is given by
\(\mathrm{coker}(\tilde{\phi})=(H/V)/\tilde{\phi}(G/U)=(H/V)/(V\cdot\phi(G)/V)\simeq H/(V\cdot\phi(G))\), if \(V\cdot\phi(G)\unlhd H\).\\
Finally, if \(\phi\) is an epimorphism, then
\(\tilde{\phi}\circ\omega_G=\omega_H\circ\phi\) is also an epimorphism,
which forces the terminal map \(\tilde{\phi}\) to be an epimorphism.
\end{itemize}
\end{proof}



\subsection{Factorization through a quotient}
\label{ss:FactThruQtn}

\noindent
Theorem
\ref{thm:IndHomQtn}
can be used to derive numerous special cases.
Usually it suffices to consider the quotient group \(G/U\) corresponding to a normal subgroup \(U\) of the source group \(G\)
of the homomorphism \(\phi:\,G\to H\) and to view the target group \(H\) as the trivial quotient \(H/1\).
In this weaker form, the existence criterion for the induced homomorphism occurs in Lang's book
\cite[p.17]{La}.


\begin{corollary}
\label{cor:FctThrQtn}
(Factorization through a quotient)\\
Suppose \(U\unlhd G\) is a normal subgroup of \(G\) and
\(\omega:\,G\to G/U\) denotes the natural epimorphism onto the quotient.

If \(U\le\ker(\phi)\),
then there exists a unique homomorphism \(\tilde{\phi}:\,G/U\to H\) such that \(\tilde{\phi}\circ\omega=\phi\), that is,
\(\tilde{\phi}(xU)=\phi(x)\) for all \(x\in G\).

Moreover, the kernel of \(\tilde{\phi}\) is given by \(\ker(\tilde{\phi})=\ker(\phi)/U\).

\end{corollary}


\noindent
Again we summarize the criterion in a formula:

\begin{equation}
\label{eqn:FctThrQtn}
(\exists\tilde{\phi}:\,G/U\to H)\ \tilde{\phi}\circ\omega=\phi 
\quad\Longleftrightarrow\quad
\phi(U)=1
\quad\Longleftrightarrow\quad
U\le\ker(\phi).
\end{equation}

\noindent
In this situation the homomorphism \(\phi\) is said to \textit{factor} or \textit{factorize} through the quotient \(G/U\)
via the canonical projection \(\omega\) and the \textit{induced} homomorphism \(\tilde{\phi}\).

The scenario of Corollary
\ref{cor:FctThrQtn}
is visualized by Figure
\ref{fig:FctThrQtn}.



{\normalsize

\begin{figure}[ht]
\caption{Homomorphism \(\phi\) factorized through a quotient}
\label{fig:FctThrQtn}


\setlength{\unitlength}{1cm}
\begin{picture}(6,4.5)(-3,-6)

\put(-2,-2){\makebox(0,0)[cc]{\(G\)}}
\put(-2,-2.5){\vector(1,-2){1.5}}
\put(-1.5,-4){\makebox(0,0)[rc]{\(\omega\)}}

\put(0,-1.9){\makebox(0,0)[cb]{\(\phi\)}}
\put(-1.5,-2){\vector(1,0){3}}

\put(2,-2){\makebox(0,0)[cc]{\(H\)}}
\put(0.5,-5.5){\vector(1,2){1.5}}
\put(1.5,-4){\makebox(0,0)[lc]{\(\tilde{\phi}\)}}

\put(0,-6){\makebox(0,0)[cc]{\(G/U\)}}

\end{picture}

\end{figure}
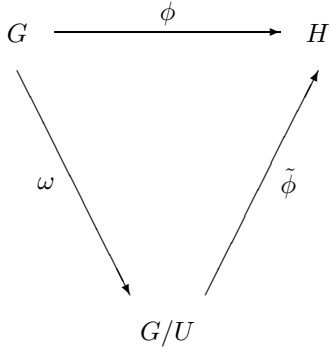

}



\begin{proof}
The claim is a consequence of Theorem
\ref{thm:IndHomQtn}
in the special case that \(V:=1\) is the trivial group.
The equivalent conditions for the existence of the induced homomorphism \(\tilde{\phi}\)
are \(\phi(U)=1\) resp. \(U\subseteq\phi^{-1}(1)=\ker(\phi)\).
\end{proof}


\begin{remark}
\label{rmk:FctThrQtn}
Note that the well-known \textit{isomorphism theorem}
(sometimes also called \textit{homomorphism theorem})
is a special case of Corollary
\ref{cor:FctThrQtn}.
If we put \(U:=\ker(\phi)\) and
if we assume that \(\phi\) is an epimorphism with \(\phi(G)=H\),
then the induced homomorphism \(\tilde{\phi}:\,G/\ker(\phi)\to\phi(G)\)
is an isomorphism, since \(\ker(\tilde{\phi})=\ker(\phi)/U\simeq 1\).
\end{remark}


\noindent
In this weakest form,

\begin{equation}
\label{eqn:IsomThm}
G/\ker(\phi)\simeq\mathrm{im}(\phi),
\end{equation}

\noindent
actually without any additional assumptions being required,
the existence theorem for the induced homomorphism appears
in almost every standard text book on group theory or algebra,
e.g.,
\cite[Thm.2.3.2, p.28]{Hl}
and
\cite[Thm.X.18, p.339]{Is}.



\subsection{Application to series of characteristic subgroups}
\label{ss:SerCharSbg}

\noindent
The normal subgroup \(U\unlhd G\) in the assumptions of Corollary
\ref{cor:FctThrQtn}
can be specialized to various characteristic subgroups of \(G\)
for which the condition \(U\subseteq\ker(\phi)\) can be expressed differently,
namely by \textit{invariants of series of characteristic subgroups}.


\begin{corollary}
\label{cor:CharSbg}
The homomorphism \(\phi:\,G\to H\) can be factorized
through various quotients of \(G\) in the following way.
Let \(n\) be a positive integer and \(p\) be a prime number.

\begin{enumerate}
\item
\(\phi\) factors through the \(n\)th derived quotient \(G/G^{(n)}\)
if and only if the derived length of \(\phi(G)\) is bounded by \(\mathrm{dl}(\phi(G))\le n\).
\item
\(\phi\) factors through the \(n\)th lower central quotient \(G/\gamma_n(G)\)
if and only if the nilpotency class of \(\phi(G)\) is bounded by \(\mathrm{cl}(\phi(G))\le n-1\).
\item
\(\phi\) factors through the \(n\)th lower exponent-\(p\) central quotient \(G/P_n(G)\)
if and only if  the \(p\)-class of \(\phi(G)\) is bounded by \(\mathrm{cl}_p(\phi(G))\le n\).
\end{enumerate}
\end{corollary}


\noindent
We summarize these criteria in terms of the length of series in a formula:

\begin{equation}
\label{eqn:CharSbg}
\begin{aligned}
(\exists\tilde{\phi}:\,G/G^{(n)}\to H)\ \tilde{\phi}\circ\omega=\phi
&\quad\Longleftrightarrow\quad& \mathrm{dl}(\phi(G))\le n,\\
(\exists\tilde{\phi}:\,G/\gamma_n(G)\to H)\ \tilde{\phi}\circ\omega=\phi
&\quad\Longleftrightarrow\quad& \mathrm{cl}(\phi(G))\le n-1, \\
(\exists\tilde{\phi}:\,G/P_n(G)\to H)\ \tilde{\phi}\circ\omega=\phi
&\quad\Longleftrightarrow\quad& \mathrm{cl}_p(\phi(G))\le n.
\end{aligned}
\end{equation}


\begin{proof}
By induction, we show that, firstly,\\
\(\phi(G^{(n)})=\phi(\lbrack G^{(n-1)},G^{(n-1)}\rbrack)=\lbrack \phi(G^{(n-1)}),\phi(G^{(n-1)})\rbrack
=\lbrack (\phi(G))^{(n-1)},(\phi(G))^{(n-1)}\rbrack=(\phi(G))^{(n)}\),\\
secondly,
\(\phi(\gamma_n(G))=\phi(\lbrack G,\gamma_{n-1}(G)\rbrack)=\lbrack \phi(G),\phi(\gamma_{n-1}(G))\rbrack
=\lbrack \phi(G),\gamma_{n-1}(\phi(G))\rbrack=\gamma_n(\phi(G))\),\\
and finally,
\(\phi(P_n(G))=\phi(P_{n-1}(G)^p\cdot\lbrack G,P_{n-1}(G)\rbrack)\)\\
\(=\phi(P_{n-1}(G))^p\cdot\lbrack \phi(G),\phi(P_{n-1}(G))\rbrack
=P_{n-1}(\phi(G))^p\cdot\lbrack \phi(G),P_{n-1}(\phi(G))\rbrack=P_n(\phi(G))\).\\
Now, the claims follow from Corollary
\ref{cor:FctThrQtn}
by observing that
\((\phi(G))^{(n)}=1\) iff \(\mathrm{dl}(\phi(G))\le n\),
\(\gamma_n(\phi(G))=1\) iff \(\mathrm{cl}(\phi(G))\le n-1\), and
\(P_n(\phi(G))=1\) iff \(\mathrm{cl}_p(\phi(G))\le n\)
\end{proof}


\noindent
The following special case is particularly well known.
Here we take the commutator subgroup \(G^\prime\) of \(G\) as our charecteristic subgroup,
which can either be viewed as the term \(\gamma_2(G)\) of the lower central series of \(G\)
or as the term \(G^{(1)}\) of the derived series of \(G\).


\begin{corollary}
\label{cor:Abelianization}
A homomorphism \(\phi:\,G\to H\) passes through the derived quotient \(G/G^\prime\) of its source group \(G\)
if and only if its image \(\mathrm{im}(\phi)=\phi(G)\) is abelian.
\end{corollary}


\begin{proof}
Putting \(n=1\) in the first statement
or \(n=2\) in the second statement
of Corollary
\ref{cor:CharSbg}
we obtain the well-known special case that
\(\phi\) passes through the abelianization \(G/G^\prime\) if and only if \(\phi(G)\) is abelian,
which is equivalent to \(\mathrm{dl}(\phi(G))\le 1\), and also to \(\mathrm{cl}(\phi(G))\le 1\).
\end{proof}


\noindent
The situation of Corollary
\ref{cor:Abelianization}
is visualized in Figure
\ref{fig:PassThruDerivedQtn}.



{\normalsize

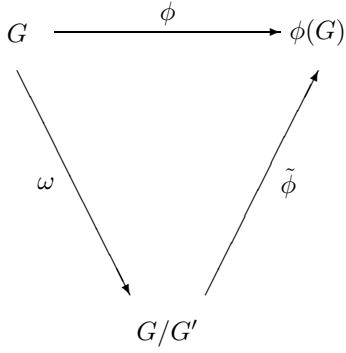
\begin{figure}[ht]
\caption{Homomorphism \(\phi\) passing through the derived quotient}
\label{fig:PassThruDerivedQtn}


\setlength{\unitlength}{1cm}
\begin{picture}(6,4.5)(-3,-6)

\put(-2,-2){\makebox(0,0)[cc]{\(G\)}}
\put(-2,-2.5){\vector(1,-2){1.5}}
\put(-1.5,-4){\makebox(0,0)[rc]{\(\omega\)}}

\put(0,-1.9){\makebox(0,0)[cb]{\(\phi\)}}
\put(-1.5,-2){\vector(1,0){3}}

\put(2,-2){\makebox(0,0)[cc]{\(\phi(G)\)}}
\put(0.5,-5.5){\vector(1,2){1.5}}
\put(1.5,-4){\makebox(0,0)[lc]{\(\tilde{\phi}\)}}

\put(0,-6){\makebox(0,0)[cc]{\(G/G^\prime\)}}

\end{picture}

\end{figure}

}



Using the first part of the proof of Corollary
\ref{cor:CharSbg}
we can recognize the behavior of several central series
under homomorphisms.


\begin{lemma}
\label{lem:CtrSerHom}
Let \(\phi:\,G\to H\) be a homomorphism of groups
and suppose that \(n\ge 1\) is an integer and \(p\ge 2\) a prime number.
Let \(U\le G\) be a subgroup with image \(V=\phi(U)\le H\).
\begin{enumerate}
\item
If \(n:=\mathrm{dl}(U)\), then
\[\mathrm{dl}(V)
\begin{cases}
=n & \Longleftrightarrow U^{(n-1)}\not\subseteq\ker(\phi), \\
<n & \Longleftrightarrow U^{(n-1)}\le\ker(\phi).
\end{cases}
\]
\item
If \(n:=\mathrm{cl}(U)\), then
\[\mathrm{cl}(V)
\begin{cases}
=n & \Longleftrightarrow \gamma_n(U)\not\subseteq\ker(\phi), \\
<n & \Longleftrightarrow \gamma_n(U)\le\ker(\phi).
\end{cases}
\]
\item
If \(n:=\mathrm{cl}_p(U)\), then
\[\mathrm{cl}_p(V)
\begin{cases}
=n & \Longleftrightarrow P_{n-1}(U)\not\subseteq\ker(\phi), \\
<n & \Longleftrightarrow P_{n-1}(U)\le\ker(\phi).
\end{cases}
\]
\end{enumerate}
\end{lemma}


\begin{proof}
\begin{enumerate}
\item
Let \(n:=\mathrm{dl}(U)\), then
\(U^{(n-1)}>U^{(n)}=1\) and
\(V^{(n-1)}=\phi(U)^{(n-1)}=\phi(U^{(n-1)})\ge\phi(U^{(n)})=\phi(U)^{(n)}=V^{(n)}=1\).
Consequently, we have \(\mathrm{dl}(V)=n\) if \(U^{(n-1)}\not\subseteq\ker(\phi)\)
and \(\mathrm{dl}(V)<n\) if \(U^{(n-1)}\le\ker(\phi)\).
\item
Let \(n:=\mathrm{cl}(U)\), then
\(\gamma_n(U)>\gamma_{n+1}(U)=1\) and
\(\gamma_n(V)=\gamma_n(\phi(U))=\phi(\gamma_n(U))\ge\phi(\gamma_{n+1}(U))=\gamma_{n+1}(\phi(U))=\gamma_{n+1}(V)=1\).
Thus, we have \(\mathrm{cl}(V)=n\) if \(\gamma_n(U)\not\subseteq\ker(\phi)\)
and \(\mathrm{cl}(V)<n\) if \(\gamma_n(U)\le\ker(\phi)\).
\item
Let \(n:=\mathrm{cl}_p(U)\), then
\(P_{n-1}(U)>P_n(U)=1\) and
\(P_{n-1}(V)=P_{n-1}(\phi(U))=\phi(P_{n-1}(U))\ge\phi(P_n(U))=P_n(\phi(U))=P_n(V)=1\).
Therefore, we have \(\mathrm{cl}_p(V)=n\) if \(P_{n-1}(U)\not\subseteq\ker(\phi)\)
and \(\mathrm{cl}_p(V)<n\) if \(P_{n-1}(U)\le\ker(\phi)\).
\end{enumerate}
\end{proof}



\subsection{Application to automorphisms}
\label{ss:IndAut}

\begin{corollary}
\label{cor:IndAutQtn}
(Induced automorphism)\\
Let \(\phi:\,G\to H\) be an \textit{epimorphism} of groups, \(\phi(G)=H\),
and assume that \(\sigma\in\mathrm{Aut}(G)\) is an automorphism of \(G\). 

\begin{enumerate}
\item
There exists an induced epimorphism \(\hat{\sigma}:\,H\to H\)
such that \(\hat{\sigma}\circ\phi=\phi\circ\sigma\),
if and only if \(\sigma(\ker(\phi))\le\ker(\phi)\),
resp. \(\ker(\phi)\le\sigma^{-1}(\ker(\phi))\).
\item
The induced epimorphism \(\hat{\sigma}\) is also an automorphism of \(H\),
\(\hat{\sigma}\in\mathrm{Aut}(H)\),
if and only if

\begin{equation}
\label{eqn:KIP}
\sigma(\ker(\phi))=\ker(\phi).
\end{equation}

\end{enumerate}

\end{corollary}


\noindent
In the second statement,
\(\phi\) is said to have the \textit{kernel invariance property} (KIP) with respect to \(\sigma\).

The situation of Corollary
\ref{cor:IndAutQtn}
is visualized in Figure
\ref{fig:InducedAutomorphism}.



{\normalsize

\begin{figure}[ht]
\caption{Induced Automorphism \(\hat{\sigma}\)}
\label{fig:InducedAutomorphism}


\setlength{\unitlength}{1cm}
\begin{picture}(6,5)(-3,-6.5)

\put(-2,-2){\makebox(0,0)[cc]{\(G\)}}
\put(-2,-2.5){\vector(0,-1){3}}
\put(-2.1,-4){\makebox(0,0)[rc]{\(\sigma\)}}
\put(-2,-6){\makebox(0,0)[cc]{\(G\)}}

\put(0,-1.9){\makebox(0,0)[cb]{\(\phi\)}}
\put(-1.5,-2){\vector(1,0){3}}
\put(-1.5,-6){\vector(1,0){3}}
\put(0,-6.1){\makebox(0,0)[ct]{\(\phi\)}}

\put(1.8,-2){\makebox(0,0)[lc]{\(H=\phi(G)\)}}
\put(2,-2.5){\vector(0,-1){3}}
\put(2.1,-4){\makebox(0,0)[lc]{\(\hat{\sigma}\)}}
\put(1.8,-6){\makebox(0,0)[lc]{\(H=\phi(G)\)}}

\end{picture}

\end{figure}

}



\begin{proof}
Since \(\phi\) is supposed to be an epimorphism,
the well-known isomorphism theorem in Remark
\ref{rmk:FctThrQtn}
yields a representation of the image
\(H=\phi(G)\simeq G/\ker(\phi)\) as a quotient.
\begin{enumerate}
\item
According to Theorem
\ref{thm:IndHomQtn},
the automorphism \(\sigma\in\mathrm{Aut}(G)\),
simply viewed as a homomorphism \(\sigma:\,G\to G\),
induces a homomorphism \(\hat{\sigma}:\,G/\ker(\phi)\to G/\ker(\phi)\)
if and only if \(\sigma(\ker(\phi))\le\ker(\phi)\).
Since \(\sigma\) is an epimorphism, \(\hat{\sigma}\) is also an epimorphism
with kernel \(\ker(\hat{\sigma})=\sigma^{-1}(\ker(\phi))/\ker(\phi)\).
\item
Finally,
\(\ker(\hat{\sigma})=1\)
\(\Longleftrightarrow\) \(\sigma^{-1}(\ker(\phi))=\ker(\phi)\)
\(\Longleftrightarrow\) \(\ker(\phi)=\sigma(\ker(\phi))\).
\end{enumerate}
\end{proof}


\begin{remark}
\label{rmk:IndAutQtn}
If \(\ker(\phi)\) is a characteristic subgroup of \(G\),
then Corollary
\ref{cor:IndAutQtn}
makes sure that any automorphism \(\sigma\in\mathrm{Aut}(G)\)
induces an automorphism \(\hat{\sigma}\in\mathrm{Aut}(H)\),
where \(H=\phi(G)\simeq G/\ker(\phi)\).
The reason is that, by definition,
a characteristic subgroup of \(G\) is invariant under any automorphism of \(G\).
\end{remark}



We conclude this section with a statement about GI-automorphisms
(\textit{generator-inverting} automorphisms)
which have been introduced by Boston, Bush and Hajir
\cite[Dfn.2.1]{BBH}.
The proof requires results of Theorem
\ref{thm:IndHomQtn},
Corollary
\ref{cor:IndAutQtn},
and Corollary
\ref{cor:CharSbg}.


\begin{theorem}
\label{thm:GIAut}
(Induced generator-inverting automorphism)\\
Let \(\phi:\,G\to H\) be an epimorphism of groups with \(\phi(G)=H\),
and assume that \(\sigma\in\mathrm{Aut}(G)\) is an automorphism
satisfying the \(\mathrm{KIP}\) \(\sigma(\ker(\phi))=\ker(\phi)\),
and thus induces an automorphism \(\hat{\sigma}\in\mathrm{Aut}(H)\).

If \(\sigma\) is generator-inverting,
that is,

\begin{equation}
\label{eqn:GI}
\sigma(x)G^\prime=x^{-1}G^\prime \text{ for all } x\in G,
\end{equation}

then \(\hat{\sigma}\) is also generator-inverting,
that is, \(\hat{\sigma}(y)H^\prime=y^{-1}H^\prime\) for all \(y\in H\).
\end{theorem}


\begin{proof}
According to Corollary
\ref{cor:IndAutQtn},\\
\(\sigma:\,G\to G\) induces an automorphism \(\hat{\sigma}:\,H\to H\),
since \(\sigma(\ker(\phi))=\ker(\phi)\).\\
Two applications of the Remark
\ref{rmk:IndAutQtn}
after Corollary
\ref{cor:IndAutQtn},
yield:\\
\(\sigma:\,G\to G\) induces an automorphism \(\tilde{\sigma}:\,G/G^\prime\to G/G^\prime\),
since \(G^\prime\) is characteristic in \(G\), and\\
\(\hat{\sigma}:\,H\to H\) induces an automorphism \(\bar{\sigma}:\,H/H^\prime\to H/H^\prime\),
since \(H^\prime\) is characteristic in \(H\).\\
Using Theorem
\ref{thm:IndHomQtn}
and the first part of the proof of Corollary
\ref{cor:CharSbg},
we obtain:\\
\(\phi:\,G\to H\) induces an epimorphism \(\tilde{\phi}:\,G/G^\prime\to H/H^\prime\),
since \(\phi(G^\prime)=\phi(G)^\prime=H^\prime\).\\
The actions of the various induced homomorphisms are given by\\
\(\hat{\sigma}(\phi(x))=\phi(\sigma(x))\) for \(x\in G\),\\
\(\tilde{\sigma}(xG^\prime)=\sigma(x)G^\prime\) for \(x\in G\),\\
\(\bar{\sigma}(yH^\prime)=\hat{\sigma}(y)H^\prime\) for \(y\in H\), and\\
\(\tilde{\phi}(xG^\prime)=\phi(x)H^\prime\) for \(x\in G\).\\
Finally, combining all these formulas and expressing \(H\ni y=\phi(x)\) for a suitable \(x\in G\),
we see that \(\tilde{\sigma}(xG^\prime)=\sigma(x)G^\prime=x^{-1}G^\prime\)
implies the required relation for \(\bar{\sigma}(yH^\prime)=\hat{\sigma}(y)H^\prime\):\\
\(\hat{\sigma}(y)H^\prime
=\hat{\sigma}(\phi(x))H^\prime
=\phi(\sigma(x))H^\prime
=\tilde{\phi}(\sigma(x)G^\prime)
=\tilde{\phi}(x^{-1}G^\prime)
=\phi(x^{-1})H^\prime
=\phi(x)^{-1}H^\prime
=y^{-1}H^\prime\)
\end{proof}



{\normalsize

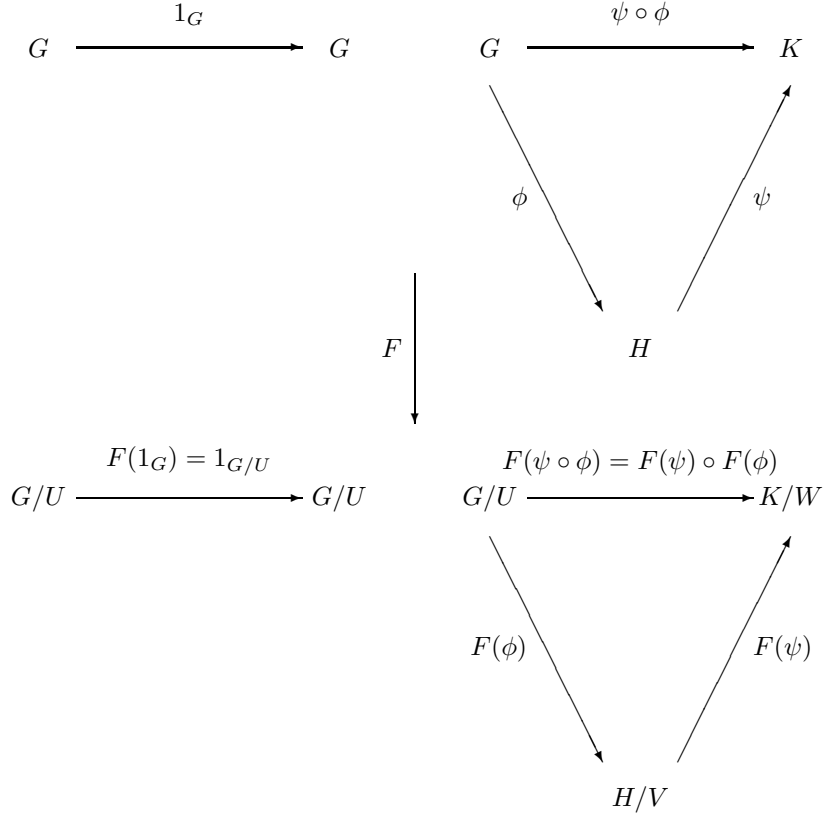
\begin{figure}[ht]
\caption{Functorial properties of induced homomorphisms}
\label{fig:FunctorialProps}


\setlength{\unitlength}{1cm}
\begin{picture}(13,11)(-9.5,-12)


\put(-8,-2){\makebox(0,0)[cc]{\(G\)}}
\put(-6,-1.7){\makebox(0,0)[cb]{\(1_G\)}}
\put(-7.5,-2){\vector(1,0){3}}
\put(-4,-2){\makebox(0,0)[cc]{\(G\)}}

\put(-8,-8){\makebox(0,0)[cc]{\(G/U\)}}
\put(-6,-7.7){\makebox(0,0)[cb]{\(F(1_G)=1_{G/U}\)}}
\put(-7.5,-8){\vector(1,0){3}}
\put(-4,-8){\makebox(0,0)[cc]{\(G/U\)}}


\put(-3.3,-6){\makebox(0,0)[cc]{\(F\)}}
\put(-3,-5){\vector(0,-1){2}}


\put(-2,-2){\makebox(0,0)[cc]{\(G\)}}
\put(-2,-2.5){\vector(1,-2){1.5}}
\put(-1.5,-4){\makebox(0,0)[rc]{\(\phi\)}}

\put(0,-1.7){\makebox(0,0)[cb]{\(\psi\circ\phi\)}}
\put(-1.5,-2){\vector(1,0){3}}

\put(2,-2){\makebox(0,0)[cc]{\(K\)}}
\put(0.5,-5.5){\vector(1,2){1.5}}
\put(1.5,-4){\makebox(0,0)[lc]{\(\psi\)}}

\put(0,-6){\makebox(0,0)[cc]{\(H\)}}

\put(-2,-8){\makebox(0,0)[cc]{\(G/U\)}}
\put(-2,-8.5){\vector(1,-2){1.5}}
\put(-1.5,-10){\makebox(0,0)[rc]{\(F(\phi)\)}}

\put(0,-7.7){\makebox(0,0)[cb]{\(F(\psi\circ\phi)=F(\psi)\circ F(\phi)\)}}
\put(-1.5,-8){\vector(1,0){3}}

\put(2,-8){\makebox(0,0)[cc]{\(K/W\)}}
\put(0.5,-11.5){\vector(1,2){1.5}}
\put(1.5,-10){\makebox(0,0)[lc]{\(F(\psi)\)}}

\put(0,-12){\makebox(0,0)[cc]{\(H/V\)}}

\end{picture}

\end{figure}

}



\subsection{Functorial properties}
\label{ss:FunctorialProps}

The mapping \(F:\,\phi\mapsto F(\phi)=\tilde{\phi}\)
which maps a homomorphism of one category to an induced homomorphism of another category
can be viewed as a \textit{functor} \(F\).

In the special case of induced homomorphisms \(\tilde{\phi}\) between quotient groups,
we define the domain of the functor \(F\) as the following \textit{category} \(\mathcal{G}_n\).\\
The objects of the category are pairs \((G,U)\)
consisting of a group \(G\) and a \textit{normal subgroup} \(U\unlhd G\),

\begin{equation}
\label{eqn:ObjGNormal}
\mathrm{Obj}(\mathcal{G}_n)=\lbrace (G,U)\mid G\in\mathrm{Obj}(\mathcal{G}),\ U\unlhd G\rbrace.
\end{equation}

\noindent
For two objects \((G,U),(H,V)\in\mathrm{Obj}(\mathcal{G}_n)\),
the set of morphisms \(\mathrm{Mor}_{\mathcal{G}_n}((G,U),(H,V))\)
consists of homomorphisms \(\phi:\,G\to H\) such that \(\phi(U)\le V\),
briefly written as arrows \(\phi:\,(G,U)\to (H,V)\),

\begin{equation}
\label{eqn:MorGNormal}
\mathrm{Mor}_{\mathcal{G}_n}((G,U),(H,V))=\lbrace\phi\in\mathrm{Mor}_{\mathcal{G}}(G,H)\mid\phi(U)\le V\rbrace.
\end{equation}

\noindent
The functor \(F:\,\mathcal{G}_n\to\mathcal{G}\)
from this new category \(\mathcal{G}_n\) to the usual category \(\mathcal{G}\) of groups\\
maps a pair \((G,U)\in\mathrm{Obj}(\mathcal{G}_n)\) to the corresponding quotient group \(F((G,U)):=G/U\in\mathrm{Obj}(\mathcal{G})\),\\
and it maps a morphism \(\phi\in\mathrm{Mor}_{\mathcal{G}_n}((G,U),(H,V))\)
to the induced homomorphism \(F(\phi):=\tilde{\phi}\in\mathrm{Mor}_{\mathcal{G}}(G/U,H/V)\),

\begin{equation}
\label{eqn:FunctorGNormal}
F:\,\mathcal{G}_n\to\mathcal{G},\ F((G,U)):=G/U,\ F(\phi):=\tilde{\phi}.
\end{equation}

\noindent
Existence and uniqueness of \(F(\phi):=\tilde{\phi}\) have been proved in Theorem
\ref{thm:IndHomQtn}
under the assumption that \(\phi(U)\le V\),
which is satisfied according to the definition of the arrow \(\phi:\,(G,U)\to (H,V)\).


The \textit{functorial properties}, which are visualized in Figure
\ref{fig:FunctorialProps},
can be expressed in the following form.\\
Firstly, \(F\) maps the identity morphism \(1_G\in\mathrm{Mor}_{\mathcal{G}_n}((G,U),(G,U))\)
having the trivial property \(1_G(U)=U\)
to the identity homomorphism

\begin{equation}
\label{eqn:FunctorialIdentity}
F(1_G)=\tilde{1}_G=1_{G/U}\in\mathrm{Mor}_{\mathcal{G}}(G/U,G/U),
\end{equation}

\noindent
and secondly, \(F\) maps the compositum \(\psi\circ\phi\in\mathrm{Mor}_{\mathcal{G}_n}((G,U),(K,W))\)
of two morphisms \(\phi\in\mathrm{Mor}_{\mathcal{G}_n}((G,U),(H,V))\)
and \(\psi\in\mathrm{Mor}_{\mathcal{G}_n}((H,V),(K,W))\),
which obviously enjoys the required property
\[(\psi\circ\phi)(U)=\psi(\phi(U))\le\psi(V)\le W,\]
to the compositum

\begin{equation}
\label{eqn:FunctorialCompositum}
F(\psi\circ\phi)=(\psi\circ\phi)\,\tilde{}=\tilde{\psi}\circ\tilde{\phi}=F(\psi)\circ F(\phi)
\in\mathrm{Mor}_{\mathcal{G}}(G/U,K/W)
\end{equation}

\noindent
of the induced homomorphisms \textit{in the same order}.\\
The last fact shows that \(F\) is a \textit{covariant} functor.



\subsection{Acknowledgements}
\label{ss:Acknowledgements}

The author would like to express his heartfelt gratitude
to Professor Mike F. Newman
from the Australian National University in Canberra, Australian Capital Territory,
for his continuing encouragement and interest in our endeavour
to strengthen the bridge between group theory and class field theory
which was initiated by the ideas of Emil Artin,
and for his untiring willingness to share his extensive knowledge and expertise
and to be a source of advice in difficult situations.

We also gratefully acknowledge that our research is supported by
the Austrian Science Fund (FWF): P 26008-N25.




\end{document}